\documentclass[11pt,twoside]{article}

\usepackage[utf8]{inputenc}
\usepackage{epsfig,amsfonts,color}
\usepackage{mathrsfs} 
\usepackage{amsmath, bbm, mathabx, amsthm}
\usepackage[normalem]{ulem}
\usepackage{amssymb, palatino, geometry,url}
\usepackage[colorlinks=true,linkcolor=blue,citecolor=blue,urlcolor=blue]{hyperref}
\usepackage{subcaption}
\usepackage{multirow}

\usepackage{algorithm}
\usepackage{algpseudocode}
\usepackage{array}
\usepackage{mathtools}

\newtheorem{proposition}{Proposition}
\usepackage{threeparttable} 

\newcounter{subsubsubsection}[subsubsection]

\usepackage{booktabs}

\usepackage{fancyhdr}
\usepackage{lineno}
\usepackage{float}
\usepackage[section]{placeins}
\usepackage[table]{xcolor}

\title{Event Constrained Programming}
\author{Daniel Ovalle$^1$, Stefan Mazzadi$^2$, Carl D. Laird$^1$, Ignacio E. Grossmann$^1$, \\ and Joshua L. Pulsipher$^{2}$\thanks{Corresponding Author: pulsipher@uwaterloo.edu}\\
    {\small $^1$Department of Chemical Engineering}\\
    {\small \;Carnegie Mellon University, Doherty Hall, 5000 Forbes Ave, Pittsburgh, PA 15213, USA}\\
    {\small $^2$Department of Chemical Engineering}\\
    {\small \;University of Waterloo, 200 University Ave W, Waterloo, ON N2L 3G1, Canada}
    }
\date{}

\begin{document}

\maketitle

\begin{abstract}
In this paper, we present event constraints as a new modeling paradigm that generalizes joint chance constraints from stochastic optimization to (1) enforce a constraint on the probability of satisfying a set of constraints aggregated via application-specific logic (constituting an event) and (2) to be applied to general infinite-dimensional optimization (InfiniteOpt) problems (i.e., time, space, and/or uncertainty domains). This new constraint class offers significant modeling flexibility in posing InfiniteOpt constraints that are enforced over a certain portion of their domain (e.g., to a certain probability level), but can be challenging to reformulate/solve due to difficulties in representing arbitrary logical conditions and specifying a probabilistic measure on a collection of constraints. To address these challenges, we derive a generalized disjunctive programming (GDP) representation of event constrained optimization problems, which readily enables us to pose logical event conditions in a standard form and allows us to draw from a suite of GDP solution strategies that leverage the special structure of this problem class. We also extend several approximation techniques from the chance constraint literature to provide a means to reformulate certain event constraints without the use of binary variables. We illustrate these findings with case studies in stochastic optimal power flow, dynamic disease control, and optimal 2D diffusion.
\end{abstract}


\section{Introduction}\label{sec:intro}
Infinite-dimensional optimization (InfiniteOpt) problems entail decision variables and constraints that are defined over continuous domains (e.g., time, space, and/or uncertainty); in other words, decision variables in these problems are functions/manifolds. Stochastic, dynamic, and partial differential equation (PDE) constrained optimization are prevalent decision-making paradigms that all can be classified as InfiniteOpt problems. These problems often employ complex modeling objects such as differential equations, integrals, and risk measures which make them challenging to formulate and solve. Illustrative engineering applications include the design and operation of process systems \cite{yuan2012state, nikolopoulou2012optimal, chen2017recent}, stochastic optimal power flow \cite{capitanescu2016critical, schmidli2016stochastic}, model predictive control \cite{rawlings2017model, paulson2020stochastic}, model identification for dynamic systems (e.g., microbial communities) \cite{shin2019scalable, miller1983sensitivity}, autonomous vehicle routing \cite{zafar2018methodology, roberge2012comparison}, and structural design \cite{wang2018structural}. 

Stochastic optimization (SO) is a particular InfiniteOpt modeling approach for decision-making under uncertainty where random phenomena are characterized via random parameters defined by a probability density function (pdf) \cite{birge2011introduction}. These random parameters index recourse variables (i.e., second-stage variables) making them infinite-dimensional if the pdf is continuous (e.g., Gaussian). Constraints that subsequently incorporate recourse variables and/or random parameters can be enforced \emph{almost surely}, meaning that they are held for every possible realization of the random parameters. This condition can readily become overly burdensome or impossible to enforce in practice. For instance, it may impractical to design a power grid that operates within certain limits for every possible extreme weather event (e.g., a category 5 hurricane). 

\emph{Chance constraints} overcome this limitation by enforcing that a constraint is held to a certain prescribed probability level (i.e., it is enforced over a certain portion of the possible random scenarios) \cite{charnes1959chance}. Similarly, joint chance constraints enforce such a condition on a set of constraints \cite{miller1965chance}. This provides a powerful modeling object for SO that has been applied to a wide variety of problem classes in the literature which include optimal power flow \cite{baker2017efficient}, model predictive control \cite{paulson2020stochastic}, scheduling \cite{liu2020cvar}, process design/intensification \cite{wendt2002nonlinear}, flexibility/reliability analysis \cite{pulsipher2019scalable, pulsipher2020measuring}, and portfolio planning \cite{pagnoncelli2009sample}. However, one key limitation of classical joint chance constraints is their implicit use of simple intersection ({\footnotesize AND}) logic to aggregate the set of constraints. This exacts that all the constraints must be enforced jointly for a particular random scenario. Such a condition may be overly restrictive in a variety of problems where application-specific logic can be incorporated to enforce a less strict condition on the constraints. For instance, it might be sufficient to satisfy only a certain subset of customer demands in a distribution system.

In other InfiniteOpt problem classes (e.g., dynamic and PDE-constrained), there are applications where enforcing a constraint strictly over the indexing domain (e.g., time and/or space) can be overly burdensome. For instance, in optimal control problems, allowing a path constraint to be violated over a small portion of the time horizon may lead to more favorable (and more often feasible) optimal control policies \cite{zhang1994stability}. Such a relaxation is often achieved in the literature using so-called soft constraints. These are typically implemented either with an exact penalty function which enforces a constraint via Lagrangian relaxation (i.e., penalizing constraint violation in the objective with an associated penalty weight) or by introducing slack variables which are penalized in the objective with a penalty weight \cite{kerrigan2000soft}. While it is straightforward to formulate these approaches, in practice, it is often difficult to choose penalty weights such that the constraints are relaxed to a desired extent (e.g., a particular fraction of the time horizon).

In \cite{pulsipher2022unifying}, Pulsipher and colleagues propose a unifying modeling abstraction for InfiniteOpt problems which provides a rigorous characterization across historically distinct modeling paradigms and is implemented in the Julia package \texttt{InfiniteOpt.jl}. This unified perspective has led to several new modeling approaches such as time-valued analogs of risk measures from SO \cite{pulsipher2022new}, the incorporation of random field theory into InfiniteOpt problems to capture uncertainty over space-time \cite{pulsipher2022random}, and a continuous-time approach for parameter estimation in dynamic systems \cite{pulsipher2022unifying}. To address the aforementioned shortcomings of chance constraints and soft constraints, the authors of \cite{pulsipher2022unifying} also use this abstraction to introduce the notion of \emph{event constraints} which generalize chance constraints from SO to:
\begin{itemize}
    \item use arbitrary logic in aggregating a set of constraints (encoding an \emph{event} we wish to constrain to a certain fractional threshold) and
    \item be applied to general InfiniteOpt problem domains (e.g., enabling time-valued analogs of chance constraints).
\end{itemize}
Here, event constraints provide the modeling flexibility to use logical operators (e.g., {\footnotesize AND} and {\footnotesize OR}) in accordance with application-specific logic to aggregate a set of constraints. Classical joint chance constraints correspond to the most restrictive special case of exclusively using {\footnotesize AND} operators, encoding the intersection of satisfying all constraints for a particular random scenario. Moreover, event constraints enable us to directly specify the fraction of the domain on which we wish to enforce a constraint in dynamic and PDE-constrained optimization problems. 

Event constraints are complex modeling objects that can be challenging to formulate and solve. In the special case of joint chance constraints, a variety of reformulation/solution techniques have been proposed in the literature. Due to the difficulty in determining the joint probability density function needed for exact analytical reformulations, big-M constraint representations in conjunction with sample average approximation (SAA) are often used to reformulate chance constraints via an indicator function \cite{pagnoncelli2009sample}. This SAA approach is simple to implement, but can become intractable for certain complex systems with a large number of samples. To alleviate this limitation, several iterative cutting-plane solution strategies have been proposed such as branch-and-cut decomposition \cite{luedtke2014branch} and combinatorial Benders' cuts \cite{codato2006combinatorial}. Moreover, a few extensions of classical disjunctive programming have been made to solve SAA representations of joint chance constraints \cite{vielma2012mixed}. Alternative solution techniques include data-driven kernel smoothing, which seeks to estimate the density function of joint chance constraints \cite{calfa2015data}, and differentiable SAA, which produces a representation that approximates the quantile function \cite{pena2020solving}. However, connections to generalized disjunctive programming (discussed further below) have not yet been explored to the best of our knowledge.

Moreover, many reformulation methods have been proposed in the literature for individual chance constraints. Exact reformulations are possible under a limited number of forms such as chance constraints without any infinite (i.e., recourse) variables and a linear dependence on a Gaussian random parameter \cite{charnes1963deterministic}. Like joint chance constraints, SAA-based approaches that use binary variables to model an indicator function representation of the chance constraint are popular, but these binary variables can make certain problems very difficult to solve (e.g., nonconvex problems) \cite{cao2020sigmoidal}. Hence, several reformulations have been proposed which use a convex function to overestimate the indicator function and thus provide a conservative approximation of the chance constraint \cite{nemirovski2007convex}. These convex conservative approximations include conditional-value-at-risk (CVaR) approximation (equivalent to Markov bound), Bernstein approximation, Chernoff bound approximation, and Chebyshev bound approximation \cite{nemirovski2012safe, pinter1989deterministic}. Here, CVaR approximation provides the least conservative approximation; however, all of these approaches are typically quite conservative in practice \cite{nemirovski2007convex}. The difference of convex functions approximation presented in \cite{shan2014smoothing} improves the tightness of the CVaR approximation, but its use of difference of max functions make it incompatible with most nonlinear programming (NLP) routines \cite{cao2020sigmoidal}. To alleviate this, the authors in \cite{geletu2015tractable} propose smooth sigmoidal approximation which provides a much tighter fit and is amenable for nonconvex problems. In \cite{cao2020sigmoidal}, the authors propose an improved sigmoidal approximation that provides a systematic approach for selecting the sigmoidal hyper-parameters using CVaR approximation. 

Mathematical programming with complementarity constraints (MPCC) is a framework  used within the NLP community as a way to represent nonsmooth (e.g., binary) decisions \cite{biegler2010nonlinear}. The main idea is to relax the integrity of the binary variables in the original problem and solve a formulation with complementarity constraints, which enforce the relaxed variables to converge to 0-1 values.
This approach could potentially be used to approximate the big-M formulation presented in \cite{pagnoncelli2009sample}, using relaxed variables to represent the indicator function.
Although using complementarity constraints alleviates the computational burden stemming from mixed-integer optimization, they introduce nonlinearity and nonconvexity to the problem \cite{wang2023mpcc}. 
Furthermore, this method often involves solving multiple NLP problems in a sequential manner, as the monolithic formulation might be challenging to solve directly as its solution does not satisfy constraint qualifications \cite{wang2023m}.

General event constraints have been solved using SAA and big-M constraints to approximate the indicator function stochastic optimal power flow, where the event logic is encoded via manually derived auxiliary constraints with additional binary variables \cite{pulsipher2022unifying}. However, as the complexity in the event logic increases, deriving valid auxiliary constraints that encode this logic is non-trivial and error prone, if done manually. Moreover, this simple big-M approach is prone to the same computational limitations observed with joint chance constraints. 

\begin{figure}[!htb]
	\includegraphics[width=0.7\textwidth]{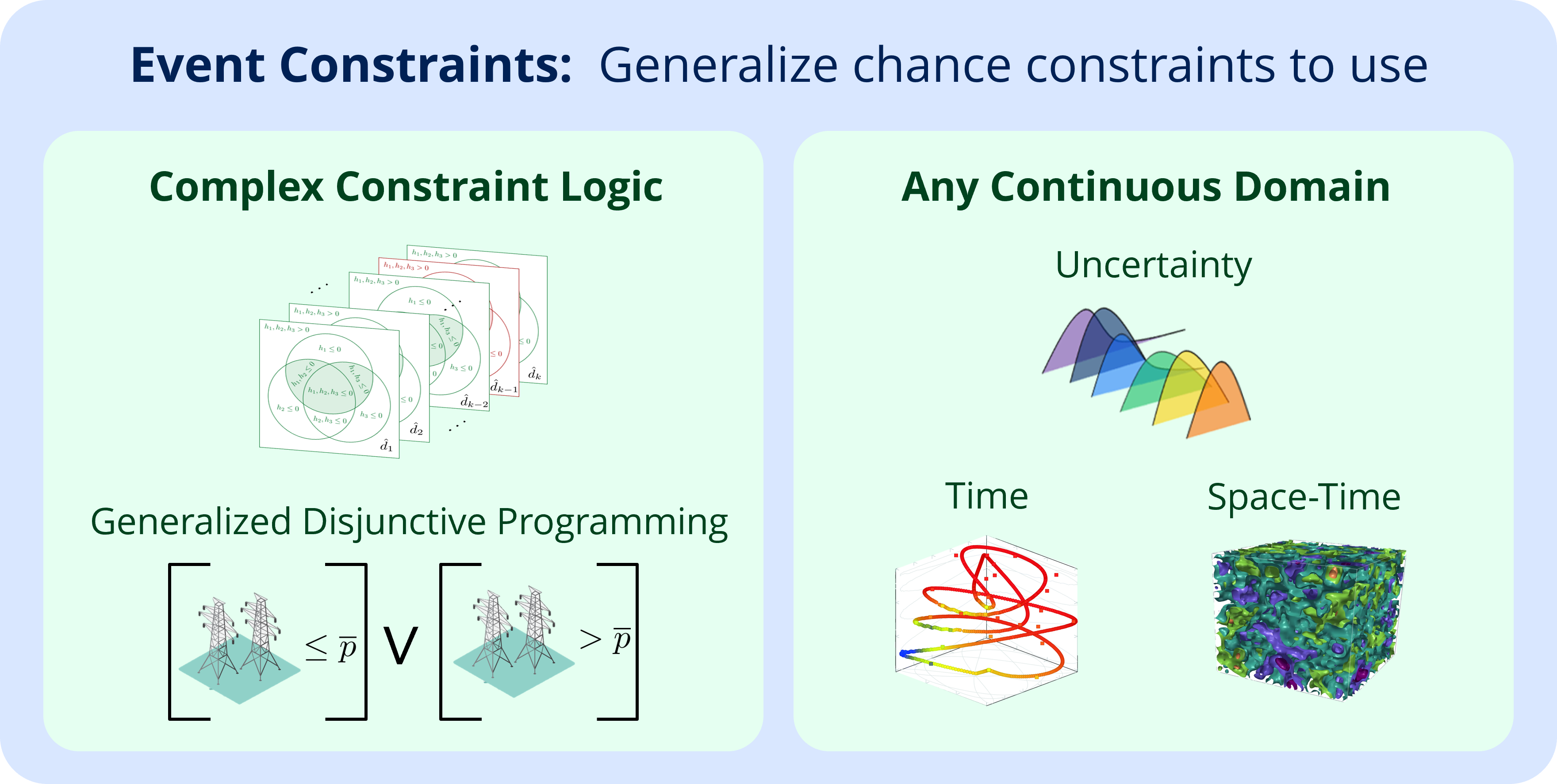}
	\centering
	\caption{A visual summary of how event constrained programming builds upon the theoretical foundation provided by chance constraints and how GDP can be used to model complex logic.}
	\label{fig:abstract}
\end{figure}

In this work, we rigorously formalize event constrained InfiniteOpt problems and demonstrate how they provide useful alternatives to traditional chance constraints and soft constraints as summarized in Figure \ref{fig:abstract}. Moreover, we propose a collection of reformulation techniques that address the shortcomings mentioned above. This includes a generalized disjunctive programming (GDP) representation of event constraints. Generalized disjunctive programming adds several extensions to the disjunctive programming first proposed in \cite{balas1985disjunctive}, and it enables a systematic approach to model systems with complex logic via disjunction constraints and logical propositions \cite{grossmann2021advanced}. Moreover, a variety of reformulation/solution techniques have been developed to effectively solve GDP problems by leveraging their special structure \cite{turkay1996logic, lee2000new, trespalacios2015improved, kronqvist2022p}. There are also open-source software tools such as \texttt{DisjunctiveProgramming.jl} and \texttt{Pyomo.GDP} that make this problem class easy to implement \cite{perez2023julia, chen2021pyomo}. All these characteristics make GDP an attractive modeling abstraction for implementing event constraints and provides a straightforward paradigm for modeling arbitrary event logic using the propositional logic GDP provides. Additionally, we also extend several continuous approximation techniques from the chance constraint literature such as CVaR and sigmoidal approximations to tackle event constraints for general InfiniteOpt problems. Finally, we investigate the use of MPCC formulations to obtain tractable continuous formulations. The contributions of this work include:
\begin{itemize}
    \item the formal mathematical definition of event constraints,
    \item the expression of event constraints via GDP,
    \item the application of GDP solution strategies to event/joint chance constrained problems,
    \item the use of constraint programming operators in characterizing events,
    \item the extension of chance constraint approximation techniques for use on event constraints,
    \item the use of MPCC to solve big-M constraints arising from event constrained problems,
    \item and the demonstration of the above contributions on diverse InfiniteOpt case studies.
\end{itemize}

The remainder of this paper is structured as follows. Section \ref{sec:background} establishes necessary notation and background for chance constraints, InfiniteOpt problems, and GDP. Section \ref{sec:event_constrs} formalizes the definition and treatment of event constraints. Section \ref{sec:formulations} presents and discusses the proposed reformulation/solution techniques. Section \ref{sec:cases} illustrates our findings with case studies pertaining to stochastic, dynamic, and PDE-constrained optimization. Finally, Section \ref{sec:conclusion} highlights key findings and outlines worthwhile future research directions.

\section{Basic Notation and Background}\label{sec:background}
In this section, we establish basic notation for traditional chance constraints and discuss some of reformulation approaches in the literature that are relevant for this work. We also review the unifying abstraction for InfiniteOpt problems presented in \cite{pulsipher2022unifying}. Moreover, for the unfamiliar reader, we provide an overview of generalized disjunctive programming and mathematical programming with complementarity constraints. Since this is not intended to be a complete review, we invite the interested reader to refer to \cite{birge2011introduction, pulsipher2022unifying, grossmann2021advanced, biegler2010nonlinear} for more thorough discussion.

\subsection{Chance Constraints}\label{sec:chance_constrs}
In SO, we typically consider an uncertain parameter $\xi \in \mathcal{D}_\xi \subseteq \mathbb{R}^{n_\xi}$ that is described by a distribution (e.g., $\xi \sim \mathcal{N}(\mu, \Sigma)$). This indexes second-stage (i.e., recourse) variables $q(\xi) \in \mathcal{Q}$ which are considered in conjunction with first-stage (e.g., design) variables $z \in \mathcal{Z} \subseteq{R}^{n_z}$. Here, $\mathcal{Q}$ encapsulates a set of feasible function choices for $q : \mathcal{D}_\xi \mapsto \mathbb{R}^{n_q}$. The chance constrained formulation considers minimizing the measured cost $R_\xi(f(z, q(\xi), \xi))$ (where $R_\xi$ is a risk measure such as $\mathbb{E}_\xi$) relative to constraints $g(z, q(\xi), \xi) \leq 0$ which are enforced almost surely and a (joint) chance constraint that acts on constraints $h(z, q(\xi), \xi) \leq 0$:
\begin{equation}
    \begin{aligned}
        &&\min_{z \in \mathcal{Z}, q(\xi) \in \mathcal{Q}} &&& R_\xi\big(f(z, q(\xi), \xi)\big) \\
        && \text{s.t.} &&& g(z, q(\xi), \xi) \leq 0, \ \ \xi \in \mathcal{D}_\xi \\
        &&&&& \mathbb{P}_\xi\big(h(z, q(\xi), \xi) \leq 0\big) \geq \alpha
    \end{aligned}
    \label{eq:chance_constr_form}
\end{equation}
where $\alpha \in (0, 1]$ is the desired probability  \cite{birge2011introduction}. Moreover, we observe that the chance constraint in \eqref{eq:chance_constr_form} implicitly uses the {\footnotesize AND} operator to aggregate the constraints $h(z, q(\xi), \xi) \leq 0$:
\begin{equation}
    \mathbb{P}_\xi\left(\underset{i \in \mathcal{I}}{\bigwedge}h_i(z, q(\xi), \xi) \leq 0\right) \geq \alpha
    \label{eq:joint_chance_constr}
\end{equation}
although this is typically omitted in the literature. Moreover, \eqref{eq:joint_chance_constr} is called an individual chance constraint when $|\mathcal{I}| = 1$. We also observe that \eqref{eq:joint_chance_constr} enforces the constraints $h(\xi) \leq 0$ almost surely when we set $\alpha = 1$:
\begin{equation}
    \mathbb{P}_\xi\left(h(\xi) \leq 0\right) \geq 1 \iff h(\xi) \leq 0, d \in \mathcal{D}_\xi
\end{equation}
where we let $h(\xi) := h(z, q(\xi), \xi)$ for compactness in notation. Hence, the constraints $g(\xi) \leq 0$ can be interpreted as a special case of a joint chance constraint. 

\subsubsection{Sample Average Approximation with big-M Constraints}
To go about reformulating \eqref{eq:joint_chance_constr} into a form that is compatible with conventional optimization solvers, we first can equivalently represent it using an indicator function $\mathbbm{1}_{h(\xi) \leq 0} : (\mathcal{Z} \times \mathcal{Q} \times \mathcal{D}_\xi) \mapsto \{0, 1\}$:
\begin{equation}
    \mathbb{E}_\xi\left[\mathbbm{1}_{h(\xi) \leq 0}(\xi)\right] \geq \alpha
    \label{eq:indicator_chance_constr}
\end{equation}
where again we let $\mathbbm{1}_{h(\xi) \leq 0}(\xi) := \mathbbm{1}_{h(z, q(\xi), \xi) \leq 0}(z, q(\xi), \xi)$ for convenience in notation. The indicator function is non-differentiable and often is reformulated via the use of a binary variable $y : \mathcal{D}_\xi \mapsto \{0, 1\}$ in combination with big-M constraints:
\begin{equation}
    \begin{aligned}
        &h_i(z, q(\xi), \xi) \leq (1 - y(\xi)) M_i, && i \in \mathcal{I}, \xi \in \mathcal{D}_\xi \\
        &\mathbb{E}_\xi[y(\xi)] \geq \alpha
    \end{aligned}
    \label{eq:binary_chance_constr}
\end{equation}
where $M_i \in \mathbb{R}_+$ is a sufficiently large upper bounding constant. This is often transformed into a finite-dimensional optimization problem via sample average approximation with Monte Carlo (MC) samples $\{\hat{\xi}_k : k \in \mathcal{K\}}$ \cite{pagnoncelli2009sample}:
\begin{equation}
    \begin{aligned}
        &h_i(z, q_k, \hat{\xi}_k) \leq (1 - y_k) M_i, && i \in \mathcal{I}, k \in \mathcal{K} \\
        &\frac{1}{|\mathcal{K}|}\sum_{k \in \mathcal{K}}y_k \geq \alpha.
    \end{aligned}
    \label{eq:SAA_binary_chance_constr}
\end{equation}
This approach is straightforward to implement and preserves the form of the underlying constraints (e.g., if $h(\xi)$ is linear then the reformulation is also linear). However, the introduction of binary variables can diminish the scalability of this solution approach, especially for nonlinear formulations. Hence, a variety of alternative strategies have been developed in literature to address these shortcomings as outlined in Section \ref{sec:intro}. We discuss some of these methods in the subsections below.

\subsubsection{Individual Chance Constraint Reformulations} \label{sec:individual_chance}
We can directly reformulate an individual chance constraint analytically if $h(\xi)$ has no dependence on $q(\xi)$ and has an algebraic form and a distribution that allows us to compute the inverse cumulative density function \cite{charnes1963deterministic}. For instance, if $h(\xi) = \xi^Tz - b$ and $\xi \in \mathbb{R}^{n_z} \sim \mathcal{N}(\mu, \Sigma)$ then we know that $h(\xi) \sim \mathcal{N}(\mu^Tz - b, z^T\Sigma z)$ which gives:
\begin{equation}
    \mathbb{P}_\xi(h(\xi) \leq 0) = \Phi\left(\frac{b - \mu^Tz}{\sqrt{z^T \Sigma z}}\right)
    \label{eq:linear_cdf}
\end{equation}
where $\Phi(\cdot)$ is the cumulative distribution function (cdf) of a standard Gaussian distribution. With \eqref{eq:linear_cdf}, we can compute an analytic form of a chance constraint by inverting the cdf:
\begin{equation}
    \mathbb{P}_\xi(h(\xi) \leq 0) \geq \alpha \iff b - \mu^Tz \geq \Phi^{-1}(\alpha)\left\|\Sigma^{\frac{1}{2}}z\right\|_2
\end{equation}
which can be reformulated as a second-order cone if $\alpha \geq 0.5$. Hence, under a limited set of circumstances we can obtain an analytic formulation, but linearity and convexity is typically not preserved.

Other alternative reformulation techniques for individual chance constraints include those that utilize a conservative continuous approximation of the indicator function used in \eqref{eq:indicator_chance_constr}. Here, convex approximations are particularly popular since they maintain the convexity of $h(\xi) \leq 0$ (assuming it is convex) \cite{nemirovski2007convex}. For a given nonnegative convex nondecreasing function $\phi : \mathbb{R} \mapsto \mathbb{R}_+$ that satisfies $\phi(0) = 1$ we obtain the bound:
\begin{equation}
    \phi\left(\lambda^{-1}\tau\right) \geq \mathbbm{1}_{\tau > 0}(\tau)
    \label{eq:conservative_indicator}
\end{equation}
for $\lambda > 0$ and $\tau \in \mathbb{R}$. Hence, it follows that:
\begin{equation}
    \mathbb{E}_\xi\left[\phi\left(\lambda^{-1}h(\xi)\right)\right] \geq \mathbb{P}_\xi(h(\xi)>0)
\end{equation}
which leads to the conservative approximation of \eqref{eq:indicator_chance_constr}:
\begin{equation}
    \mathbb{E}_\xi\left[\phi\left(\lambda^{-1}h(\xi)\right)\right] \leq 1 - \alpha
    \label{eq:conservative_chance_constr}
\end{equation}
where the relative tightness of the approximation can be adjusted by choice of $\phi(\cdot)$ and $\lambda$. In the literature, \eqref{eq:conservative_chance_constr} is typically rewritten as 
\begin{equation}
    \underset{\lambda > 0}{\inf}\left\{\lambda \mathbb{E}_\xi\left[\phi\left(\lambda^{-1}h(\xi)\right)\right] - \lambda (1-\alpha)\right\} \leq 0
    \label{eq:conservative_chance_constr2}
\end{equation}
which is convex in $\lambda$, $z$, and $q(\xi)$ if $h(\xi)$ is a convex function for a given $\xi$ \cite{nemirovski2007convex}. If we choose $\phi(\tau) = [1 + \tau]_+$ we obtain:
\begin{equation}
    \underset{\lambda > 0}{\inf}\left\{\mathbb{E}_\xi\left[[h(\xi) + \lambda]_+\right] - \lambda (1-\alpha)\right\} \leq 0
\end{equation}
where $[\tau]_+ := \max(0, \tau)$. By rearranging and setting $\lambda = - \lambda$ we obtain a constraint on the conditional-value-at-risk (CVaR):
\begin{equation}
    \text{CVaR}_\xi(h(\xi); \alpha) \leq 0
    \label{eq:cvar_chance_constr}
\end{equation}
which is the least conservative convex approximation of an individual chance constraint \cite{nemirovski2012safe}. Moreover, the SAA representation of \eqref{eq:cvar_chance_constr} does not require the use of binary variables which gives it a clear computational advantage over \eqref{eq:binary_chance_constr}. However, the CVaR approximation is often quite conservative in practice. Other popular choices for $\phi(\cdot)$ include the Bernstein approximation ($\phi(\tau) = \exp^\tau$) and the Chebyshev approximation ($\phi(\tau) = [\tau + 1]^2_+$).

Finally, we review the SigVaR approximation for individual chance constraints proposed in \cite{cao2020sigmoidal}. Here, we use a modified sigmoidal function that outer approximates the indicator function:
\begin{equation}
    \phi_\text{sig}(\tau) := \left[2\frac{1 + \beta}{\beta + \exp(-\gamma \tau)}-1\right]_+
\end{equation}
where $\beta,\gamma \in \mathbb{R}_+$ are hyper-parameters. With this, we can approximate \eqref{eq:indicator_chance_constr}:
\begin{equation}
    \mathbb{E}_\xi\left[\phi_\text{sig}(h(\xi))\right] \leq 1 - \alpha
    \label{eq:sig_chance_constr}
\end{equation}
which is a conservative approximation for any $\beta, \gamma \in \mathbb{R}_+$ that becomes exact as $\beta \rightarrow \infty$ if we choose $\gamma(\beta) := (1 + \beta)\theta$ with $\theta > 0$. The authors of \cite{cao2020sigmoidal} define the sigmoidal value-at-risk (SigVaR):
\begin{equation}
    \text{SigVaR}_\xi(T; \alpha, \beta, \gamma) := \underset{\lambda \in \mathbb{R}}{\inf}\left\{\mathbb{E}_\xi\left[\phi_\text{sig}(T -\lambda)\right] \leq 1 - \alpha\right\}
    \label{eq:sigvar}
\end{equation}
which can be used to equivalently reformulate \eqref{eq:sig_chance_constr} as:
\begin{equation}
    \text{SigVaR}_\xi(h(\xi); \alpha, \beta, \gamma) \leq 0.
\end{equation}
Using this representation, the authors are able to show that a feasible choice of the hyper-parameters $\beta, \gamma$ can be made based on the optimal value $\lambda_\text{CVaR}^*$ stemming from the CVaR approximated solution of Problem \eqref{eq:chance_constr_form}. In particular, we select $\beta \geq \bar{\beta}$ where $\bar{\beta}$ is the positive solution to $\bar{\beta} - \log(2 + \bar{\beta}) = 1$ and then set $\gamma = \frac{-\beta-1}{2\lambda_\text{CVaR}^*}$ (i.e., $\theta = -(2\lambda_\text{CVaR}^*)^{-1}$) to provide a SigVaR approximation that is at least as accurate as the Bernstein approximation \cite{cao2020sigmoidal}. Moreover, we approach the exact solution of Problem \eqref{eq:chance_constr_form} by sequentially solving the SigVaR approximated version of Problem \eqref{eq:chance_constr_form} with increasing values of $\beta$ while updating $\gamma = \frac{-\beta-1}{2\lambda_\text{CVaR}^*}$. In practice, this sequential approach often leads to solutions that are significantly less conservative then the those corresponding to the CVaR approximation \cite{cao2020sigmoidal}.

\subsubsection{Converting Joint Chance Constraints to Individual Chance Constraints} \label{sec:joint2individ}
We can apply the above individual chance constraint approximations outlined in Section \ref{sec:individual_chance} to joint chance constraints that are first converted into individual chance constraints. One common approach is to conservatively approximate the joint chance constraint in \eqref{eq:joint_chance_constr} as a collection of $|\mathcal{I}|$ individual chance constraints:
\begin{equation}
    \mathbb{P}_\xi(h_i(\xi) > 0) \leq 1 - \alpha_i
    \label{eq:bonferroni_approximation}
\end{equation}
where $\sum_{i \in \mathcal{I}} (1 - \alpha_i) \leq (1 - \alpha)$ and $\alpha_i \in (0, 1]$ which follows from the Bonferroni inequality. Typically, $\alpha_i = (1-\alpha)/|\mathcal{I}|$ is selected since treating $\alpha_i$ as decision variables typically leads to intractable formulations; moreover, \eqref{eq:bonferroni_approximation} tends to be quite conservative in practice even with optimal choices of $\alpha_i$ \cite{nemirovski2007convex}. 

Another popular approach is to enforce an individual chance constraint on the pointwise maximum of the constraint functions $h(\xi)$:
\begin{equation}
    \mathbb{P}_\xi\left(\max_{i \in \mathcal{I}}h_i(\xi) \leq 0\right) \geq \alpha
    \label{eq:max_joint_chance}
\end{equation}
which unlike \eqref{eq:bonferroni_approximation} is equivalent to \eqref{eq:joint_chance_constr}. For certain approximation techniques, the additional maximum function can lead to intractable formulations; however, for most SAA driven approaches (e.g., SAA approximations of CVaR and SigVaR) this leads to tractable formulations \cite{geng2019data}. 

\subsection{Unifying Abstraction for InfiniteOpt Problems}\label{sec:abstraction}
In \cite{pulsipher2022unifying}, the authors present a unifying modeling abstraction for InfiniteOpt problems that is based on a collection of general modeling objects. These are implemented along with a library of transformation approaches in the Julia package \texttt{InfiniteOpt.jl} which extends \texttt{JuMP.jl} (a popular algebraic modeling language for mathematical programming) to model problems in their natural infinite-dimensional form. This modeling approach has enabled the authors to transfer theory across disciplines and discover several new modeling approaches which are described in \cite{pulsipher2022unifying}, \cite{pulsipher2022new}, and \cite{pulsipher2022random}. The core modeling objects to this unifying abstraction are summarized in Figure \ref{fig:abstraction} and outlined in the discussion below.

\begin{figure}[!htb]
	\includegraphics[width=\textwidth]{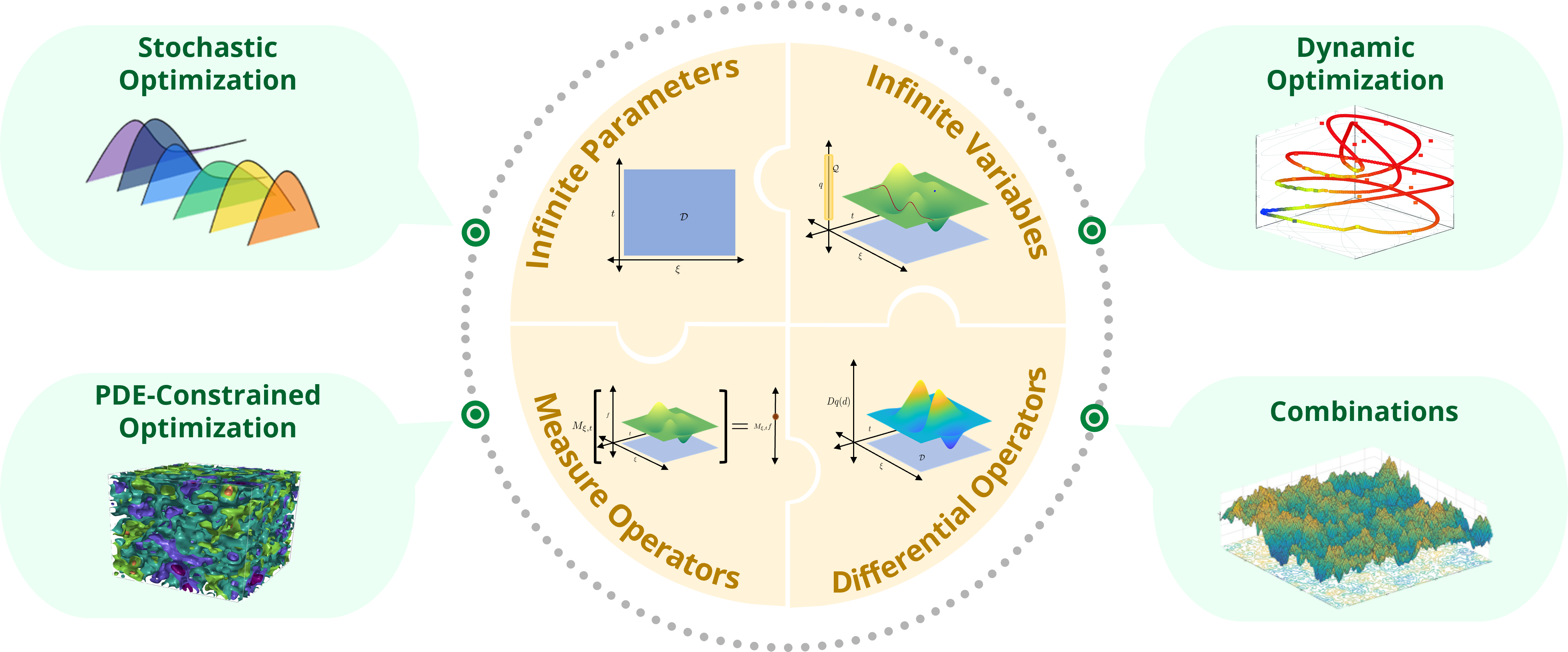}
	\centering
	\caption{A depiction of how using infinite parameter, infinite variable, measure operator, and differential operator modeling objects allows us to capture formulations in stochastic, dynamic, and PDE-constrained optimization and combinations.}
	\label{fig:abstraction}
\end{figure}

First, infinite parameters $d \in \mathcal{D} \subseteq \mathbb{R}^{n_d}$ index the decision domain $\mathcal{D}$ (also called the infinite domain) of the InfiniteOpt problem. These act as the independent variables to the decision variables and are what parameterize the feasible region. The infinite parameters can be defined over a wide variety of problem domains. For instance, we can incorporate time $t \in [t_0, t_f] \subset \mathbb{R}$, 3D space $x \in \mathcal{S} \subset \mathbb{R}^3$, an uncertain parameter $\xi \in \mathcal{D}_\xi \subseteq \mathbb{R}^{n_\xi}$, and combinations of these.

Decision variables include infinite variables $q : \mathcal{D} \mapsto \mathcal{Q} \subseteq \mathbb{R}^{n_q}$ which are functions that encode the choices we can make over the domain $\mathcal{D}$. These can be alternatively interpreted as a collection of $n_q$ finite variables indexed by an infinite parameter (giving an infinite collection of variables). Specific examples include uncertainty-dependent recourse $q(\xi)$ in stochastic optimization and time-dependent state/control trajectories $q(t)$ in dynamic optimization. We can also incorporate finite variables $z \in \mathbb{Z} \subseteq \mathbb{R}^{n_z}$ which capture decisions that are independent of the domain $\mathcal{D}$ (e.g., design variables and 1st stage variables). 

Next, differential operators $D : \mathscr{Q} \mapsto \mathscr{D}$ are applied to infinite variables to capture their rate of change with respect to $d$. Here, $\mathscr{Q}$ is a function space for $q(d)$ and $\mathscr{D}$ is the function space of the output differential functions. Example operators include temporal partial derivatives $\partial / \partial t$ and spatial Laplacian operators $\nabla_x$.

Finally, measure operators $M_d : \mathscr{Q} \mapsto \mathbb{R}$ summarize infinite variables by a scalar value. A more general output function space $\mathscr{M}$ is considered in \cite{pulsipher2022infinite} to account for summarizing over a subset of infinite parameters, but to reduce complexity, we restrict this to $\mathbb{R}$ which is sufficient for this work. Examples include time-valued integrals and risk measures from SO (e.g., CVaR and expectation).

With the above modeling components, we formulate a general InfiniteOpt problem:
\begin{equation}
    \begin{aligned}
        &&\min_{z \in \mathcal{Z}, q(d) \in \mathcal{Q}} &&& M_d f(z, Dq, q(d), d) \\
        && \text{s.t.} &&& g(z, Dq, q(d), d) \leq 0, \ \ d \in \mathcal{D} \\
        &&&&& M_d h(z, Dq, q(d), d) \geq 0
    \end{aligned}
    \label{eq:infiniteopt_form}
\end{equation}
which features two classes of constraints: algebraic constraints $g(d) \leq 0$ and measured constraints $M_d h(d) \geq 0$. This form is quite general and captures a large collection of problems in stochastic, dynamic, and PDE-constrained optimization. For instance, this readily captures Problem \eqref{eq:chance_constr_form} if we let $d = \xi \in \mathcal{D}_\xi$, remove $Dq$, let $M_d = R_\xi$ in the objective, and define 
\begin{equation}
    M_\xi h(\xi) := \mathbb{P}_\xi\left(\bigwedge_{i \in \mathcal{I}}h_i(\xi) \leq 0\right) - \alpha. 
    \label{eq:chance_measure}
\end{equation}
In Section \ref{sec:event_constrs}, we will see how this connection readily inspires the generalization of chance constraints for general domains $\mathcal{D}$. 

\subsection{Generalized Disjunctive Programming}\label{sec:gdp}

Generalized disjunctive programming (GDP) is a mathematical optimization framework that provides a systematic way to model and solve problems with logical constraints.
It was first introduced by \cite{raman1994modelling} as an extension of the disjunctive programming paradigm proposed by \cite{balas1985disjunctive}, which modeled a feasible region as a disjunction of convex sets. The general GDP formulation that facilitates the modeling of logical and disjunctive constraints is as follows \cite{grossmann2013systematic}:
\begin{equation}\label{eqn:gdp.general}
\begin{aligned}
&&\min &&& f(x) \\
&&\text{s.t.}  &&& g(x) \leq 0\\
&&&&&\ \bigvee_{i\in I_k} \left[
    \begin{gathered}
    Y_{ik} \\
    r_{ik}(x)\leq 0\\
    \end{gathered}
\right], && k \in K\\
&&&&& {\text{\footnotesize EXACTLY}}(1, Y_{k}),  && k \in K \\
&&&&& \Omega(Y) = \text{True} \\
&&&&& x \in X\\
&&&&& Y_{ik} \in \{\text{True}, \text{False}\}, && k \in K, i \in I_k\\
\end{aligned} 
\end{equation}
where $x \in \mathbb{R}^n$ are continuous variables, $f(\cdot)$ is the objective, $g(\cdot)$ are global constraint functions, $Y_{ik}$ is a Boolean variable that indicates whether to enforce the constraints $r_{ik}(x)\leq0$ which are contained in its corresponding disjunct, and $Y_k$ denotes the set of logical variables $\{Y_{ik}: i \in I_k\}$. Here, the disjunctions are indexed over $k$ and are each comprised of $|I_k|$ disjuncts. Every disjunct is related through an OR ($\vee$) and an exactly one operator ($\text{\footnotesize EXACTLY}(1,\cdot)$) is used to require that only one disjunct (i.e., collection of constraints) is enforced per disjunction \cite{perez2023modeling}. $\Omega(Y)$ is the set of logical prepositions composed of clauses relating the Boolean variables through the first order operators {\footnotesize AND} ($\wedge$), {\footnotesize OR} ($\vee$), {\footnotesize XOR} ($\underline{{\vee}}$), {\footnotesize NEGATION} ($\neg$), {\footnotesize IMPLICATION} ($\Rightarrow$) and {\footnotesize EQUIVALENCE} ($\Leftrightarrow$). Moreover, $\Omega$ can be composed by cardinality clauses defining rules over $m$ elements such as $\text{\footnotesize{EXACTLY}}(m,\cdot)$, $\text{\footnotesize{ATLEAST}}(m,\cdot)$, and $\text{\footnotesize{ATMOST}}(m,\cdot)$ \cite{yan1999tight}.

GDP provides a framework for directly and intuitively modeling logical relationships in optimization problems and to avoid manually posing logical considerations through mixed-integer constraints that are prone to modeling errors \cite{grossmann2013systematic,grossmann2002review}. Moreover, it enables the modeler to rapidly switch between solution methods. These methods include using transformations (e.g., big-M \cite{raman1994modelling}, hull \cite{lee2000new}, hybrid approaches \cite{sawaya2005cutting}) that convert GDP formulations into mixed-integer formulations that can be solved with traditional mixed-integer solvers. There are solution algorithms, such as logic-based outer approximation (LOA) \cite{turkay1996logic},  logic-based branch-and-bound (LBB) \cite{grossmann2002review}, and logic-based discrete-steepest descent (LD-SDA) \cite{bernal2022process, ovalle2024logic}, that directly operate on GDP formulations, leveraging their unique structure.

An important feature of GDP, that we exploit in this work, is its ability to systematically translate logical constraints $\Omega(Y) = \text{True}$ into algebraic constraints that use binary variables $y := \mathbbm{1}(Y)$ \cite{raman1991relation, hooker2002logic, grossmann2021advanced}. This is accomplished by systematically applied De Morgan's laws to transform $\Omega$ into its conjunctive normal form, the logical clauses of which can be readily converted into standard MIP constraints.
For more detail and examples on this systematic procedure we refer the readers to \cite{grossmann2021advanced}.


GDP modeling objects (i.e., disjunctions and logical constraints) are readily supported by software tools as \texttt{pyomo.GDP} \cite{chen2021pyomo}, \texttt{DisjunctiveProgramming.jl} \cite{perez2023julia}, and \texttt{GAMS} which automate the transformation and solution techniques discussed above. In Section \ref{sec:gdp-representation}, we will see how GDP provides a natural framework to model event constraints with complex logic. 

\subsection{Mathematical Programming with Complementarity Constraints} \label{sec:mpcc}

Mathematical programming with complementarity constraints (MPCC) offers a strategy to potentially avoid introducing the binary variables traditionally used to solve chance constraints via big-M reformations. The canonical form of a complementarity constraint is:
\begin{equation}
    \begin{aligned}
        & 0 \leq y^0 \perp y^1 \geq 0 \\
        & y^0, y^1 \in \mathbb{R}_+^0
    \end{aligned}
    \label{eq:complementarity_constraints}
\end{equation}
where $\perp$ is the complementarity operator that enforces that at least one of the variables is at the bound \cite{biegler2010nonlinear}. The complementarity operator yields an inclusive OR, meaning both variables could be at the bound simultaneously. Therefore, exclusivity needs to be enforced with an additional summation constraint. By adding the exclusivity constraint together with an upper bound of $1$, a binary variable $y \in \{0,1\}$ can be expressed in terms of continuous variables as:
\begin{equation}
    \begin{gathered}
        0 \leq y^0 \perp y^1 \geq 0\\
        y^0 + y^1 = 1 \\
        0 \leq y^0, y^1 \leq 1
    \end{gathered}
    \label{eq:binary_complementarity_constraints}
\end{equation}
where $y^0$ and $y^1$ indicate if $y$ takes the value of $0$ or $1$, respectively.

Several alternative forms of \eqref{eq:binary_complementarity_constraints} have been proposed to make the constraints better posed for NLP solvers \cite{anitescu2007elastic, hu2004convergence, leyffer2006interior}. In this work, we choose the smooth-max approximation of complementarity constraints, since it has performed well in practice \cite{baumrucker2008mpec, dabadghao2023complementarity}. Nevertheless, other representations for MPCC constraints can be used within the approach we propose in Section \ref{sec:infinte_mpcc}. The  smooth-max approximation represents \eqref{eq:complementarity_constraints} as:
\begin{equation}
    y^1 - \max(0, y^1-y^0)=0
    \label{eq:max_cc_reformulation}
\end{equation}
where the $\max$ operator can be approximated smoothly as:
\begin{equation}
    \frac{1}{2}\left(y^1 + y^0 - \sqrt{(y^1 + y^0)^2 +\delta^2} \right)=0
    \label{eq:max_cc_approximation}
\end{equation}
where $\delta \in \mathbb{R}_+$ is a small numerical tolerance such that the first- and second-order constraint gradients remain continuous.

In practice, MPCC problems are often not directly solved in a single-solve formulation since complementarity constraints are often not numerically well-posed. Therefore, a sequence of formulations with relaxed constraints:
\begin{equation}
    \begin{aligned}
        & -\epsilon \leq \frac{1}{2}\left(y^1 + y^0 - \sqrt{(y^1 + y^0)^2 +\delta^2} \right) \leq \epsilon\\
        &  1 - \epsilon \leq y^0 + y^1 \leq 1 + \epsilon \\
        & 0 \leq y^0, y^1 \leq 1 
    \end{aligned}
    \label{eq:binary_complementarity_constraints_approx}
\end{equation}
are solved with shrinking values of the tolerance $\epsilon \in (0,1]$ which yields the solution as $\epsilon \rightarrow 0$. For a more thorough discussion on MPCC, alternative NLP formulations, and convergence properties, we refer the reader to \cite{biegler2010nonlinear, wang2023mpcc}.

\section{Event Constraints}\label{sec:event_constrs}
In this section, we formalize the definition of event constrained InfiniteOpt problems and show how event constraints stem from chance constraints by generalizing the domain and the constraint aggregation logic.

\subsection{General Formulation}
We define an event constrained InfiniteOpt formulation by replacing the measured constraint in \eqref{eq:infiniteopt_form} with an event constraint:
\begin{equation}
    \begin{aligned}
        &&\min_{z \in \mathcal{Z}, q(d) \in \mathcal{Q}} &&& M_d f(z, Dq, q(d), d) \\
        && \text{s.t.} &&& g(z, Dq, q(d), d) \leq 0, \ \ d \in \mathcal{D} \\
        &&&&& \mathbb{P}_d\Big(\Omega\big(h(z, Dq, q(d), d) \leq 0\big)\Big) \geq \alpha.
    \end{aligned}
    \label{eq:event_form}
\end{equation}
Here, $\mathbb{P}_d$ is a generalization of the stochastic measure operator $\mathbb{P}_\xi$ for general infinite parameters $d$ and $\Omega : \{\text{True}, \text{False}\}^{n_h} \mapsto \{\text{True}, \text{False}\}$ encodes \emph{event logic} to summarize the constraints $h(d) \leq 0$. Problem \eqref{eq:event_form} is quite general and can be applied to a wide breadth of InfiniteOpt problem classes such as two-stage stochastic optimization, optimal control, and PDE-constrained optimization. In similar manner to \eqref{eq:infiniteopt_form}, we observe that \eqref{eq:event_form} produces the stochastic formulation in \eqref{eq:chance_constr_form} as a special case when we choose $d = \xi \in \mathcal{D}_\xi$, remove $Dq$, let $M_d = R_\xi$, and use {\footnotesize AND} logic in $\Omega(\cdot)$:
\begin{equation}
    \Omega\big(h(\xi) \leq 0\big) = \left(\bigwedge_{i \in \mathcal{I}}h_i(\xi) \leq 0\right).
\end{equation}
We discuss the implications of the unique event constraint modeling components in the sections below where we demonstrate how event constraints naturally generalize the scope of traditional chance constraints.

\subsection{Generalizing the Domain} \label{sec:general_domain}
Following the use of a general infinite parameter in Problem \eqref{eq:infiniteopt_form} and the measure operator representation of a chance constraint in Equation \eqref{eq:chance_measure}, we can generalize the measure operator in \eqref{eq:chance_measure} to act on general InfiniteOpt constraint functions $h(z, Dy, q(d), d)$ (denoted $h(d)$ for compactness):
\begin{equation}
    M_d h(d) = \mathbb{P}_d\left(\bigwedge_{i \in \mathcal{I}}h_i(d) \leq 0\right) - \alpha
\end{equation}
which readily leads to an analog of a chance constraint for general InfiniteOpt problems:
\begin{equation}
    \mathbb{P}_d\left(\bigwedge_{i \in \mathcal{I}}h_i(d) \leq 0\right) \geq \alpha
    \label{eq:chance_constr_analog}
\end{equation}
where we explicitly show the {\footnotesize AND} operator $\bigwedge$ that is typically omitted in the joint chance constraint literature. Here, the probabilistic measure operator $\mathbb{P}_d$ is an unconventional modeling object for deterministic formulations (e.g., $d = t$), but we can reformulate it using an integral measure. First, we can rewrite \eqref{eq:chance_constr_analog} with a generalized expectation operator $\mathbb{E}_d$ and an indicator function:
\begin{equation}
    \mathbb{E}_d\left[\mathbbm{1}_{h(d) \leq 0}(d)\right] \geq \alpha
    \label{eq:expect_event}
\end{equation}
where again we write $\mathbbm{1}_{h(d) \leq 0}(d) := \mathbbm{1}_{\bigwedge_{i \in \mathcal{I}} h_i(z, Dy, q(d), d) \leq 0}(z, Dy, q(d), d)$ for compactness. Assuming $d$ is continuous over $\mathcal{D}$, we obtain:
\begin{equation}
    \int_{d \in \mathcal{D}}\mathbbm{1}_{h(d) \leq 0}(d) p(d) dd \geq \alpha
    \label{eq:general_chance_int}
\end{equation}
where $p(d)$ is a weighting function that satisfies $\int_{d \in \mathcal{D}} p(d) dd = 1$. We refer to the logic encoded in the indicator function (i.e., $\bigwedge_{i \in \mathcal{I}}h_i(d) \leq 0$) as an \emph{event} which we wish to enforce to a certain measured extent over the domain $\mathcal{D}$ (e.g., enforce the event that all constraints $h(t) \leq 0$ are respected for a certain fraction $\alpha$ of a time horizon $\mathcal{D}_t$). Hence, we refer to this generalized class of chance constraints as \emph{event constraints} since the term "chance" is not well-suited for use with deterministic InfiniteOpt formulations. 

A key consideration in formulating event constraints is the choice of weighting function $p(d)$. For SO formulations (i.e., $d = \xi$), $p(\xi)$ corresponds to the pdf of the distribution that describes $\xi$. However, for deterministic formulations, $p(d)$ can be interpreted as a weighting function that gives us the modeling flexibility to place arbitrary priority over the domain $\mathcal{D}$ following the analyses presented in \cite{pulsipher2022unifying} and \cite{pulsipher2022new}. For instance, in optimal control with $d = t \in \mathcal{D}_t = [0, t_f]$, it is natural to select the uniform weighting 
\begin{equation}
    p(t) = \frac{1}{t_f}
    \label{eq:uniform_time}
\end{equation}
which makes the integral in \eqref{eq:general_chance_int} equivalent to the geometric mean of $\mathbbm{1}_{h(t) \leq 0}(t)$ \cite{stewart2009calculus}. Alternatively, we could choose a time-valued analog of a truncated exponential pdf:
\begin{equation}
    p(t; \nu) = \frac{\exp\big(-\frac{t}{\nu}\big)}{\nu - \nu \exp\big(-\frac{t_f}{\nu}\big)}
\end{equation}
where $\nu > 0$. This places priority on earlier times, making it analogous to a discount factor \cite{shin2022exponential}. 

\subsection{Generalizing the Logic}
Traditional chance constraints and the event constraint generalization shown in \eqref{eq:chance_constr_analog} are restricted to the logical intersection of a constraint set $h(d) \leq 0$ for every value of $d$ as illustrated in Figure \ref{fig:joint_logic}. This presents a straightforward strategy for posing a chance/event constraint over a collection of constraints and corresponds to joint probability distributions in a SO context. However, intersection logic can be overly conservative for a variety of applications where complex logic (i.e., logical regions beyond the intersection shown Figure \ref{fig:joint_logic}) is used in practice to enforce conditions. For instance, power grid operators will often place different priorities on respecting capacity restrictions over the diverse collection of generators and lines in a particular system. Thus, mandating that all constraints be simultaneously satisfied at a single instance of $d$ may be overly restrictive. 

\begin{figure}[!htb]
	\includegraphics[width=0.3\textwidth]{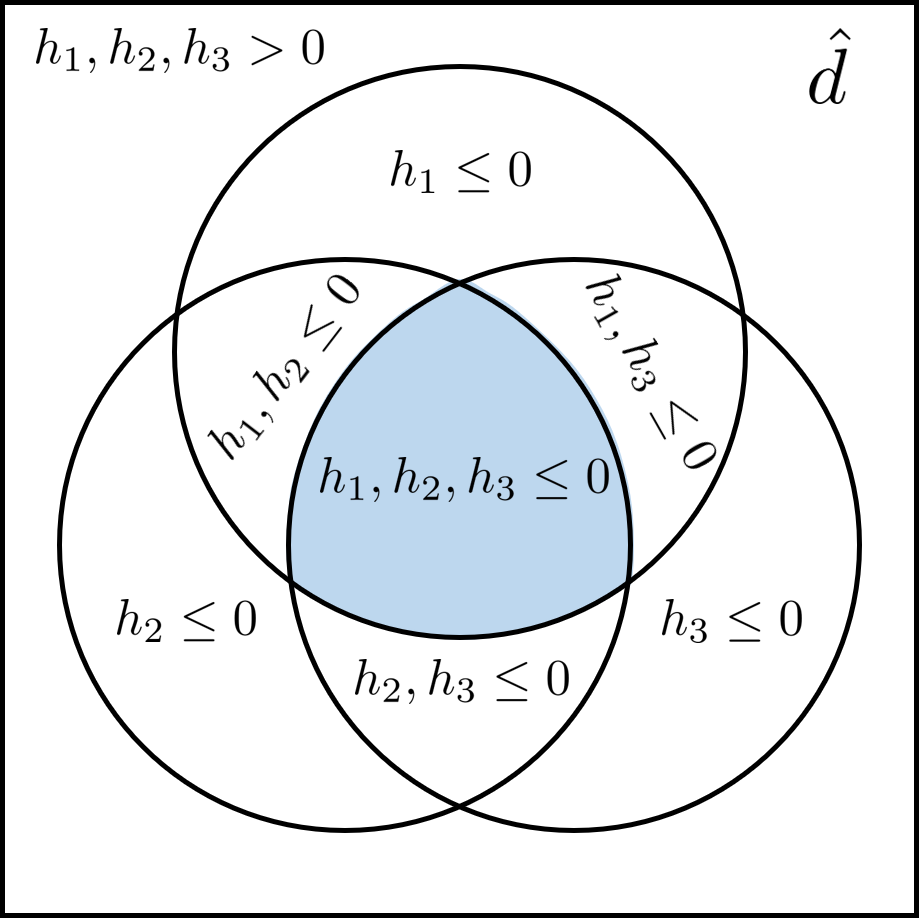}
	\centering
	\caption{An illustration of the classical logical intersection at a particular value $\hat{d}$ of $d$ for a tertiary constraint system in \eqref{eq:chance_constr_analog}. Using logical operators this is expressed $h_1(\hat{d}) \leq 0 \wedge h_2(\hat{d}) \leq 0 \wedge h_3(\hat{d}) \leq 0$.}
	\label{fig:joint_logic}
\end{figure}

The explicit usage of the {\footnotesize AND} operator $\bigwedge$ in \eqref{eq:joint_chance_constr} and \eqref{eq:chance_constr_analog} highlights how logic operators summarize the constraint set to a single logical value $\{\text{True}, \text{False}\}$. This perspective readily inspires the use of more complex combinations of logical operators to aggregate the constraint set $h(d) \leq 0$. Possible choices include the aforemetioned operators {\footnotesize AND}, {\footnotesize OR}, {\footnotesize XOR}, {\footnotesize NEGATION}, {\footnotesize IMPLICATION}, and {\footnotesize EQUIVALENCE}. The use of these operators enables us to capture larger logical regions and avoid the conservativeness associated with using the intersection (see Figure \ref{fig:and-or}). Moreover, we can adapt logic functions from constraint programming such as {\footnotesize ATLEAST}, {\footnotesize ATMOST}, and {\footnotesize EXACTLY}. Figure \ref{fig:atleast} illustrates the logical region captured with {\footnotesize ATLEAST} logic. 

\begin{figure}[!htb]
     \centering
     \begin{subfigure}[b]{0.4\textwidth}
         \centering
         \includegraphics[width=0.75\textwidth]{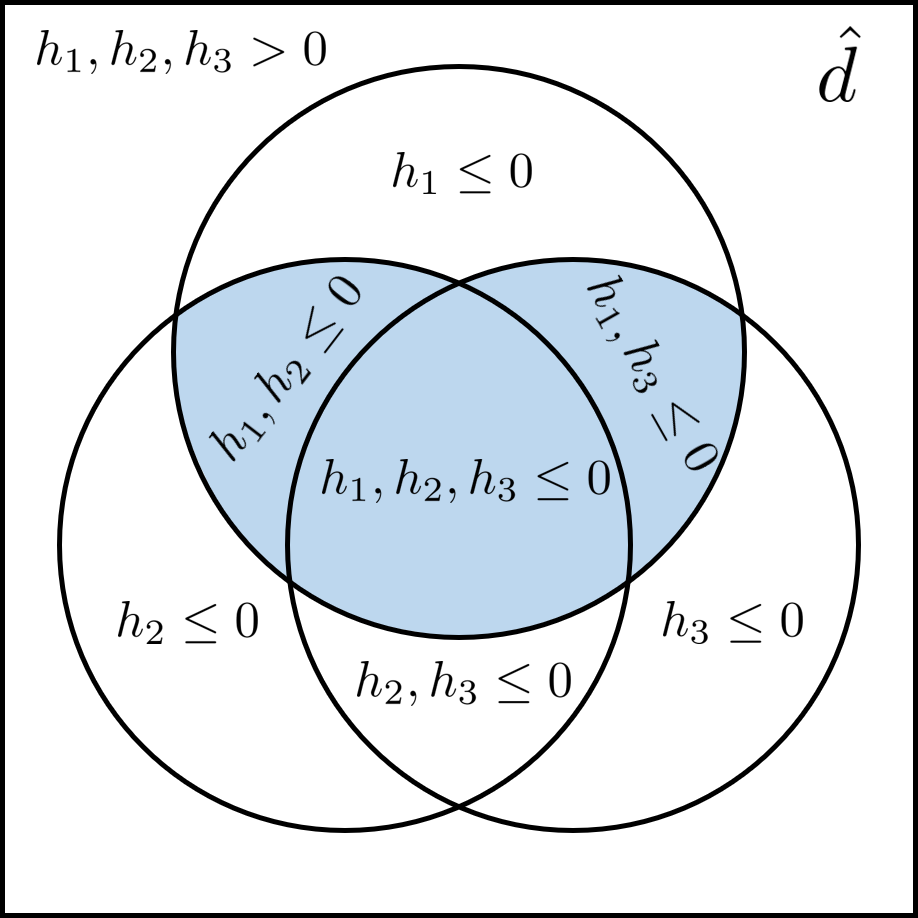}
         \caption{$h_1(\hat{d}) \leq 0 \wedge (h_2(\hat{d}) \leq 0 \vee h_3(\hat{d}) \leq 0)$}
         \label{fig:and-or}
     \end{subfigure}
     \hspace{1cm}
     \begin{subfigure}[b]{0.4\textwidth}
         \centering
         \includegraphics[width=0.75\textwidth]{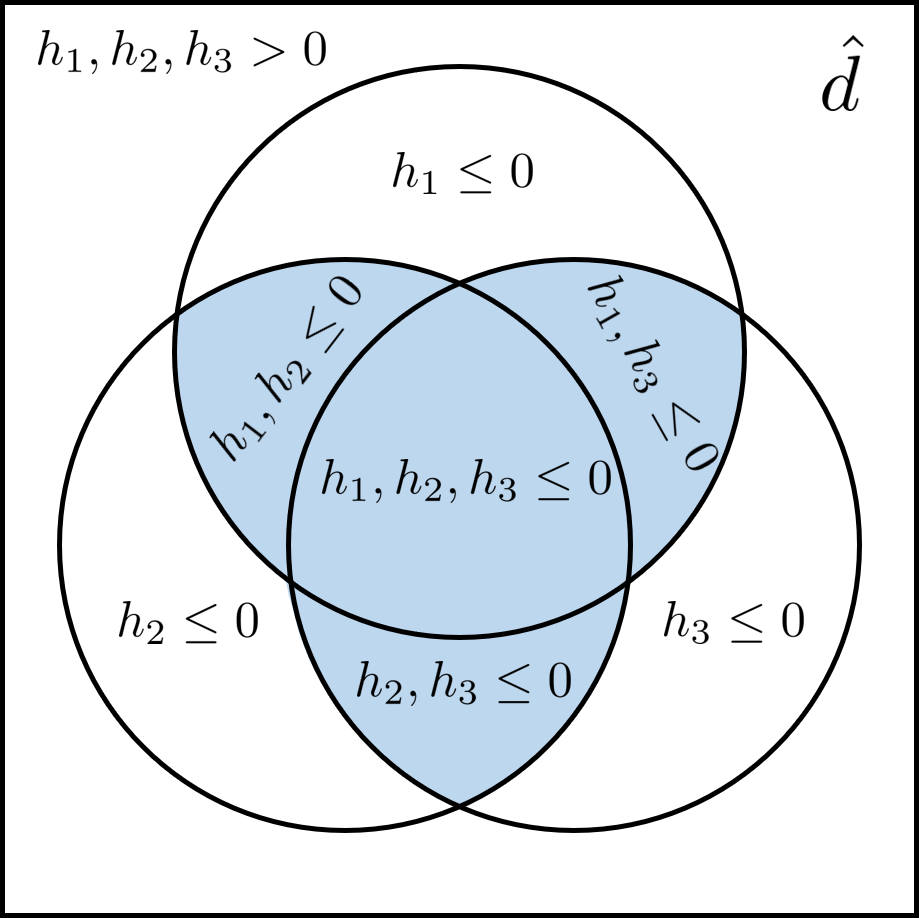}
         \caption{{\footnotesize ATLEAST}$(2, h_1(\hat{d}) \leq 0, h_2(\hat{d}) \leq 0, h_3(\hat{d}) \leq 0)$}
         \label{fig:atleast}
     \end{subfigure}
    \caption{Examples of generalized logic in a tertiary constraint system at a particular value $\hat{d}$ that encompasses more than just the logical intersection.}
    \label{fig:general_logic}
\end{figure}

Hence, we can generalize event constraints to use arbitrary combinations of logical operators (encoded in a summarizing function $\Omega(\cdot)$):
\begin{equation}
    \mathbb{P}_d\Big(\Omega\big(h(d) \leq 0\big)\Big) \geq \alpha
    \label{eq:general_event_constr}
\end{equation}
rather than the intersection considered in \eqref{eq:chance_constr_analog}. Substituting \eqref{eq:general_event_constr} in \eqref{eq:infiniteopt_form} leads to the general event constrained program presented in \eqref{eq:event_form}.

\subsection{Previous Solution Strategies} \label{sec:previous_strategies}
In \cite{pulsipher2022unifying} (where event constraints are first proposed), the authors used a big-M approach (analogous to that shown for chance constraints in Section \ref{sec:chance_constrs}) to reformulate the indicator function in \eqref{eq:expect_event} via binary variables $y_i(d) \in \{0, 1\}, \ i \in \mathcal{I}$ to obtain:
\begin{equation}
    \begin{aligned}
        &h_i(d) \leq (1 - y_i(d)) M_i, \ \ i \in \mathcal{I}, d \in \mathcal{D} \\
        &\mathbb{E}_d[\Omega(y(d))] \geq \alpha
    \end{aligned}
    \label{eq:event_bigm}
\end{equation}
where $M_i \in \mathbb{R}_+$ is a sufficiently large upper bound that allows relaxing the constraint when $y_i(\xi) = 0$. However, no systematic approach was proposed to convert $\Omega(y(d))$ into a system of algebraic constraints. In this work, we will address this methodological gap using GDP in as depicted in Section \ref{sec:gdp}. For now, we highlight that in the special case where $\Omega(\cdot)$ only uses {\footnotesize AND} operators, we can express \eqref{eq:event_bigm} using a single binary variable $y(d) \in \{0, 1\}$:
\begin{equation}
    \begin{aligned}
        &h_i(d) \leq (1 - y(d)) M_i, \ \ i \in \mathcal{I}, d \in \mathcal{D} \\
        &\mathbb{E}_d[y(d)] \geq \alpha.
    \end{aligned}
    \label{eq:infinite_bigm_and}
\end{equation}

Similar formulations in the literature are often discretized using a finite set of points $\hat{\mathcal{D}} := \{\hat{d}_k : k \in \mathcal{K}\}$ to obtain a finite-dimensional optimization problem that is compatible with conventional optimization solvers. The integral behind the expectation is approximated using an appropriate numerical scheme over $\hat{\mathcal{D}}$ (e.g., trapezoid rule). In the case that the points are uniformly spaced for a deterministic problem or Monte Carlo sampled for an SO problem, \eqref{eq:infinite_bigm_and} becomes:
\begin{equation}
    \begin{aligned}
        &h_i(\hat{d}_k) \leq (1 - y_k) M_i, \ \ i \in \mathcal{I}, k \in \mathcal{K} \\
        &\frac{1}{|\mathcal{K}|}\sum_{k \in \mathcal{K}}y_k p(\hat{d}_k)\geq \alpha.
    \end{aligned}
    \label{eq:finite_bigm_and}
\end{equation}

In Section \ref{sec:formulations}, we propose a collection of alternative reformulation/solution approaches to improve the computational tractability of event constrained InfiniteOpt problems. These include generalized analogs of the nonlinear approximation approaches discussed in Section \ref{sec:individual_chance} to avoid the introduction of binary variables $y(d)$ which often limit the tractability of nonlinear event constrained problems.

\section{Event Constraint Solution Techniques}\label{sec:formulations}
In this section, we present new representations and reformulations of event constraints that enable us to systematically model arbitrary constraint aggregation logic and provide a toolbox of methods to better promote solution efficiency. Similar to Section \ref{sec:general_domain}, throughout this section we address InfiniteOpt constraint functions $h(z, Dy, q(d), d)$ as $h(d)$ for compactness. For the same reason, we do not explicitly write the domain of variables $ q(d) \in \mathcal{Q}$ and $ z \in \mathcal{Z}$.

\subsection{GDP Representation}\label{sec:gdp-representation}
We propose and discuss an infinite-dimensional GDP-based formulation that is equivalent to \eqref{eq:event_form}. We provide strategies to solve this formulation using existing GDP modeling tools and draw a parallel between these techniques and some existing methods for solving chance constraints.

To derive the formulation in its natural infinite-dimensional form, we begin by defining Boolean variables $Y_i(d) \in \{\text{True}, \text{False}\}, \ i \in \mathcal{I},$ for each constraint $h_i(d) \leq 0, \ i \in \mathcal{I}$ in \eqref{eq:general_event_constr} that indicate when each constraint is satisfied (i.e., if $Y_i(d) = \text{True}$, then $h_i(\cdot) \leq 0$). With these we can impose a disjunction at each constraint:
\begin{equation}
    \left[\begin{gathered}
        Y_i(d)\\
        h_i(d) \leq 0
    \end{gathered}\right] \vee
    \left[\begin{gathered}
        \neg Y_i(d)\\
        h_i(d) > 0
    \end{gathered}\right], \ \ i \in \mathcal{I}, \ d \in \mathcal{D}.\\
    \label{eq:gdp_disjunctions}
\end{equation}
We can use the Boolean variables $Y_i$ to enforce the event constraint in \eqref{eq:general_event_constr}:
\begin{subequations}
    \begin{align}
        &\Omega(Y(d)) \Leftrightarrow W(d), \ \  d \in \mathcal{D} \label{eq:event_logic}\\
        &\mathbb{E}_d[W(d)] \geq \alpha
    \end{align}
    \label{eq:gdp_event_constraint}
\end{subequations}
where $W(d) \in \{\text{True}, \text{False}\}$ indicates whether an event occurs satisfying the event logic function $\Omega(\cdot)$. Here, the logical constraints encoded in \eqref{eq:event_logic} can be systematically transformed into linear inequalities by first converting them to CNF following the methodology described in Section \ref{sec:gdp}. Substituting \eqref{eq:gdp_disjunctions} and \eqref{eq:gdp_event_constraint} into \eqref{eq:event_form}, we obtain the full GDP formulation for event-constrained InfiniteOpt problems:
\begin{equation}
    \begin{aligned}
        &&\min_{} &&& M_d f(z, Dq, q(d), d) \\
        && \text{s.t.} &&& g_j(z, Dq, q(d), d) \leq 0, && j \in \mathcal{J}, \ d \in \mathcal{D} \\
        &&&&&   \left[\begin{gathered}
                    Y_i(d)\\
                    h_i(d) \leq 0
                \end{gathered}\right] \vee
                \left[\begin{gathered}
                    \neg Y_i(d)\\
                    h_i(d) > 0
                \end{gathered}\right], && i \in \mathcal{I}, \ d \in \mathcal{D} \\ 
        &&&&& \Omega(Y(d)) \iff W(d), &&  d \in \mathcal{D} \\
        &&&&& \mathbb{E}_d[W(d)] \geq \alpha \\
        &&&&& Y_i(d), W(d) \in \{\text{True}, \text{False}\}, && i \in \mathcal{I}, \ d \in \mathcal{D}\\
    \end{aligned}
    \label{eq:infinite_gdp}
\end{equation}
Equation \eqref{eq:gdp_disjunctions} has a special structure where a single Boolean variable is related with its negation in a disjunction that only has two disjuncts. 
This structure guarantees exclusivity in the disjunction, hence the $\text{\footnotesize EXACTLY}(1,\cdot)$ operator shown in \eqref{eqn:gdp.general} is not required.

\begin{proposition}
\label{thm:exact_gdp}
    Formulation \eqref{eq:infinite_gdp} is exact meaning that a Boolean variable $Y_i(d)= \text{True}$ if and only if constraints $h_i(d) \leq 0$ are satisfied for $i \in \mathcal{I}, \ d \in \mathcal{D}$. Consequently, the expectation in \eqref{eq:gdp_event_constraint} captures every realization where $h_i(d) \leq 0$ holds.
\end{proposition} 

\begin{proof}
    We will prove the proposition by showing both directions of the implication.

    \begin{itemize}
        \item \textbf{($\Rightarrow$)} This follows from applying the definition of disjunction to the left disjunct. In GDP formulations, a Boolean variable in a disjunction yields one-way implication logic. 
If the Boolean variable is True, its set is enforced yielding:
\begin{equation}
Y_i(d) \Rightarrow \{ h_i(d) \leq 0 \} \quad  i \in \mathcal{I}, \ d \in \mathcal{D}\\
    \label{eq:right_side}
\end{equation}
however, satisfying a set does not require its indicator Boolean variable to be True. Therefore, thus far there could exist instances of $i \in \mathcal{I}, \ d \in \mathcal{D}$ with $Y_i(d)=\text{False}$ where $ h_i(d) \leq 0$ holds. 

\item \textbf{($\Leftarrow$)} The case mentioned above where $Y_i(d)=\text{False}$ and $ h_i(d) \leq 0$ holds is prevented by the right disjunct in \eqref{eq:gdp_disjunctions} as:
\begin{equation}
\{ Y_i(d)=\text{False} \} \Leftrightarrow \neg Y_i(d) \Rightarrow \{ h_i(d) > 0 \} \quad  i \in \mathcal{I}, \ d \in \mathcal{D}.\\
    \label{eq:right_disjunct}
\end{equation}
Note that the sets considered in \eqref{eq:gdp_disjunctions} correspond to a complementary partition of the domain, implying the entire feasible region of the problem is contained in:
\begin{equation}
\{ h_i(d) \leq 0\} \cup \{ h_i(d) > 0 \} \quad  i \in \mathcal{I}, \ d \in \mathcal{D}.\\
    \label{eq:union}
\end{equation}
Hence, the proposed disjunction is always proper \cite{trespalacios2016cutting} since:
\begin{equation}
\{ h_i(d) \leq 0\} \cap \{ h_i(d) > 0 \} =\varnothing \quad  i \in \mathcal{I}, \ d \in \mathcal{D}.\\
    \label{eq:intersect}
\end{equation}
and the entire domain can be implied by the value of a single Boolean variable $Y_i(d)$ as:
\begin{equation}
Y_i(d)=\begin{cases}
        \text{True} \Rightarrow [ h_i(d) \leq 0 ], \\
        \text{False} \Rightarrow [ h_i(d) > 0 ]
    \end{cases}
    \quad  i \in \mathcal{I}, \ d \in \mathcal{D}.
    \label{eq:all_space}
\end{equation}

Given that a Boolean variable has two values and \eqref{eq:union} covers the entire domain, we find that $h_i(d) \leq 0$ is violated only when $Y_i(d) = \text{False}$, which in the context context of a proper disjunction implies.
\begin{equation}
 \{ h_i(d) \leq 0 \} \Rightarrow Y_i(d) \quad  i \in \mathcal{I}, \ d \in \mathcal{D}\\
    \label{eq:left_side}
\end{equation}
    \end{itemize}
effectively showing the GDP formulation in \eqref{eq:infinite_gdp} is exact.
\end{proof}

Formulation \eqref{eq:infinite_gdp} can be made finite-dimensional via SAA; however, some of the GDP solution methods described in the next subsection can be utilized directly to \eqref{eq:infinite_gdp} before SAA or any other finite transformation technique is applied.

\subsubsection{Solution Techniques}\label{sec:gdp_solution_techniques}
To solve the infinite-dimensional GDP formulation shown in \eqref{eq:infinite_gdp}, all equations must be in terms of non-strict inequalities. Therefore we must add a small positive numerical tolerance $\delta$ to the right disjunct in \eqref{eq:gdp_disjunctions} as:
\begin{equation}
    \left[\begin{gathered}
        Y_i(d)\\
        h_i(d) \leq 0
    \end{gathered}\right] \vee
    \left[\begin{gathered}
        \neg Y_i(d)\\
        - h_i(d) + \delta \leq 0
    \end{gathered}\right], \ \ i \in \mathcal{I}, \ d \in \mathcal{D}\\.
    \label{eq:gdp_disjunctions_delta}
\end{equation}
which states that constraints $h_i(d)$ have to be violated by at least $\delta$ to be considered a violation. Note that the result from Proposition \ref{thm:exact_gdp} now holds asymptotically as $\delta$ tends to zero. By substituting \eqref{eq:gdp_disjunctions} for \eqref{eq:gdp_disjunctions_delta} in \eqref{eq:infinite_gdp}, we obtain an infinite-dimensional GDP formulation for event constraints that fits traditional mathematical programming notation.

Traditional GDP to mixed-integer (non)linear program (MI(N)LP) transformations such as the big-M reformulation (BM) \cite{raman1994modelling} or the hull reformulation (HR) \cite{lee2000new} can be applied to directly to solve the problem. GDP formulations can also be transformed by indicator constraints which are supported by several modern solvers such as Gurobi \cite{belotti2016handling, perez2023julia, achterberg2019s}.
A one-sided big-M transformation, as described in \cite{pulsipher2022unifying} (see \eqref{eq:event_bigm}), mirrors traditional big-M approaches common in chance constraint literature. Strictly speaking, this reformulation, transforms solely the left disjunct in \eqref{eq:gdp_disjunctions_delta}, and differs from the canonical GDP big-M transformation (referred to as BM), which addresses both sides of the disjunction (hence called two-sided reformulation). The BM reformulation of \eqref{eq:infinite_gdp} is given by:

\begin{equation}
    \begin{aligned}
        &&\min_{} &&& M_d f(z, Dq, q(d), d) \\
        && \text{s.t.} &&& g_j(z, Dq, q(d), d) \leq 0, && j \in \mathcal{J}, \ d \in \mathcal{D} \\
        &&&&& h_i(d) \leq  M_i (1 - y(d)),  && i \in \mathcal{I}, \ d \in \mathcal{D} \\
        &&&&& \delta - h_i(d) \leq M_i y(d),  && i \in \mathcal{I}, \ d \in \mathcal{D} \\ 
        &&&&& \mathbb{E}_d[\Omega_{MIP}(y(d))] \geq \alpha \\
        &&&&& y_i(d) \in \{0,1\}, && i \in \mathcal{I}, \ d \in \mathcal{D}\\
    \end{aligned}
    \label{eq:infinite_gdp_BM}
\end{equation}
where $\Omega_{MIP}(\cdot)$ represents the logical propositions reformulated as mixed-integer programming constraints following the systematic procedure described in Section \ref{sec:gdp}. Note the traditional chance constraint big-M reformulation is not exact, implying there might be instances where $y_i(d)=0$ and $h_i(d) \leq 0$ is still satisfied for some $i \in \mathcal{I}, \ d \in \mathcal{D}$ (an infeasible scenario when applying the GDP-derived BM reformulation according to Proposition \ref{thm:exact_gdp}). However, such occurrences are rare in practice as objective functions typically drive towards exact solutions.

In this study, we focus solely on GDP-MINLP transformations, such as BM and HR, as they entail a single reformulation step that has been studied for InfiniteOpt problems. While more advanced logic-based solution techniques like LOA \cite{turkay1996logic} and LBB \cite{grossmann2002review} could theoretically be employed to address \eqref{eq:infinite_gdp}, they lack formalization and study in the context of infinite GDP formulations. Consequently, they are beyond the scope of this work and remain as potential directions for future research.

\subsection{Approximate Continuous Reformulations}\label{sec:approximations}
In Section \ref{sec:individual_chance}, we discussed various approaches to approximate the behavior of a chance constraint without using binary variables. 
In this section, we demonstrate how these techniques, originally developed for stochastic domains, can also be applied to general InfiniteOpt domains. The argument follows the one presented in \cite{pulsipher2020measuring} which states that deterministic InfiniteOpt problems (e.g., dynamic optimization) can be posed as a special case of two-stage stochastic optimization problems. Consider Problem \eqref{eq:chance_constr_form} where the stochastic infinite parameter $\xi :\Psi \mapsto \mathcal{D}_\xi$ is defined by the probability space $(\Psi, \mathscr{F}, \mathcal{D}_\xi)$ which means that the infinite domain $\mathcal{D}_\xi$ corresponds to the co-domain of the distribution. In the special case that the sample space $\Psi$ only contains a single realization (i.e., $|\Psi| = 1$), the distribution collapses and $\xi$ becomes deterministic. Hence, any deterministic infinite parameter $d \in \mathcal{D}$ (e.g., time and/or position) can be substituted into Problem \eqref{eq:chance_constr_form} since the resulting deterministic InfiniteOpt problem can be viewed as a special case of the original chance-constrained formulation. This illustrates how Problem \eqref{eq:chance_constr_form} can be generalized to other domains $d$ which is what led to the general event constrained formulation given in \eqref{eq:event_form}. With this observation, we argue that the solution techniques (with their associated properties) developed in the literature for chance constrained problems can be directly transferred to event constraints. This is consistent with the previous observation that risk measures from SO can be readily transferred to dynamic optimization problems \cite{pulsipher2020measuring}. Note that this section focuses on approximations for event constraints that only involve a single constraint. For event constraints with multiple constraints, the approaches discussed in Section \ref{sec:joint2individ} can be used.

\subsubsection{CVaR Approximation}\label{sec:CVaR}
Following the discussion in Section \ref{sec:individual_chance}, the CVaR approximation provides the tightest conservative convex approximation of individual chance constraints \cite{nemirovski2012safe}. The conservativeness of the approximate solution is a key property since it guarantees that the chance constraint is satisfied at the solution. Moreover, this approximation can be efficiently solved using standard (N)LP techniques, eliminating the need to introduce binary variables that require the use of mixed-integer solvers that are more computationally demanding, especially for nonlinear problems.

Following the observation that chance constraint solution techniques can be applied to other infinite domains, we rewrite \eqref{eq:cvar_chance_constr} as:
\begin{equation}
    \text{CVaR}_d(h(d); \alpha) \leq 0
    \label{eq:cvar_chance_constr_general}
\end{equation}
which when integrated into \eqref{eq:event_form} can be equivalently expressed (imposing $|\mathcal{I}| = 1$):
\begin{equation}
    \begin{aligned}
        &&\min_{} &&& M_d f(z, Dq, q(d), d) \\
        && \text{s.t.} &&& g_j(z, Dq, q(d), d) \leq 0, && j \in \mathcal{J}, \ d \in \mathcal{D} \\
        &&&&& \phi(d) \geq h(d) - \lambda, && d \in \mathcal{D} \\ 
        &&&&& \mathbb{E}_d[\phi(d)] \leq -\lambda (1-\alpha) \\
    \end{aligned}
    \label{eq:infinite_cvar}
\end{equation}
where $\phi$ is a variable that approximates the indicator function shown in \eqref{eq:conservative_indicator} and $\lambda \in \mathbb{R}$ is a finite variable. Formulation \eqref{eq:infinite_cvar} can be solved using standard InfiniteOpt transformation techniques such as direct transcription. As previously stated, CVaR approximations often prove overly conservative in practice, obtaining significantly larger objective values relative to big-M reformulations \cite{nemirovski2007convex}. The conservativeness is clearly demonstrated by the numerical studies presented in Section \ref{sec:cases}.

\subsubsection{SigVaR Approximation}\label{sec:SigVaR}
The SigVaR approximation method discussed in Section \ref{sec:individual_chance} provides a nonconvex approximation for the indicator function based on a modified sigmoidal function \cite{cao2020sigmoidal}. Like CVaR, this approximation is strictly conservative and avoids introducing binary variables such that gradient-based NLP solvers like \textsc{Ipopt} can be used. We can extend \eqref{eq:sig_chance_constr} to a more general domain $\mathcal{D}$ to obtain:
\begin{equation}
    \begin{aligned}
        &&\min_{} &&& M_d f(z, Dq, q(d), d) \\
        && \text{s.t.} &&& g_j(z, Dq, q(d), d) \leq 0, && j \in \mathcal{J}, \ d \in \mathcal{D} \\
        &&&&& \phi(d) \geq 2\frac{1 + \beta}{\beta + \exp(-\gamma h(d))} - 1, && d \in \mathcal{D} \\ 
        &&&&& \mathbb{E}_d[\phi(d)] \leq 1-\alpha \\
    \end{aligned}
    \label{eq:infinite_sigvar}
\end{equation}
where again $\phi$ approximates the indicator function shown in \eqref{eq:conservative_indicator} and $\beta, \gamma \in \mathbb{R}_+$ are the sigmoidal parameters. Selecting these parameters is not straightforward, as it involves balancing approximation quality and numerical stability. The steepness of the first and second derivatives of the approximation increases with $\mathcal{O}(\gamma)$ and $\mathcal{O}(\gamma^2)$, respectively. Therefore, the authors in \cite{cao2020sigmoidal} offer an algorithm to initialize and refine these parameters progressively via sequential solution of the SigVaR formulation which we generalize for Problem \eqref{eq:infinite_sigvar} in Algorithm \ref{alg:sigvar}. 

\begin{algorithm}
\caption{Sequential SigVaR Algorithm \cite{cao2020sigmoidal}}
\begin{algorithmic} 
\State \textbf{Input:} Desired event satisfaction fraction ($\alpha \in (0,1]$), target ($\beta^*\in \mathbb{R}_+$), step size ($\eta > 1$).
\State \textbf{Output:} Optimal value of variables and objective $(z_k, q_k(d), f_k)$.
\vspace{2mm}
\State \textbf{1. Solve CVaR \eqref{eq:infinite_cvar} to obtain $\lambda(\alpha)$ and Initialize} 
\State $\ \ \ \ $ Set $k \leftarrow 0$. 
\State $\ \ \ \ $ Set $ \Gamma \leftarrow -\frac{1}{\lambda(\alpha)}$
 ; $ \beta_k \leftarrow \bar{\beta}$ (where $\bar{\beta}$ is the positive solution to $\bar{\beta} - \log(2+\bar{\beta})=1$) ; $\gamma_k=\Gamma\frac{\beta_k+1}{2}$.
\State $\ \ \ \ $  Update iteration $k \leftarrow k+1$.
\State \textbf{2. Solve SigVaR \eqref{eq:infinite_sigvar} with $\beta_k$ and $\gamma_k$ to obtain $(z_k, q_k(d), f_k)$} \Comment{Initialize with $(z_{k-1}, q_{k-1}(d))$}
    \State $\ \ \ \ $ \textbf{if} {$\beta_k < \beta^*$} \textbf{then}
        \State $\ \ \ \ \ \ \ \ $ Go to \textbf{Step 3.}
    \State $\ \ \ \ $ \textbf{else}
        \State  $\ \ \ \ \ \ \ \ $ Go to \textbf{Step 4.}
    \State $\ \ \ \ $ \textbf{end if}
\State \textbf{3. Update sigmoidal parameters}
\State  $\ \ \ \ $ Set $\beta_{k+1} \leftarrow \eta\beta_k$ ; $\gamma_{k+1} \leftarrow \Gamma\frac{\beta_{k+1}+1}{2}$.
\State  $\ \ \ \ $ Update iteration $k \leftarrow k+1$.
\State  $\ \ \ \ $ Go to \textbf{Step 2.}
\State \textbf{4. Terminate and return $(z_k, q_k(d), f_k)$}
\end{algorithmic}
\label{alg:sigvar}
\end{algorithm}

The algorithm requires defining a target parameter, $\beta^*$, where higher values lead to a better approximation of the indicator function. From our numerical experiments, we found that $\beta^* = 10^5$ typically provided a good starting point. Similarly, specifying the step size $\eta$ entails balancing faster convergence against numerical stability, particularly as the function sharpens with increasing $\beta$. The algorithm initializes sigmoidal parameters using the solution from the CVaR approximation \eqref{eq:infinite_cvar}. Subsequently, it iteratively solves SigVaR approximations \eqref{eq:infinite_sigvar}, starting from the values of the previous iteration to improve numerical convergence, and proceeds to update the sigmoidal parameters. Termination occurs upon reaching or surpassing the desired $\beta^*$ parameter, effectively obtaining the optimal solution. In practice where determining the right $\beta^*$ can be error prone, we found it useful to terminate the algorithm if the solver was unable to find an optimal solution on the latest iteration and use the solution from the previous iteration as the final result. The numerical results in Section \ref{sec:cases} illustrate how SigVaR approximation is able to find high quality solutions similar to standard big-M reformulations, but at a fraction of the computational cost. For more information on the theoretical properties of using the SigVaR approximation, we refer the reader to \cite{cao2020sigmoidal}.

\subsubsection{MPCC Approximation} \label{sec:infinte_mpcc}
MPCC offers an alternative approach to obtain continuous approximations of big-M formulations in event constraint programming. By expressing the big-M equations from \eqref{eq:infinite_gdp_BM} in terms of continuous variables enforced to behave as binary through a complementarity constraint as shown in \eqref{eq:binary_complementarity_constraints}, we derive:
\begin{equation}
    \begin{aligned}
        &&&&& h_i(d) \leq  M_i (1 - y^1(d)), && i \in \mathcal{I}, \ d \in \mathcal{D} \\ 
        &&&&& 0 \leq y^0(d) \perp y^1(d) \geq 0, && d \in \mathcal{D}\\
        &&&&&  y^0(d) + y^1(d) = 1, && d \in \mathcal{D}\\
        &&&&&  0 \leq y^0(d), y^1(d) \leq 1, && d \in \mathcal{D}.
    \end{aligned}
    \label{eq:bigm_mpcc}
\end{equation}
Note that the big-M equations are now expressed with continuous variables that indicate the behavior of the original binary variables. By substituting \eqref{eq:bigm_mpcc} into \eqref{eq:infinite_gdp_BM}, we obtain the MPCC formulation for event constrained problems:
\begin{equation}
    \begin{aligned}
        &&\min_{} &&& M_d f(z, Dq, q(d), d) \\
        && \text{s.t.} &&& g_j(z, Dq, q(d), d) \leq 0, && j \in \mathcal{J}, \ d \in \mathcal{D} \\
        &&&&& h_i(d) \leq M_i (1 - y^1(d)),  && i \in \mathcal{I}, \ d \in \mathcal{D} \\
        &&&&& 0 \leq y^0(d) \perp y^1(d) \geq 0, && d \in \mathcal{D}\\
        &&&&&  y^0(d) + y^1(d) = 1, && d \in \mathcal{D}\\
        &&&&& \mathbb{E}_d[\Omega_{MIP}(y^1(d))] \geq \alpha \\
        &&&&&  0 \leq y^0(d), y^1(d) \leq 1, && d \in \mathcal{D}.\\ 
    \end{aligned}
    \label{eq:infinite_BM_MPCC}
\end{equation}
Here, the complementarity constraints are nonsmooth and pose a significant challenge for NLP solvers like \textsc{Ipopt}. Hence, smooth approximations like the smooth-max approximation are required in practice as discussed in Section \ref{sec:mpcc}. For this work, we focus on the smooth-max approximation for standard big-M formulations that are analogous to those used in the chance constraint literature. Unlike the CVaR and SigVaR approximations, we cannot guarantee that approximate solutions derived via MPCC formulations are conservative and do not incur measured constraint violations that exceed $1-\alpha$. This means that an approximate solution may violate the event constraint in proportion to the size of the tolerance $\epsilon$. However, the numerical results presented in Section \ref{sec:cases} show that conservative solutions are typically obtained in practice. We also note that MPCC reformulation can be applied to BM and HR GDP transformations discussed in Section \ref{sec:gdp_solution_techniques}, but we leave such investigations to future work.

\section{Case Studies}\label{sec:cases}
In this section, we present numerical studies to illustrate the modeling and solution techniques described in the above sections. These demonstrate the advantages of using event constraints to relax constraints over a portion of the indexing domains in InfiniteOpt problems. Moreover, we investigate the relative advantages/disadvantages of the different solution techniques proposed in Section \ref{sec:formulations}. 
In particular, the case study in Section \ref{sec:grid_case} demonstrates how incorporating domain-specific logic can improve the objective function by avoiding overly conservative designs. 
Sections \ref{sec:pandemic_case} and \ref{sec:pde_case} illustrate the utility of the proposed continuous approximation methods for nonlinear cases in dynamic and PDE-constrained optimization where solving the exact big-M formulation becomes computationally intractable.

The results are evaluated on a Linux machine with 8 Intel\textregistered \, Xeon\textregistered \, Gold 6234 CPUs running at 3.30 GHz with 128 total hardware threads and 1 TB of RAM running Ubuntu. The solvers used where 
Gurobi \texttt{v11.0.2}, 
\textsc{Ipopt} \texttt{v3.14.14} \cite{wachter2006implementation}, and 
Juniper \texttt{v0.9.2} \cite{juniper}. 
Similarly, the solvers 
\textsc{Conopt4} \texttt{v4.31} \cite{drud1994conopt}, 
BARON \texttt{v45.6.0}  \cite{kilincc2018exploiting}
and SCIP \texttt{v8.1} \cite{bestuzheva2021scip} were utilized and accessed through 
GAMS \texttt{v45.6.0}.  
All the scripts were developed in InfinteOpt.jl and are freely available at \url{https://github.com/pulsipher/event_constraints}.

\subsection{Power Grid Design}\label{sec:grid_case}
We begin by examining the design of the IEEE-14 power distribution network, as shown in Figure \ref{fig:ieee_topology} \cite{dabbagchi1962ieee}. In this diagram, generators, demands, and nodes are represented by blue squares, red squares, and green circles, respectively. The objective is to minimize the cost of increasing generator and line capacity to meet a set of stochastic demands, while ensuring compliance with an event constraint related to capacity limits (i.e., how the capacity limits are maintained under varying demand conditions). This case study demonstrates how incorporating domain-specific logic can result in less conservative system designs.
\begin{figure}[!htb]
	\includegraphics[width=0.6\textwidth]{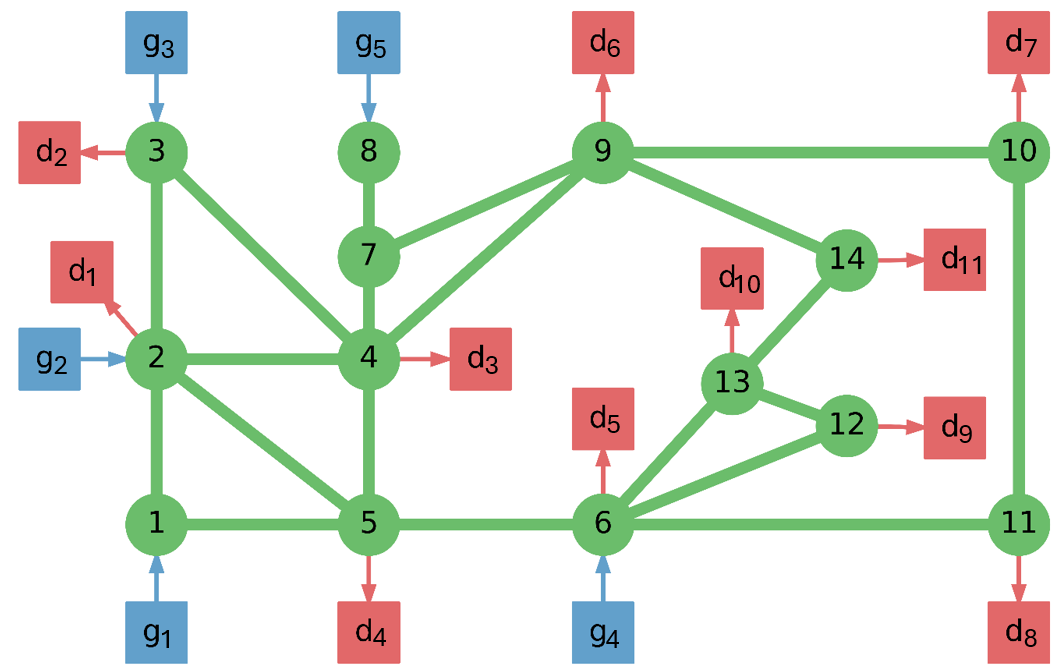}
	\centering
	\caption{IEEE-14 power grid network topology.}
	\label{fig:ieee_topology}
\end{figure}

\subsubsection{Formulation}
We adapted the formulation in \cite{pulsipher2019scalable} to an infinite-dimensional GDP formulation to account for event constraints. We let $\mathcal{G}$ and $\mathcal{L}$ be the sets of generators and lines respectively. The objective is to minimize the cost of the installed generator capacities $z_{g, i} \in [0, 300], \ i \in \mathcal{G},$ and line capacities $z_{l, j} \in [0, 100], \ j \in \mathcal{L}$ as:
\begin{equation}
    \min \sum_{i\in \mathcal{G}}z_{g,i} + \sum_{j \in \mathcal{L}}z_{l,j}.
    \label{eq:grid_objective}
\end{equation} 
The network is subject to an energy balance at each node $n \in \mathcal{N}$ with line flows $q_l(\xi): \mathcal{D}_\xi \times \mathcal{L} \rightarrow [-150, 150]$, generation $q_g(\xi): \mathcal{D}_\xi \times \mathcal{G} \rightarrow [0, 632]$, and uncertain demand $\xi$:
\begin{equation}
    \sum_{j\in \mathcal{L}^{\text{in}}_n}q_{l, j}(\xi) - \sum_{j\in \mathcal{L}^{\text{out}}_n}q_{l, j}(\xi) + \sum_{i\in \mathcal{G}_n}q_{g,i}(\xi) - \sum_{r\in \mathcal{R}_n} \xi_r = 0, \qquad n \in \mathcal{N}, \ \xi \in \mathcal{D}_\xi
    \label{eq:grid_balance}
\end{equation}
where $\mathcal{G}_n$ is the set of generators connected to node $n$, $\mathcal{R}_n$ is the set of demands at node $n$, and $\mathcal{L}^{\text{in}}_n$, $\mathcal{L}^{\text{out}}_n$ denote the set of lines that flow toward and away from a node $n$, respectively. The stochastic demands $\xi$ are described by a multivariate normal distribution $\mathcal{N}(\mu, \Sigma)$ using the mean vector $\mu$ from Table \ref{tab:grid_mean} in Appendix \ref{app:power_grid} and a covariance matrix with $1200$ on the diagonal and $240$ off-diagonal. 
In the case study, a key constraint is determining whether, for a given realization of $\xi$, installed capacities exceed a safety operating threshold. To address this, first we express the disjunction arising from satisfying the capacity limit at each generator $i \in \mathcal{G}$ for a given demand $\xi$:
\begin{equation}
    \left[
    \begin{gathered}
        Y_{g,i}(\xi) \\
        q_{g,i}(\xi) \leq \bar{\bar{q}}_{g,i} + z_{g,i} \\
    \end{gathered}
    \right]
    \vee
    \left[
    \begin{gathered}
        \lnot Y_{g,i}(\xi) \\
        q_{g,i}(\xi) > \bar{\bar{q}}_{g,i} + z_{g,i}  \\
    \end{gathered}
    \right], \qquad i \in \mathcal{G}, \ \xi \in \mathcal{D}_\xi
    \label{eq:gen_disjunction}
\end{equation}
where $\bar{\bar{q}}_{g,i}$ is the threshold generator capacity as shown in Table \ref{tab:grid_bounds} in Appendix \ref{app:power_grid}. 
Similarly, we enforce a disjunction at each line  to determine if the maximum safety threshold $\bar{\bar{q}}_{l, j}=50, \ j \in \mathcal{L}$ is respected:
\begin{equation}
    \begin{aligned}
        &\left[
        \begin{gathered}
            Y^L_{l,j}(\xi) \\
            -\bar{\bar{q}}_{l, j} - z_{l,j} > q_{l, j}(\xi)  \\
        \end{gathered}
        \right]
        \vee
        \left[
        \begin{gathered}
            Y_{l,j}(\xi) \\
            -\bar{\bar{q}}_{l, j} - z_{l,j} \leq q_{l, j}(\xi)\\
            q_{l, j}(\xi) \leq \bar{\bar{q}}_{l, j} + z_{l,j}
        \end{gathered}
        \right]
        \vee
        \left[
        \begin{gathered}
            Y^U_{l,j}(\xi) \\
            q_{l, j}(\xi) > \bar{\bar{q}}_{l,j} + z_{l,j} \\
        \end{gathered}
        \right], \quad j \in \mathcal{L}, \ \xi \in \mathcal{D}_\xi
    \end{aligned}
    \label{eq:line_disjunction}
\end{equation}
where either the capacity limit of the line is met $Y_{l,j}$, or one of the lower $Y^L_{l, j}$ or upper $Y^U_{l,j}$ limits is violated. We seek to enforce an event constraint on whether the safety limits for the generators and lines are respected:
\begin{equation}
    \begin{aligned}
        &&&&&W(\xi) \iff {\Omega}(Y_{g}(\xi), Y_{l}(\xi)), && \xi \in \mathcal{D}_\xi \\
        &&&&&\mathbb{E}_\xi[W(\xi)] \geq \alpha. \\
    \end{aligned}
    \label{eq:grid_event_constr}
\end{equation}

We solve the formulation given by Equations \eqref{eq:grid_objective}-\eqref{eq:grid_event_constr} by applying SAA with 1,000 Monte Carlo samples of $\xi$ where each element of each realization is truncated at 0 such that no negative demands are incurred. 
The resulting model constitutes an mixed-integer linear program that was solved using Gurobi. 
The size of the model depends on the chosen GDP solution method. For reference, when using the standard GDP big-M formulation, the resulting model contains 27,025 continuous variables, 68,000 binary variables, and 133,001 constraints.

\subsubsection{Event Constraint with Intersection and Arbitrary Logic}

We explore how different event logic $\Omega(\cdot)$ impacts the Pareto frontier of this design problem. In particular, we first consider intersection logic (i.e., a joint-chance constraint) as:
\begin{equation}
    \Omega_\wedge(Y_g, Y_l) := \left(\bigwedge_{i \in \mathcal{G}} Y_{g, i} \right) \wedge \left(\bigwedge_{j \in \mathcal{L}} Y_{l, j} \right).
    \label{eq:case_joint_logic}
\end{equation}
which enforces that all generators and lines must simultaneously stay within their respective safety limits for a given scenario with a minimum probability $\alpha$. This logic may be overly conservative since it might not be necessary to have all equipment simultaneously within their safety limits, resulting in prohibitively expensive designs. 

To explore the effect of using different event logic, we examine how incorporating domain-specific knowledge into the logic can facilitate less costly designs. We illustrate this by only requiring a minimum number of generator and line limits to be simultaneously respected:
\begin{equation}
    \begin{aligned}
    \Omega_\text{atleast}(Y_g, Y_l; Y_{g, \text{min}}, Y_{l, \text{min}}) := &\text{\footnotesize ATLEAST}(Y_{g, \text{min}}, Y_g) \wedge \text{\footnotesize ATLEAST}(Y_{l, \text{min}}, Y_l)
    \end{aligned}
    \label{eq:case_atleast_logic}
\end{equation}
where $Y_{g, \text{min}} \in \{1, \dots, |\mathcal{G}|=5\}$, $Y_{l, \text{min}} \in \{1, \dots, |\mathcal{L}|=20\}$ are the minimum number generator and line capacity limits enforced for a particular value of $\xi$, respectively. Note that each {\footnotesize ATLEAST}$(\cdot)$ operator in the above expression can be reformulated into a Boolean variable:
\begin{equation}
\text{\footnotesize ATLEAST}(Y_{k, \text{min}}, Y_k) \Leftrightarrow Y^{\text{ATLEAST}}_k ,  \qquad k \in \{\mathcal{G},\mathcal{L}\}
\end{equation}
associated to the following disjunction following disjunction:
\begin{equation}
    \begin{aligned}
    \left[
    \begin{gathered}
        Y^{\text{ATLEAST}}_k \\
         \sum_{i \in k} y_{k,i} \geq Y_{k,min}  \\
    \end{gathered}
    \right]
    \vee
    \left[
    \begin{gathered}
        \lnot Y^{\text{ATLEAST}}_k \\
        \sum_{i \in k} y_{k,i} \leq Y_{k,min} - 1  \\
    \end{gathered}
    \right], && \qquad k \in \{\mathcal{G},\mathcal{L}\}
    \end{aligned}
\end{equation}
where $y_k$ represents the binary variable associated with $Y_k$. 

\begin{figure}[!htb]
	\includegraphics[width=0.7\textwidth]{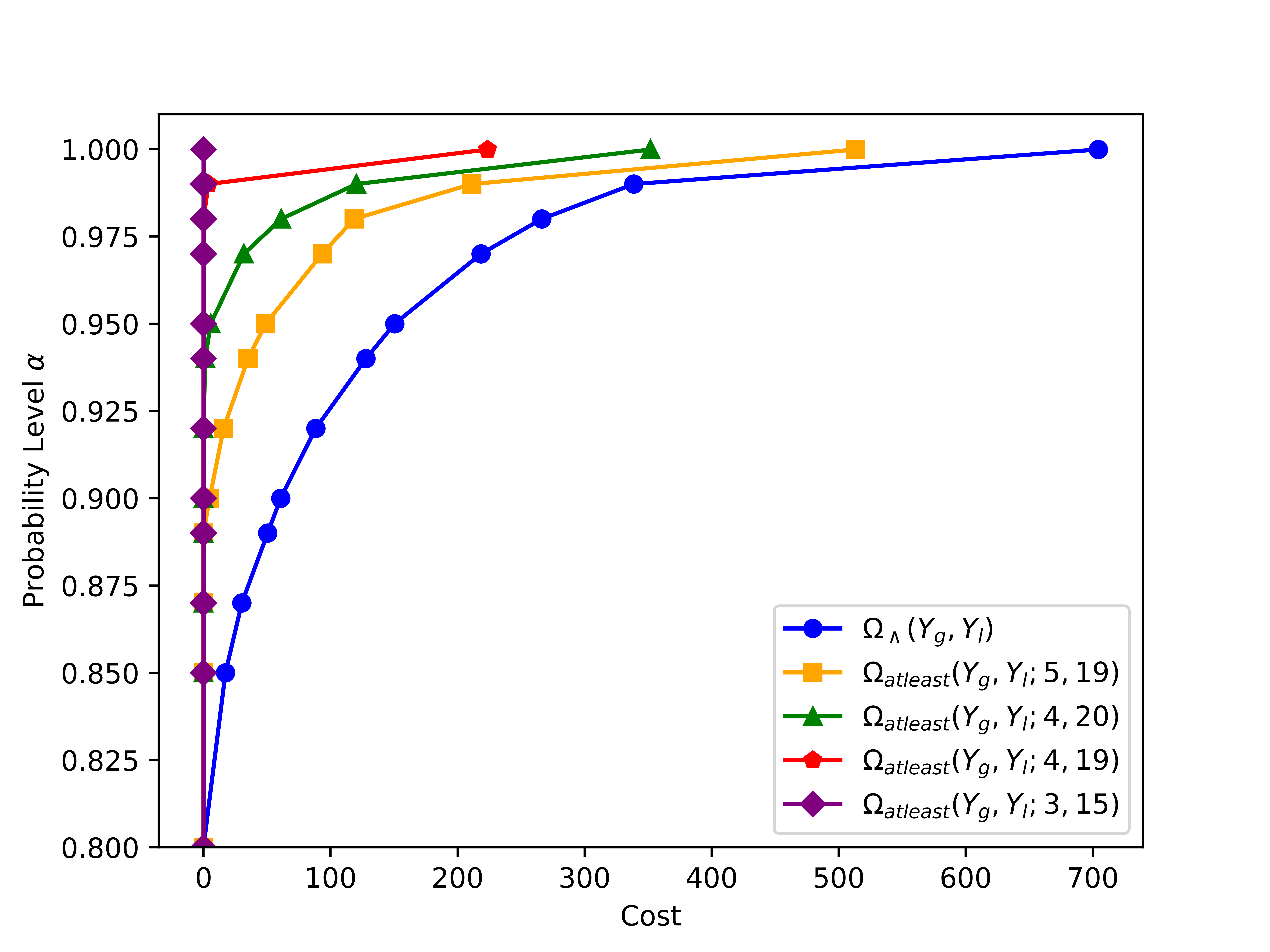}
	\centering
	\caption{Optimal Pareto curves using different event constraint logic $\Omega_\wedge(\cdot)$ and $\Omega_\text{atleast}(\cdot)$.}
	\label{fig:pareto}
\end{figure}

By adjusting the values of $Y_{g, \text{min}}$ and $Y_{l, \text{min}}$ to generate different logic configurations (i.e., events), varying the target probability $\alpha$, and solving the problem, we derive the optimal Pareto curves shown in Figure \ref{fig:pareto}.
Each curve illustrates the trade-off between design robustness, ensuring higher safety probability, and the associated cost. Essentially, achieving greater robustness comes at a higher cost.

As we would expect, the intersection logic $\Omega_\wedge(\cdot)$ incurs the greatest costs since it strictly enforces that every safety constraint be satisfied for each instance of $\xi$. Hence, we have to add more capacity to the design relative to the frontiers derived from $\Omega_\text{atleast}(\cdot)$ for a fixed probability level $\alpha$. We observe how decreasing the values of $Y_{g, \text{min}}$ and $Y_{l, \text{min}}$ decreases the costs in engineering sufficient capacity for feasible operation. For this application, this means that if only a subset of line/generator capacity constraints need to be respected for a particular demand profile $\hat{\xi}$, then we can embed that logic into our problem formulation and obtain a lower cost design relative to using traditional joint-chance constraint formulations. Hence, we can use application specific logic to avoid over-engineering, avoiding unnecessary costs.

\begin{figure}[!htb]
	\centering
    \includegraphics[width=0.7\textwidth]{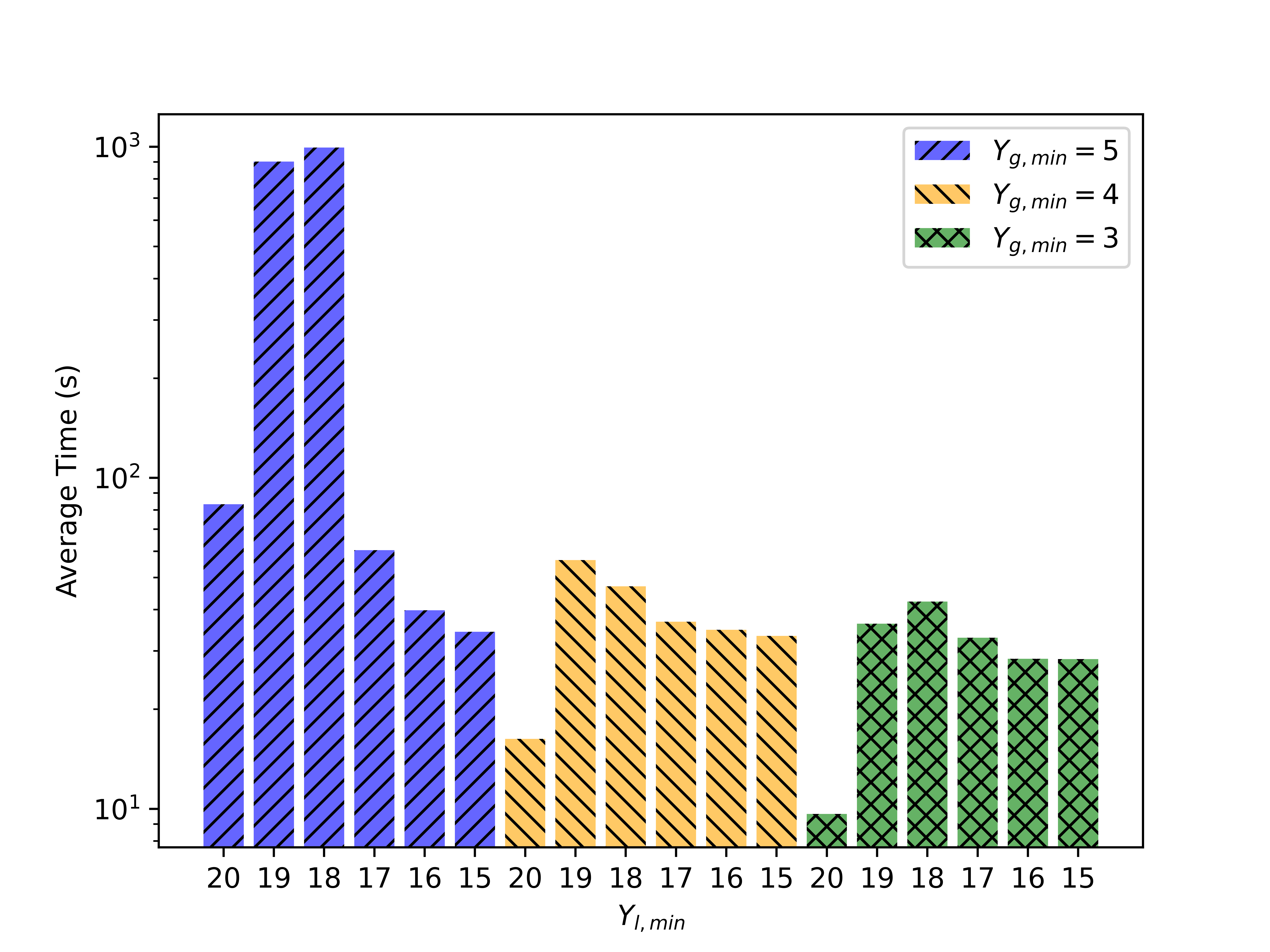}
    \caption{The mean computational time required to solve each Pareto pair for varied $Y_{g, \text{min}}$ and $Y_{l, \text{min}}$ using one-sided big-M reformulation method.}
    \label{fig:grid_time}
\end{figure}

Integrating arbitrary logic implies adding extra constraints that might impact the computational time to solve the model. Figure \ref{fig:grid_time} shows how the different event logic affects the solution time of the problem. In particular, we average over the time required to solve each Pareto pair when computing the Pareto frontier with a specified $Y_{g, \text{min}}$ and $Y_{l, \text{min}}$ using the traditional one-sided big-M reformulation method. Note that the event with $Y_{g, \text{min}} = 5$ and $Y_{l, \text{min}} = 20$ denotes traditional intersection logic $\Omega_\wedge(\cdot)$ that arises from joint-chance constraints. Interestingly, all but two instances (i.e., $Y_{g, \text{min}} = 5$ with $Y_{l, \text{min}} \in \{19, 18\}$) exhibit reduced computational times relative to the joint-chance baseline. 
This demonstrates that events with more complex logic aggregation do not inherently lead to higher computational costs; in some cases, they can actually reduce it. 
We chose the one-sided big-M reformulation as it proved to be the fastest method, but the trend of reduced computational cost with added logic was consistently observed across various GDP solution techniques.

 \begin{figure}[!htb]
	\centering
    \includegraphics[width=0.7\textwidth]{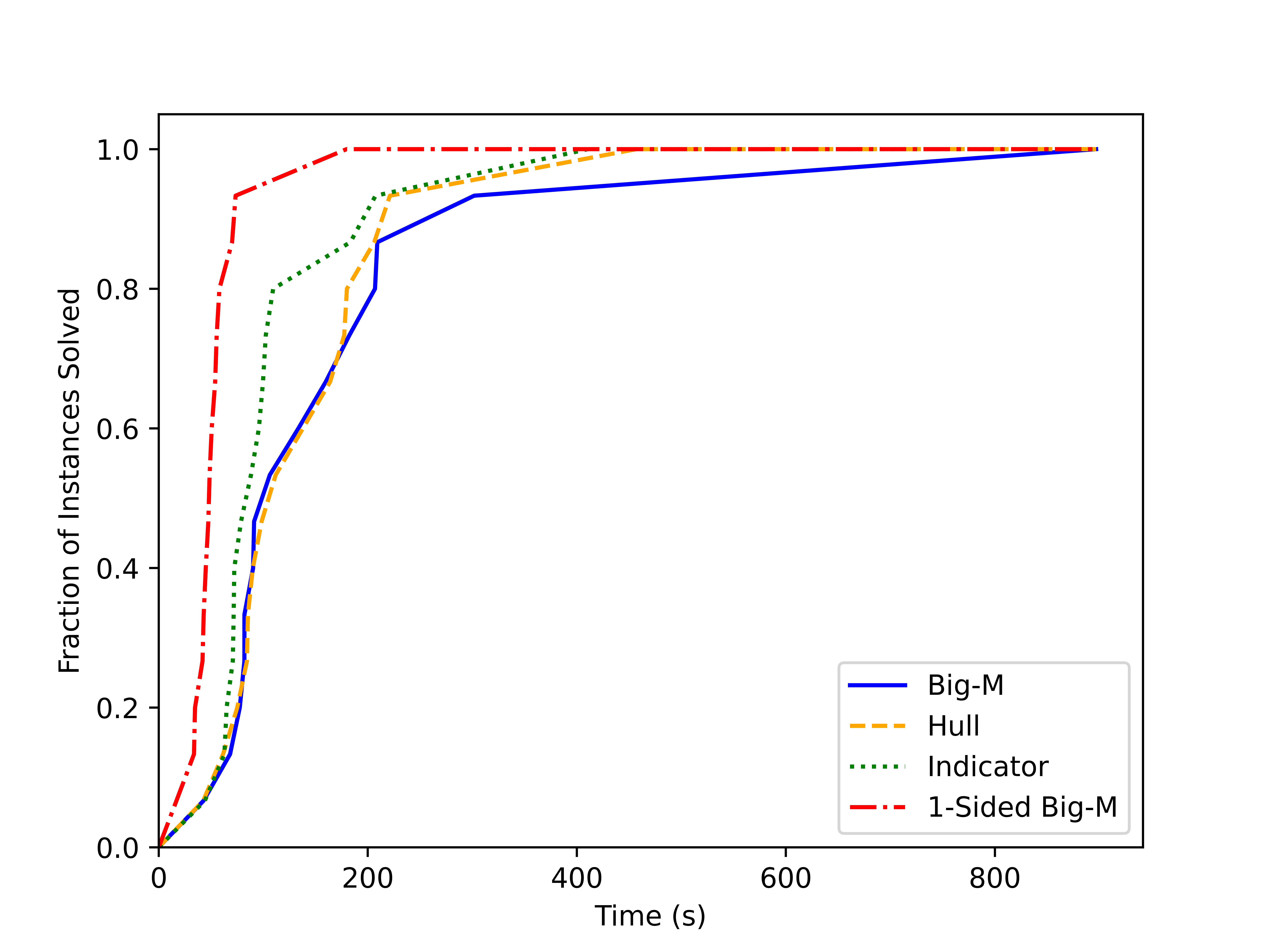}
    \caption{Performance of various GDP transformation methods in solving Pareto pairs with $\Omega_\text{atleast}(Y_g, Y_l ; 4, 19)$ logic.}
    \label{fig:grid_performance}
\end{figure}

Finally, we compare the performance of the different GDP transformation methods in this case study. In particular, we evaluate the GDP big-M (labeled as big-M), the hull reformulation, the indicator constraint reformulation, and the standard one-sided big-M solution methods.
In both big-M methods, tight values of $M$ were computed through interval arithmetic techniques \cite{trespalacios2015improved}. For illustrative purposes we show the event where at least four generator and 19 lines have to be within the safety threshold limit (i.e., $\Omega_\text{atleast}(Y_g, Y_l ; 4, 19)$). Figure \ref{fig:grid_performance} shows the performance plot indicating the fraction of Pareto pair instances that each method solved as a function of wall-time. 

In general, the traditional one-sided big-M approach consistently required the least time to solve each Pareto pair. This efficiency is likely due to the one-sided formulation having fewer constraints and variables, as the two-sided big-M introduces an additional constraint to handle the negation disjunction, while the hull reformulation requires more variables and constraints due to disaggregation. Both the traditional GDP methods and indicator constraint reformulation exhibited comparable performance. We highlight the $\Omega_\text{atleast}(Y_g, Y_l ; 4, 19)$ event for in Figure \ref{fig:grid_performance} to observe the potential advantages of the indicator constraint and hull reformulations over the two-sided big-M. However, across all evaluated events, the one-sided big-M consistently outperformed the others, despite the overall similarity in performance among these three methods.

\subsection{Optimal Disease Control}\label{sec:pandemic_case}

In this case study, we aim to optimally control quarantine measures to mitigate the spread of an infectious disease with minimal intervention. The objective of this model is to reduce isolation measures while ensuring the proportion of infected individuals remains below a specified threshold. This study highlights the application of event constraints in dynamic optimization, particularly within a deterministic optimal control framework, and offers a comparative analysis of different solution techniques based on their performance.

\subsubsection{Formulation}
The spread of the virus can be modeled through a given population using the susceptible-exposed-infected-recovered (SEIR) model \cite{ARON1984665}, which considers 4 population subsets that follow: 
\begin{center}
Susceptible → Exposed → Infectious → Recovered 
\end{center}
We define the fractional populations as follows: susceptible individuals $s(t): \mathcal{D}_t \rightarrow [0,1]$, exposed individuals who are not yet infectious $e(t): \mathcal{D}_t \rightarrow [0,1]$, infectious individuals $i(t): \mathcal{D}_t \rightarrow [0,1]$, and recovered individuals $r(t): \mathcal{D}_t \rightarrow [0,1]$, who are considered immune to future infection. The control variable $u(t) \in [0, 0.8]$ represents the enforced isolation measures, corresponding to levels of social distancing. These variables are normalized such that $s(t)+e(t)+i(t)+r(t)=1$. The deterministic SEIR model is formalized as follows:
\begin{equation}
    \begin{aligned}
        &&&&& \frac{ds(t)}{dt} = (u(t) - 1) \rho s(t)i(t), && t \in \mathcal{D}_t \\
        &&&&& \frac{de(t)}{dt} = (1 - u(t)) \rho s(t)i(t) - \zeta e(t), && t \in \mathcal{D}_t \\
        &&&&& \frac{di(t)}{dt} = \zeta e(t) - \eta i(t), && t \in \mathcal{D}_t \\
        &&&&& \frac{dr(t)}{dt} = \eta i(t), && t \in \mathcal{D}_t \\
        &&&&& s(0) = s_0, e(0) = e_0, i(0) = i_0, r(0) = r_0 \\
    \end{aligned}
    \label{eq:pandemic_SEIR}
\end{equation}
where $\rho$ is the infection rate, $\eta$ is the recovery rate, $\zeta$ is the incubation rates, and $s_0, e_0, i_0, r_0 \in \mathbb{R}$ denote the initial population fractions. The objective is to minimize the control interventions over the entire time horizon, defined as:
\begin{equation}
    \begin{aligned}
        &&\min_{u(t)} &&& \int_{t \in \mathcal{D}_t} u(t) dt. \\
    \end{aligned}
    \label{eq:pandemic_objective}
\end{equation}
Moreover, we enforce a path constraint that requires fraction of infected individuals $i(t)$ to remain below a limit $i_{max} = 0.02$:
\begin{equation}
    \begin{aligned}
        &&&&& i(t) \leq i_{max}, && t \in \mathcal{D}_t \\
    \end{aligned}
    \label{eq:pandemic_constraint}
\end{equation}
which loosely represents the capacity of the healthcare system to treat infected individuals. We solve the full formulation given by Equations \eqref{eq:pandemic_SEIR}-\eqref{eq:pandemic_constraint} solved via direct transcription using the parameters defined in Table \ref{tab:pandemic_parameters} with $\mathcal{D}_t = [0, 200]$ discretized over 101 equidistant time points.

\begin{table}[!htb]
    \centering
    \caption{Parameter values for the optimal disease control case study.}
    \label{tab:pandemic_parameters}
    \begin{tabular}{ccccccccc}
        \toprule
        $\rho$ & $\eta$ & $\zeta$ & $s_0$ & $e_0$ & $i_0$ & $r_0$ \\ 
        \midrule
        0.727 & 0.303 & 0.3 & $1-10^{-5}$ & $10^{-5}$ & 0 & 0 \\ 
        \bottomrule
    \end{tabular}
\end{table}

Constraint \eqref{eq:pandemic_constraint} presents a trade-off with objective function \ref{eq:pandemic_objective}, where the goal is to minimize intervention actions while ensuring the infected population remains within acceptable bounds. At one extreme where Constraint \eqref{eq:pandemic_constraint} is omitted, no control action is taken, and the fraction of infected individuals quickly peaks to $10$\% as shown in Figure \ref{fig:pandemic_unconstrained} in Appendix \ref{app:pandemic}. Interestingly, $u(t)$ remains under $u_{max}$ for $81.08$\% of the time horizon. In contrast, Figure \ref{fig:pandemic_constrained} in Appendix \ref{app:pandemic} illustrates the model with Constraint \eqref{eq:pandemic_constraint}, where considerable control inputs $u(t)$ are used to strictly keep $i(t) \leq i_{max}$ over $100$\% of the time horizon.

\subsubsection{Dynamic Event Constraint} \label{sec:pandemic_event}

We use an event constraint to precisely control the trade-off between Objective \eqref{eq:pandemic_objective} and Constraint \eqref{eq:pandemic_constraint}. This is given by embedding Constraint \eqref{eq:pandemic_constraint} into \eqref{eq:general_event_constr} to obtain:
\begin{equation}
    \mathbb{P}_t(i(t) \leq i_{max}) \geq \alpha
    \label{eq:pandemic_event_constraint}
\end{equation}
where $\alpha$ determines the minimum fraction of the time horizon in which the infection rate must remain the limit. The underlying expectation $\mathbb{E}_t$ is taken with constant weighting such that Constraint \eqref{eq:pandemic_event_constraint} becomes:
\begin{equation}
    \mathbb{E}_t\left[\mathbbm{1}_{i(t) \leq i_{max}}(t)\right] = \frac{1}{t_f}\int_0^{t_f}\mathbbm{1}_{i(t) \leq i_{max}}(t) dt \geq \alpha.
\end{equation}
However, we do have the modeling flexibility to choose a time-dependent weighting and/or restrict the portion of the time horizon we wish to relax (e.g., require that the constraint be strictly enforced for the first 25\% of the horizon). We also note that based on Figure \ref{fig:pandemic_unconstrained}, the minimum realizable fraction of time to enforce Constraint \eqref{eq:pandemic_constraint} is $81.08$\%; hence, setting $\alpha \leq 0.8108$ is equivalent to simply omitting Constraint \eqref{eq:pandemic_constraint} entirely.

We compare the various solution methods presented in Section \eqref{sec:formulations}. Initially, we apply a big-M transformation to obtain an exact reformulation of the problem, resulting in a nonconvex MINLP model. Subsequently, we solve the problem using the continuous approximation methods proposed in this work: the MPCC, CVaR, and SigVaR approximations.

\vspace{2mm}
\noindent \underline{Big-M Formulation}
\vspace{2mm}

\noindent The event constraint described by \eqref{eq:pandemic_event_constraint} can be reformulated using big-M as:
\begin{equation}
    \begin{aligned}
        &&&&& i(t)-i_{max} \leq M(1 - y(t)), && t \in \mathcal{D}_t \\
        &&&&& \mathbb{E}_t[y(t)] \geq \alpha &&   
    \end{aligned}
    \label{eq:pandemic_bigM}
\end{equation}
where $y(t) \in \{0, 1\}$ indicates whether Constraint \eqref{eq:pandemic_constraint} holds and we set $M=0.98$. Motivated by the performance observed in the previous case study, we employ a one-sided big-M transformation instead of the two-sided GDP reformulation. The resulting MINLP has 1,000 continuous variables, 111 binary variables, and 997 constraints.

\begin{figure}[!htb]
     \centering
     \begin{subfigure}[b]{0.45\textwidth}
         \centering
         \includegraphics[width=1\textwidth]{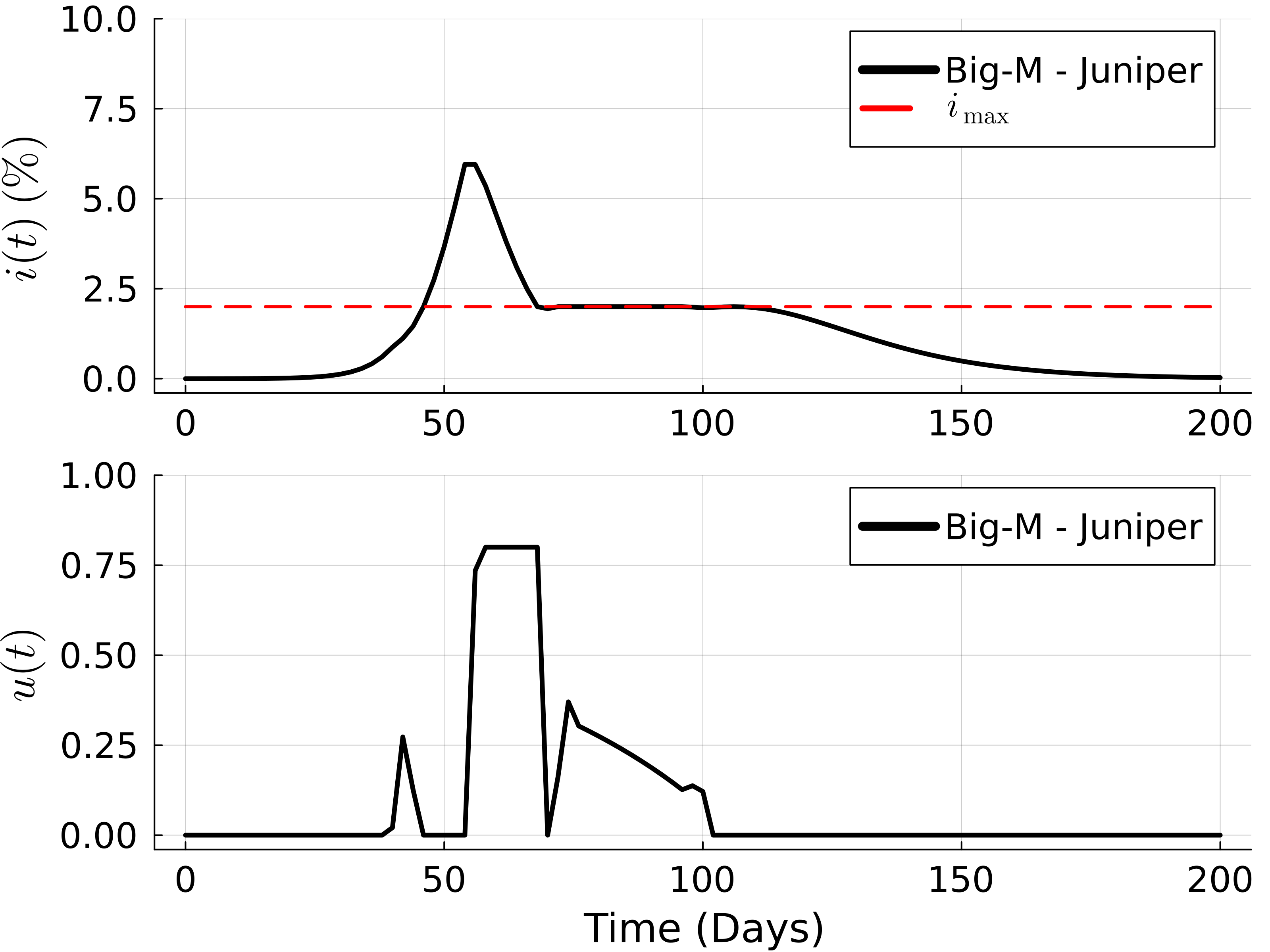}
         \caption{Big-M: $\alpha = 0.9$}
         \label{fig:pandemic_bigM_0.9}
     \end{subfigure}
          \hspace{1cm}
     \begin{subfigure}[b]{0.45\textwidth}
         \centering
         \includegraphics[width=1\textwidth]{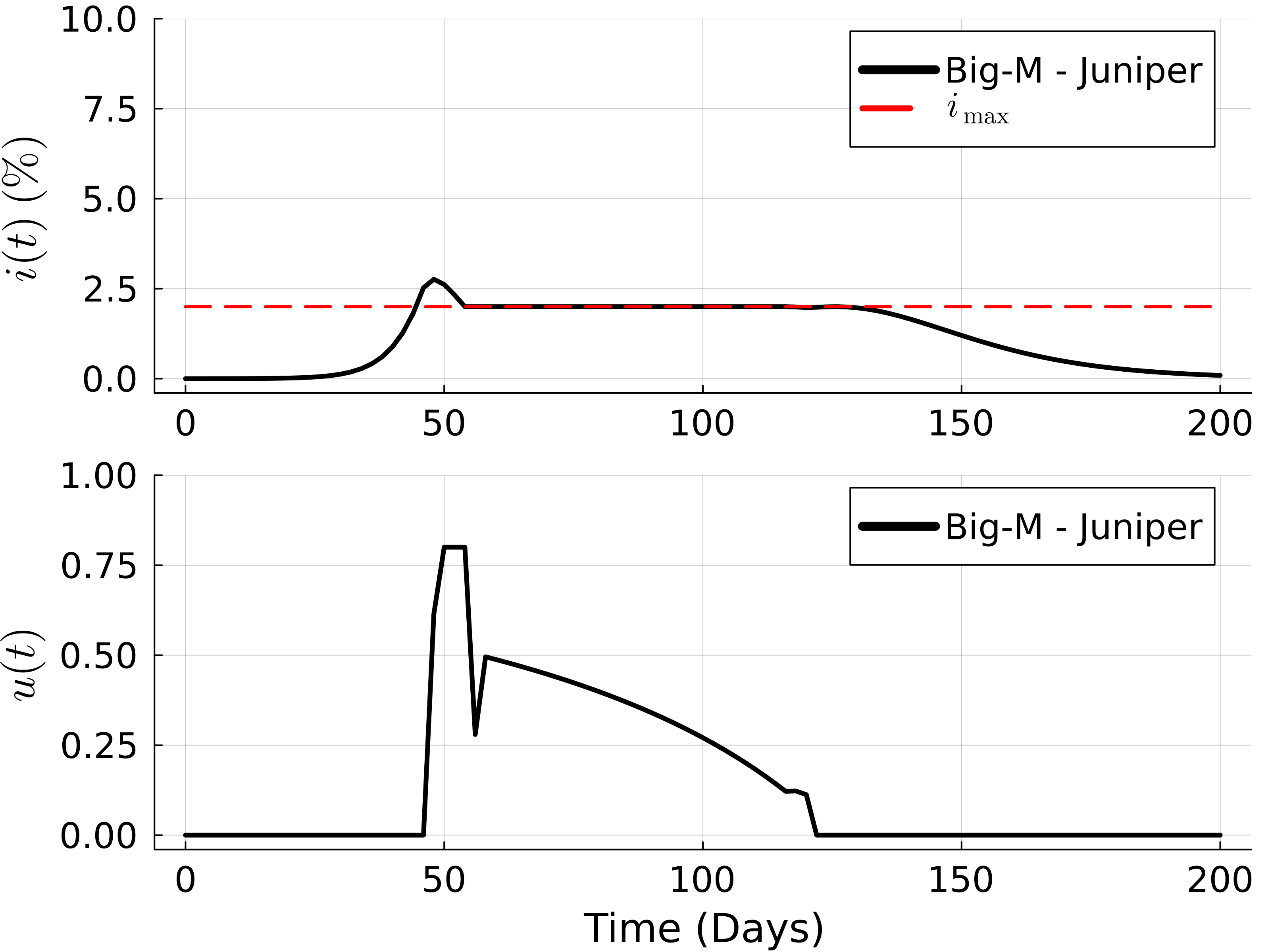}
         \caption{Big-M: $\alpha = 0.96$}
         \label{fig:pandemic_bigM_0.96}
     \end{subfigure}
    \caption{Optimal control policy and infected population fraction for different $\alpha$ values using the big-M reformulation.}
    \label{fig:pandemic_bigM}
\end{figure}

Figure \ref{fig:pandemic_bigM} shows the optimal $i(t)$ and $u(t)$ trajectories obtained using the heuristic MINLP solver Juniper \cite{juniper} with $\alpha \in \{0.9, 0.96\}$. Similar plots for other values of $\alpha$ are provided in Figure \ref{fig:pandemic_bigM_app} in the Appendix \ref{app:pandemic}. Note that the global optimizers SCIP and BARON were tested, but both failed to produce a feasible solution for the optimization problem. It is clearly evident how $\alpha$ controls the extent to which the infection limit is violated (precisely a $1-\alpha$ fraction of the time horizon). Increasing $\alpha$ decreases the peak amount of infected individuals, but this comes at the cost of implementing more stringent interventions which increases the objective value.

While this solution is based on an exact big-M reformulation, it greatly increases the computational cost to an average of 2.3 hours in contrast to solving the original optimal control problem with a hard constraint which only took a few seconds with \textsc{Ipopt} on average. This motivates the use of continuous approximations to alleviate the added computational burden imposed by the event constraint.

\vspace{2mm}
\noindent \underline{Continuous Approximation via MPCC}
\vspace{2mm}

\noindent The MPCC formulation of the event constraint can be expressed similarly to \eqref{eq:pandemic_bigM}. However, in this case, the formulation involves continuous variables $y^0(t), y^1(t)$, which are linked to the system via complementarity conditions:
\begin{equation}
\begin{aligned}
    &&&&& i(t)-i_{max} \leq M(1 - y^1(t)), && t \in \mathcal{D}_t \\
    &&&&& 0 \leq y^0(t) \perp y^1(t) \geq 0, &&  t \in \mathcal{D}_t\\
    &&&&&  y^0(t) + y^1(t) = 1, && t \in \mathcal{D}_t\\
    &&&&& \mathbb{E}_t[y(t)] \geq \alpha &&   \\
    &&&&&  y^0(t), y^1(t) \in [0,1], &&  t \in \mathcal{D}_t
\end{aligned}
\label{eq:pandemic_mpcc}
\end{equation}
This formulation is solved via \textsc{Ipopt} and \textsc{Conopt4} using the smooth-max approximation with progressively smaller values of $\epsilon$ as detailed in Table \ref{tab:tolerances} in Appendix \ref{app:pandemic}.
Each of the 40 NLP problems solved has 1,218 variables and 1,441 constraints.
The resulting profiles obtained using the MPCC approximation for $\alpha \in \{0.9, 0.96\}$ are presented in Figure \ref{fig:pandemic_MPCC} and are compared against optimal trajectories obtained by enforcing Constraint \eqref{eq:pandemic_constraint} which is the most conservative and costly scenario. Similar results for a wider collection of $\alpha$ values are detailed in Figure \ref{fig:pandemic_MPCC_app} in the Appendix \ref{app:pandemic}.

\begin{figure}[!htb]
     \centering
     \begin{subfigure}[b]{0.45\textwidth}
         \centering
         \includegraphics[width=1\textwidth]{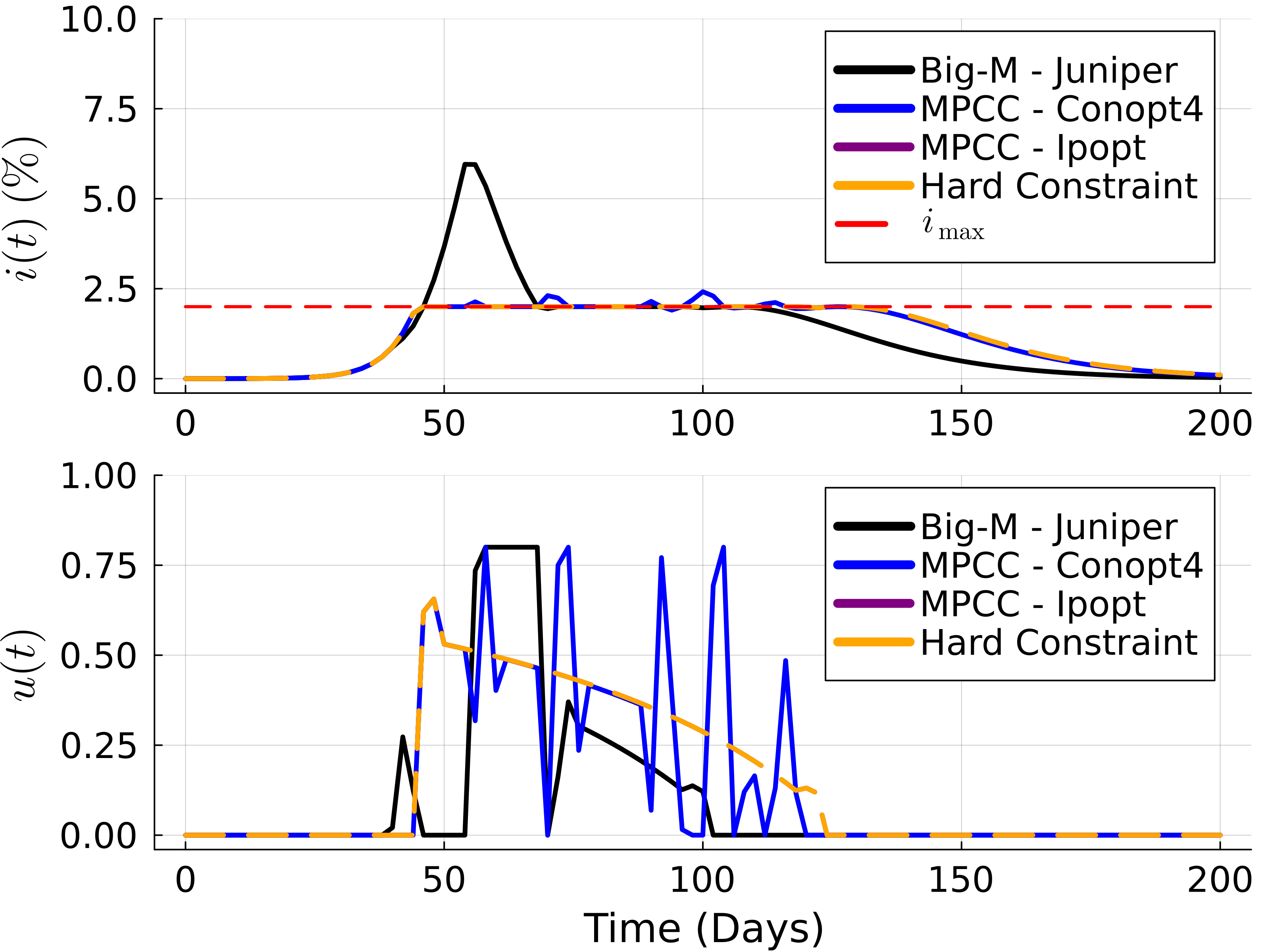}
         \caption{MPCC: $\alpha = 0.9$}
         \label{fig:pandemic_MPCC_0.9}
     \end{subfigure}
          \hspace{1cm}
     \begin{subfigure}[b]{0.45\textwidth}
         \centering
         \includegraphics[width=1\textwidth]{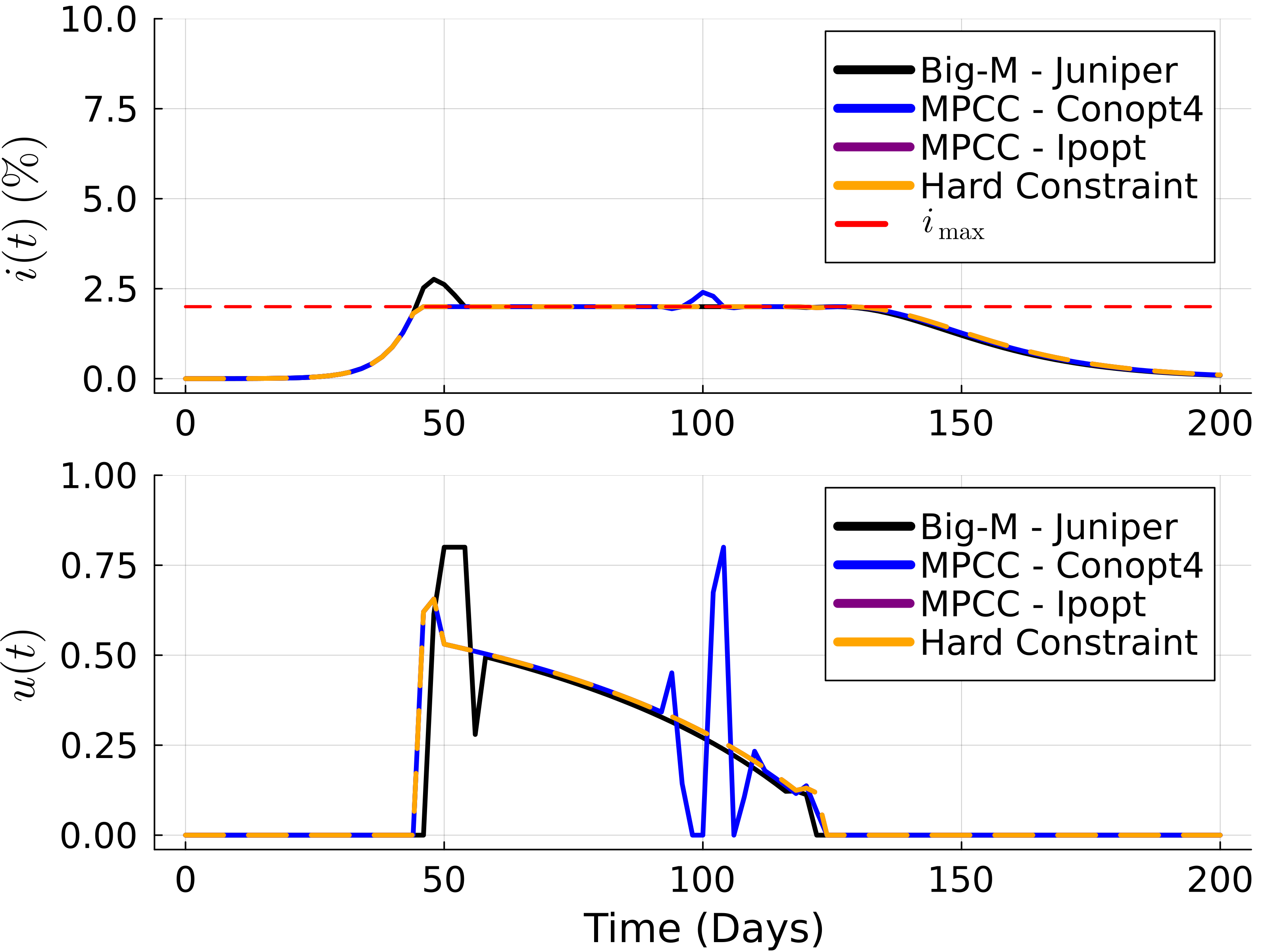}
         \caption{MPCC: $\alpha = 0.96$}
         \label{fig:pandemic_MPCC_0.96}
     \end{subfigure}
    \caption{Optimal control policy and infected population fraction for different $\alpha$ using the MPCC approximation.}
    \label{fig:pandemic_MPCC}
\end{figure}

Figure \ref{fig:pandemic_MPCC} illustrates that the solution depends on the solver used. With \textsc{Ipopt}, the solution remains consistent regardless of the value of $\alpha$, converging to the hard constraint solution. Hence, this solution fails to exploit the flexibility introduced by the event constraint and results in a higher objective function compared to the big-M solution. In contrast, \textsc{Conopt4} converges to a local minimum that does relax Constraint \eqref{eq:pandemic_constraint} in proportion to $\alpha$. Although this solution differs from the big-M solution and remains suboptimal, it demonstrates that the MPCC approximation is capable of providing approximate solutions the do relax the constraint contained in the event constraint.

\vspace{2mm}
\noindent \underline{Continuous Approximation via CVaR}
\vspace{2mm}

\noindent The CVaR approximation provides another continuous approximation, offering the tightest convex approximation to the indicator function inherent in Constraint \eqref{eq:pandemic_event_constraint}. To apply this approach, Constraint \eqref{eq:pandemic_event_constraint} becomes:
\begin{equation}
    \begin{aligned}
    &&&&& \phi(t) \geq i(t)-i_{max} - \lambda , && t \in \mathcal{D}_t \\
    &&&&& \mathbb{E}_t[\phi(t)] \leq \lambda(1-\alpha)
    \end{aligned}
    \label{eq:pandemic_cvar}
\end{equation}
where $\phi(t) \in \mathbb{R}_+$ and $\lambda \in \mathbb{R}$ are continuous variables, overall resulting in a formulation with 1,108 variables and 997 constraints. 

\begin{figure}[!htb]
     \centering
     \begin{subfigure}[b]{0.45\textwidth}
         \centering
         \includegraphics[width=1\textwidth]{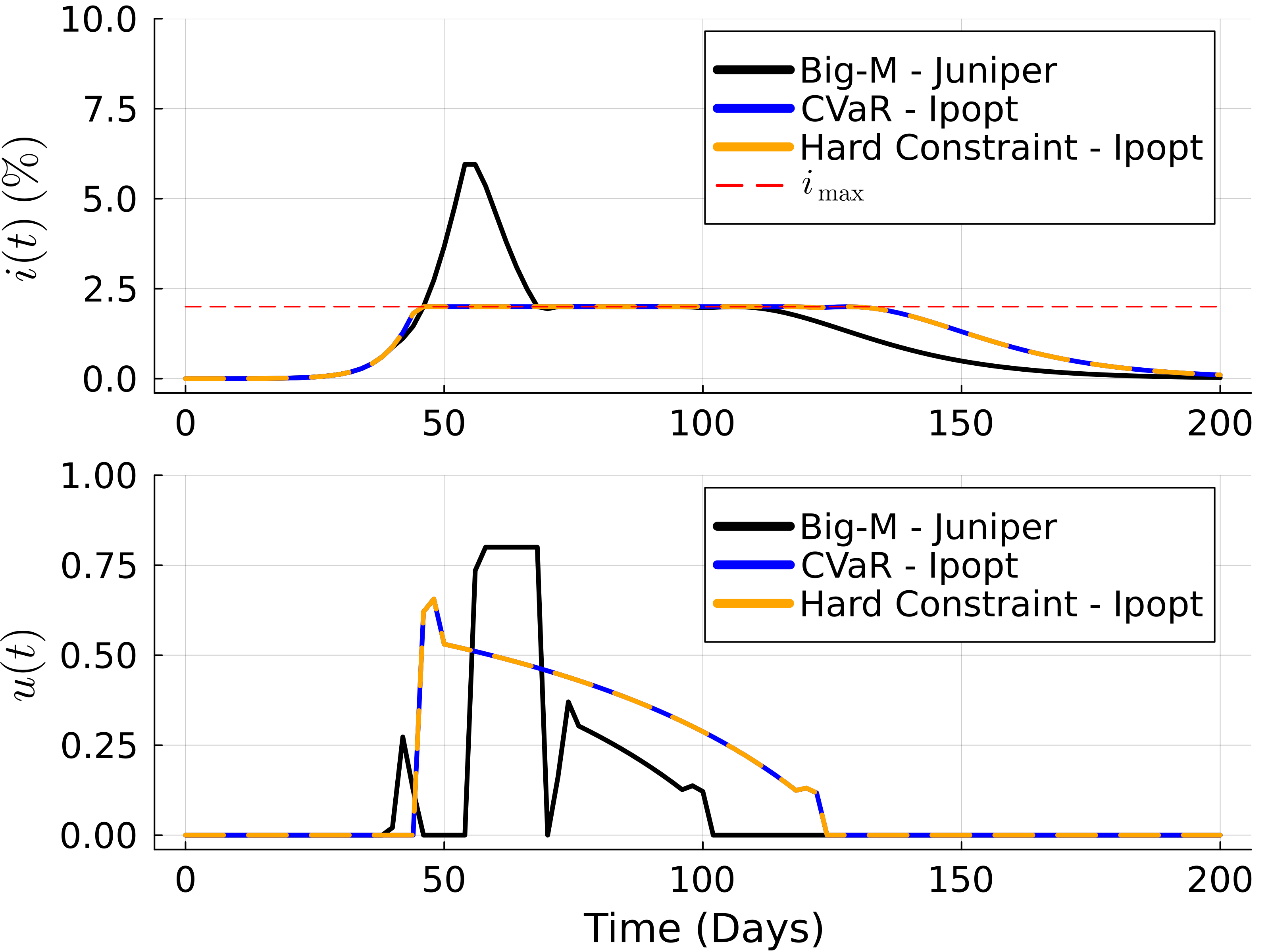}
         \caption{CVaR: $\alpha = 0.90$}
         \label{fig:pandemic_CVaR_0.90}
     \end{subfigure}
     \hspace{1cm}
     \begin{subfigure}[b]{0.45\textwidth}
         \centering
         \includegraphics[width=1\textwidth]{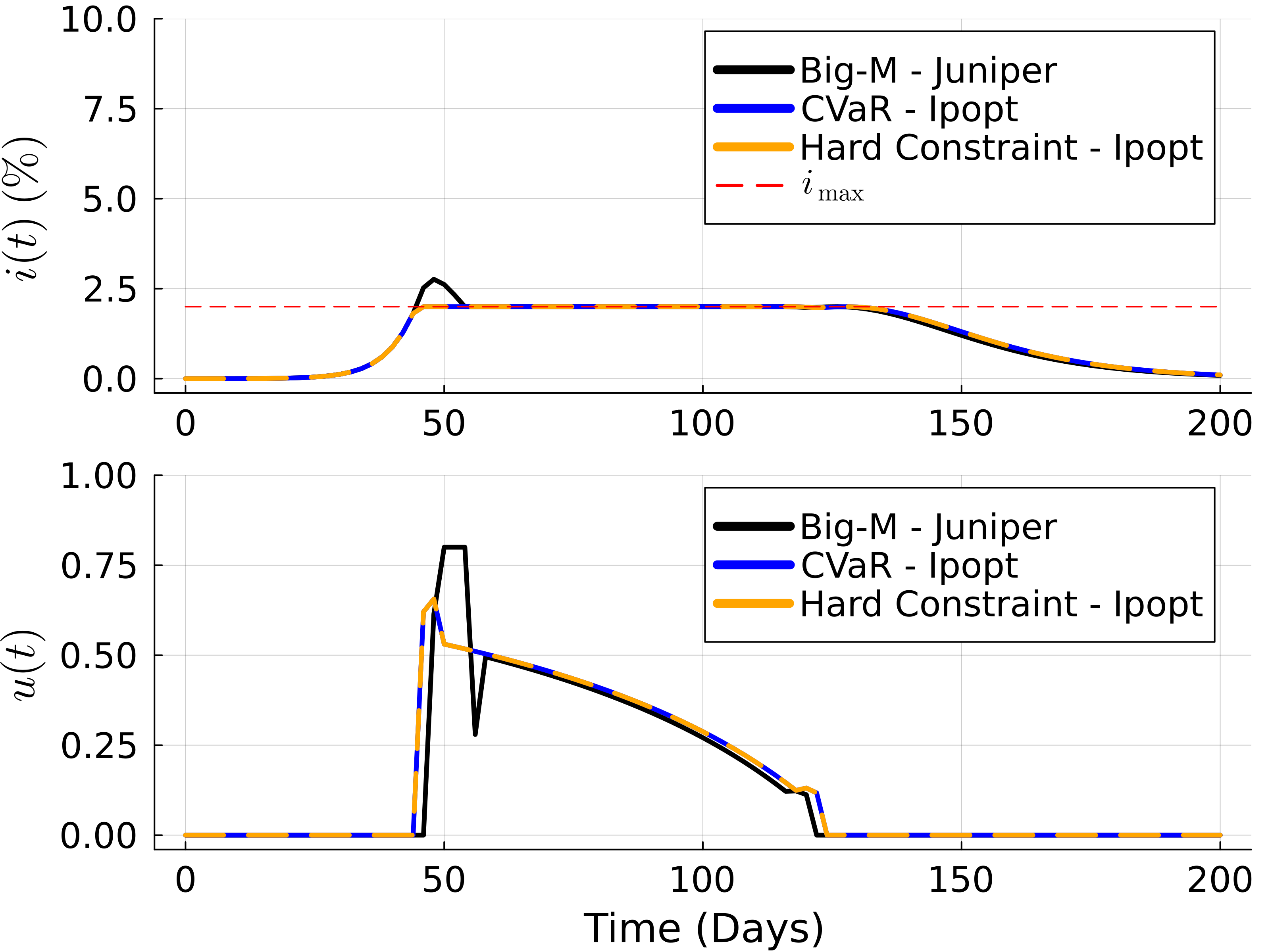}
         \caption{CVaR: $\alpha = 0.96$}
         \label{fig:pandemic_CVaR_0.96}
     \end{subfigure}
    \caption{Optimal control policy and infected population fraction for different $\alpha$ using the CVaR approximation.}
    \label{fig:pandemic_CVaR}
\end{figure}

Figure \ref{fig:pandemic_CVaR} presents the optimal profiles for $i(t)$ and $u(t)$ with $\alpha\in\{0.9, 0.96\}$ in comparison to the conservative solution obtained by enforcing Constraint \eqref{eq:pandemic_constraint} as a hard constraint. As expected, the CVaR approximation is overly conservative, yielding the same solution as the hard constrained case for all values of $\alpha$ tested. Notably, \textsc{Ipopt} and \textsc{Conopt4} yielded the same results.

\FloatBarrier

\vspace{2mm}
\noindent \underline{Continuous Approximation via SigVaR}
\vspace{2mm}

\noindent The SigVaR approximation provides a continuous conservative approximation that is often much tighter than CVaR since it uses a sigmoidal function to approximate the indicator in \eqref{eq:conservative_indicator}. For this problem, Constraint \eqref{eq:pandemic_event_constraint} becomes:
\begin{equation}
\begin{aligned}
    &&&&& \phi(t) \geq 2 \dfrac{1+\beta}{\beta + \exp(-\gamma(i(t)-i_{max}))}  - 1, && t \in \mathcal{D}_t \\
    &&&&& \mathbb{E}_t[\phi(t)] \leq 1-\alpha
\end{aligned}
\label{eq:pandemic_sigvar}
\end{equation}
where $\phi(t) \in \mathbb{R}_+$, $\lambda \in \mathbb{R}$, and initial values for $\beta$ and $\gamma$ are set to $1.55$ and $63.76$, respectively \cite{cao2020sigmoidal}. For each iteration of the SigVaR, an NLP consisting of 1,106 variables and 996 constraints is solved.

\begin{figure}[!htb]
     \centering
     \begin{subfigure}[b]{0.45\textwidth}
         \centering
         \includegraphics[width=1\textwidth]{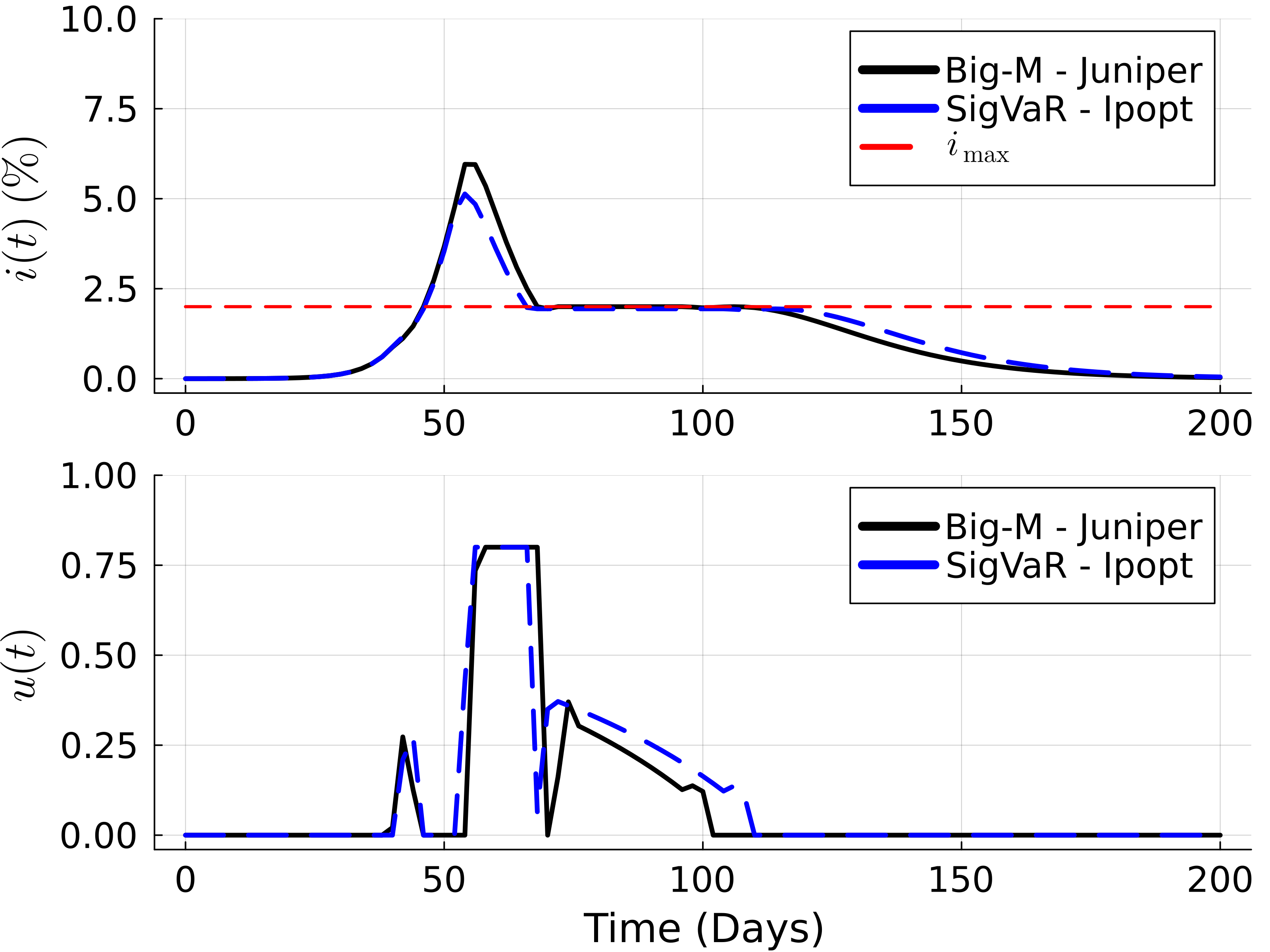}
         \caption{SigVaR: $\alpha = 0.90$}
         \label{fig:pandemic_SigVaR_0.90}
     \end{subfigure}
     \hspace{1cm}
     \begin{subfigure}[b]{0.45\textwidth}
         \centering
         \includegraphics[width=1\textwidth]{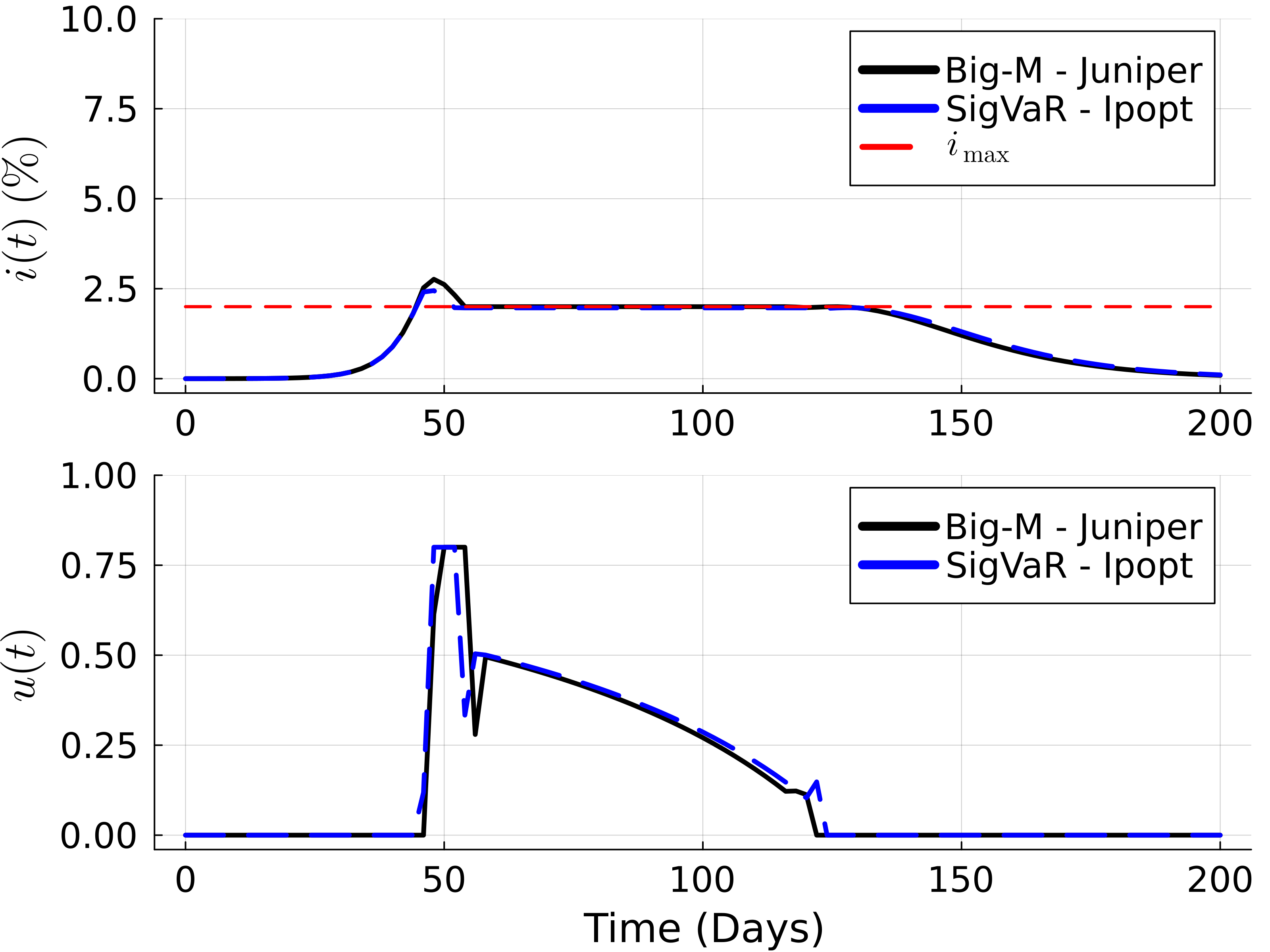}
         \caption{SigVaR: $\alpha = 0.96$}
         \label{fig:pandemic_SigVaR_0.96}
     \end{subfigure}
    \caption{Optimal control policy and infected population fraction for different $\alpha$ using the SigVaR approximation.}
    \label{fig:pandemic_SigVaR}
\end{figure}

Figure \ref{fig:pandemic_SigVaR} displays the optimal trajectories obtained using the final iteration of the SigVaR approximation for $\alpha \in \{0.9, 0.96\}$, using \textsc{Ipopt} as the solver. In both instances, the SigVaR approximation yields a close approximation relative to the big-M solution. Hence, the SigVaR method is able to provide higher quality approximations that the CVaR and MPCC approximations.

\FloatBarrier

\vspace{2mm}
\noindent \underline{Solution Method Comparison}
\vspace{2mm}

\begin{table}[!htb]
    \centering
    \caption{Solution times and optimal objective values obtained with each of the proposed solution methods for the optimal disease control problem over varied $\alpha$ values. Note that the CVaR and SigVaR solutions are obtained with \textsc{Ipopt} and the MPCC solutions are obtained via \textsc{Conopt4}.}
    \label{tab:pandemic_method_comparison}
    \begin{tabular}{c|cc|cc|cc|cc}
    \toprule
    & \multicolumn{2}{c|}{\textbf{Big-M}} & \multicolumn{2}{c|}{\textbf{MPCC}} & \multicolumn{2}{c|}{\textbf{CVaR}} & \multicolumn{2}{c}{\textbf{SigVaR}} \\ 
    \cmidrule{2-9}
    $\alpha$ & Time [s] & Objective & Time [s] & Objective & Time [s] & Objective & Time [s] & Objective \\ \midrule
    $0.85$   & $3,584$ & $10.51$ &  $3.43$  &  $28.27$  & $0.13$  & $28.81$ & $4.98$ & $11.19$   \\ 
    $0.90$   & $6,509$ & $18.35$    & $2.43$ &  $28.31$ & $0.12$  & $28.81$ & $8.63$ & $21.58$ \\
    $0.95$   & $22,338$ & $26.75$    & $2.37$  & $28.59$  & $0.13$  & $28.81$ & $3.12$ & $28.06$ \\
    $0.96$   & $11,942$ & $27.55$    & $2.30$  &  $28.60$ & $0.13$  & $28.81$ & $2.46$ & $28.70$  \\ 
    $0.97$   & $6,132$ & $28.22$    &  $2.31$ &  $28.79$ & $0.13$  & $28.81$ & $5.38$ & $29.33$  \\ 
    $0.99$   & $156.53$  & $28.74$ & $1.90$  & $28.81$  & $0.11$  & $28.81$ & $1.90$ & $29.88$  \\ 
    \bottomrule
    \end{tabular}
\end{table}

\noindent Table \ref{tab:pandemic_method_comparison} provides a summary of the optimal objective values provided and computational times incurred by each solution method. In all cases, the big-M approach provides the best objective values, but requires $17,610 s$ to solve on average. The CVaR approximation substantially reduces the computationally cost by solving a single continuous NLP, but yields a highly conservative solution that effectively enforces Constraint \eqref{eq:pandemic_constraint} over the entire horizon. Both the MPCC and SigVaR approximations require solving multiple NLP problems, leading to a greater computational cost relative to using CVaR. Interestingly, MPCC converges in less time than SigVaR even though it uses a greater number of iterations. However, the SigVaR approximation stands out as the only continuous approximation that is able to provide objectives that follow the behavior of the big-M solution, while incurring a computational cost that is orders-of-magnitude smaller.

\begin{figure}[!htb]
     \centering
     \begin{subfigure}[b]{0.45\textwidth}
         \centering
         \includegraphics[width=1\textwidth]{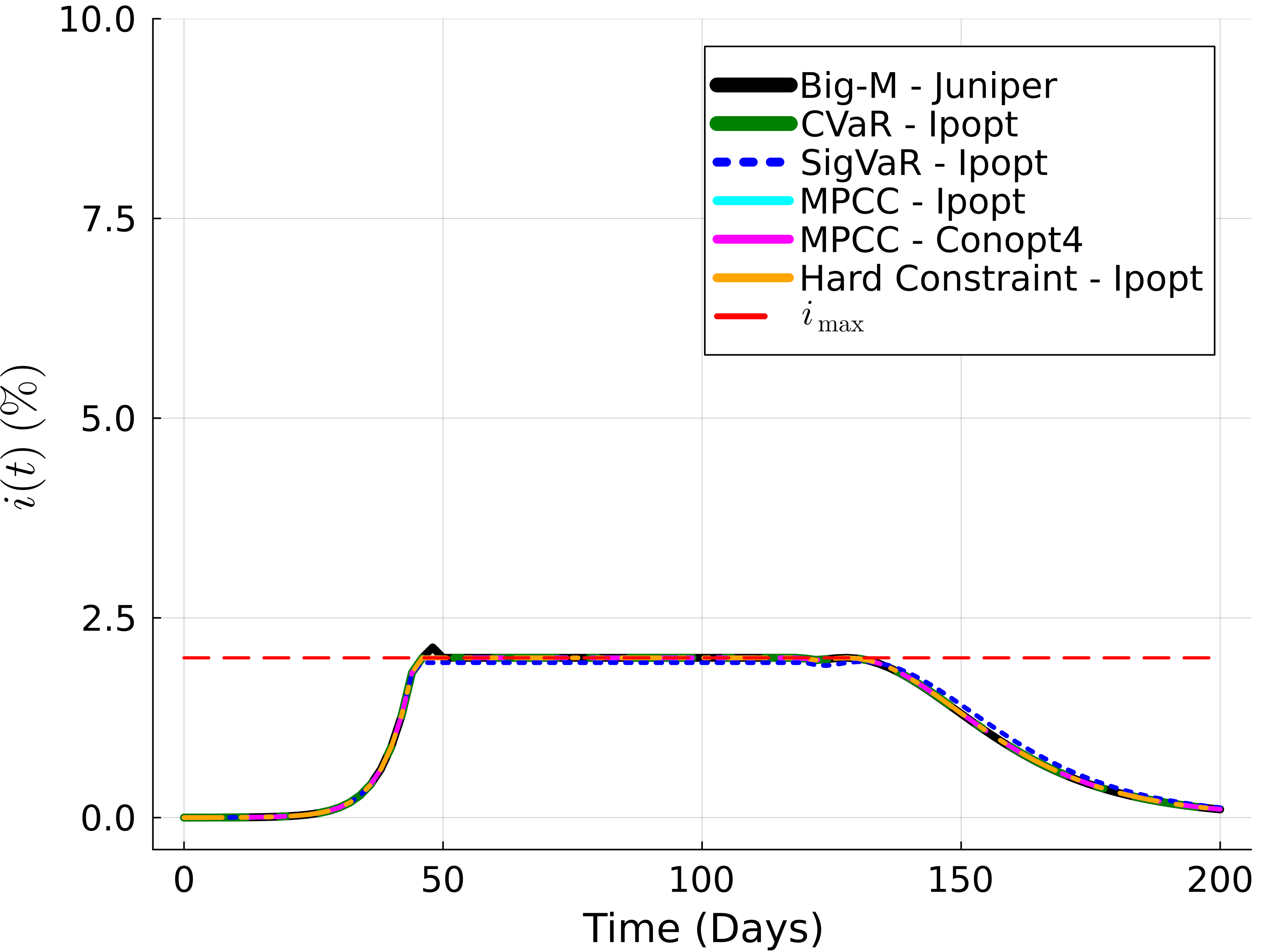}
         \caption{Optimal $i(t)$ profile}
         \label{fig:pandemic_method_comparison_i}
     \end{subfigure}
     \hspace{1cm}
     \begin{subfigure}[b]{0.45\textwidth}
         \centering
         \includegraphics[width=1\textwidth]{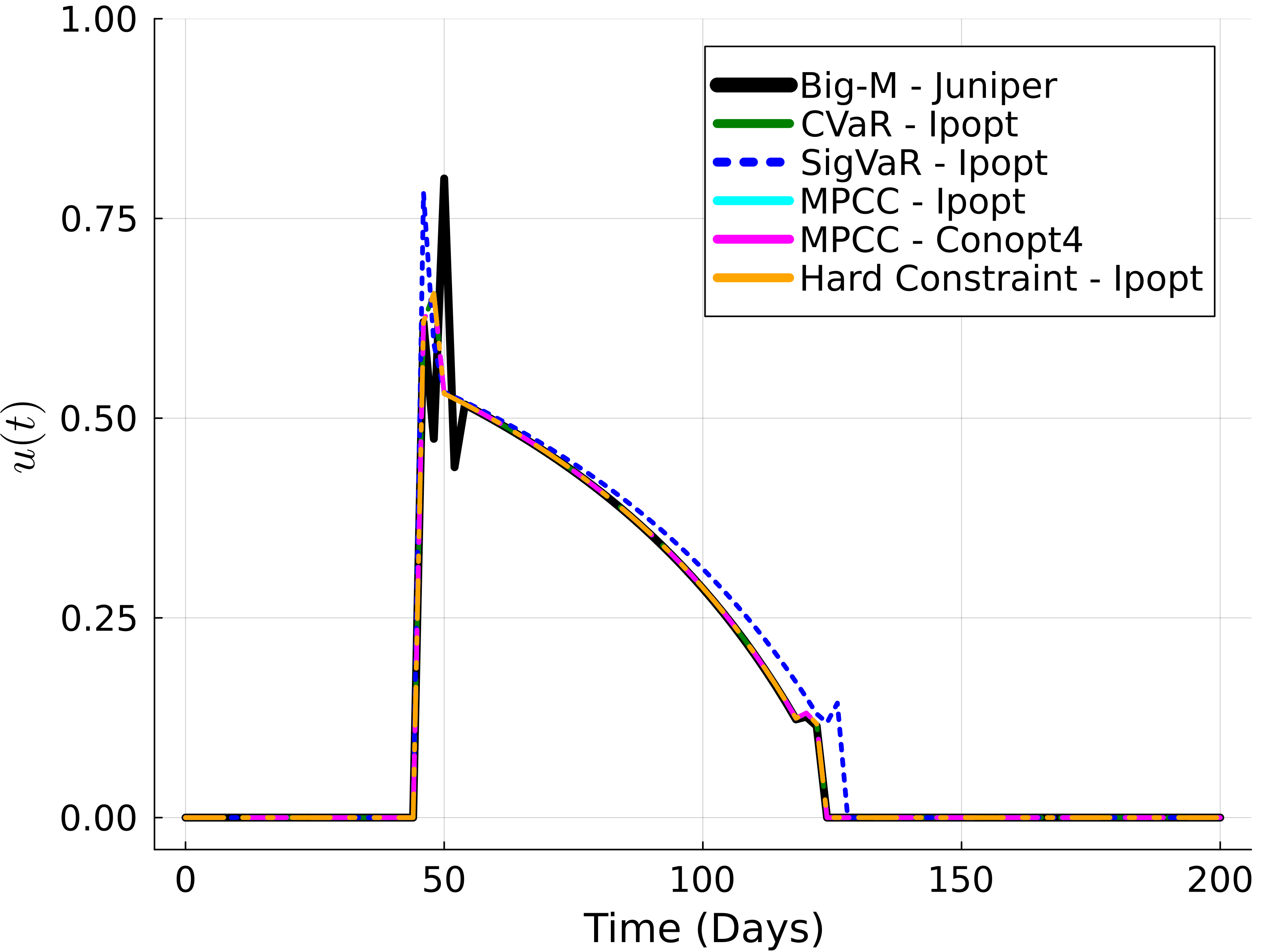}
         \caption{Optimal $u(t)$ profile}
         \label{fig:pandemic_method_comparison_u}
     \end{subfigure}
    \caption{Comparison of the optimal control policy and infected population fraction obtained by different solution methods for $\alpha$=0.85.}
    \label{fig:pandemic_method_comparison}
\end{figure}

Figure \ref{fig:pandemic_method_comparison} provides the optimal trajectories obtained using all the solutions methods with $\alpha=0.85$. Notably, all of these correspond to local optima which is why the MPCC trajectories vary with the choice of solver. It is readily evident that the SigVaR trajectories closely closely approximate those obtained using big-M profiles, making it a promising approach for approximating the solution of event constrained problems at a significantly reduced computational cost relative to big-M.

A detailed analysis of the approximation errors for each solution method (evaluating their performance in approximating both the optimal profiles and the indicator function relative to the big-M solution) is provided in Appendix \ref{app:pandemic_error}. Additionally, the appendix provides insights into how the parameters and quality of the SigVaR approximation evolve with the number of iterations.

\subsection{2D Temperature Control of a Heated Plate} \label{sec:pde_case}
In this case study, we explore the control of a 2D diffusion-based thermal field described by a system of PDEs, similar to the optimal control problem presented in \cite{LU2021146}. The goal is to minimize temperature deviations across a metal plate from a specified setpoint, ensuring no point on the plate exceeds a maximum allowable temperature. The temperature field is influenced by a uniform grid of point heaters, simulating the behavior of a solid metal slab under thermal control \cite{LU2021146}. This study highlights how an event constraint can be applied to spatial constraints in PDE-constrained optimization, and it provides a comparative analysis of the solution techniques proposed in this work.

\subsubsection{Formulation}
We model the two-dimensional heat transfer through a metal plate using the PDE system:
\begin{equation}
    \begin{aligned}
        &&&&& D\left(\frac{\partial^2 T(x)}{\partial x_{1}^2} + \frac{\partial ^2 T(x)}{\partial x_{2}^2}\right) + \sqrt{u(x)} =  q_{loss}, && x \in \mathcal{D}_{x} \\
        &&&&& T[-1, x_{2}] = 0, && x \in \mathcal{D}_x \\
        &&&&& T[1, x_{2}] = 0, && x \in \mathcal{D}_x \\
        &&&&& T[x_{1}, -1] = 0, && x \in \mathcal{D}_x \\
        &&&&& T[x_{1}, 1] = 0, && x \in \mathcal{D}_x \\
    \end{aligned}
    \label{eq:2D_diffusion_ODE}
\end{equation}
where $D=0.05$ is the diffusion coefficient, $q_{loss} = 0.1$ is the unit heat loss rate, and$\mathcal{D}_x = [-1, 1]^2$ is the spatial domain. Here, $u(x)$ denotes the heater inputs and can vary between $0$ and $2500$ at each of the 36 equally spaced heaters, but is constrained to be zero elsewhere. The objective is to minimize the temperature field deviation from a constant setpoint $T_{sp} = 1$:
\begin{equation}
    \min_{T(x)} \quad \int_{x \in \mathcal{D}_x} (T(x)-T_{sp})^2 dx.
    \label{eq:2D_diffusion_objective}
\end{equation}
Moreover, the temperature over the plate is to be kept under $T_{max} = 1.1$:
\begin{equation}
    T(x) \leq T_{max}, \quad x \in \mathcal{D}_{x}
    \label{eq:2D_diffusion_constraint}
\end{equation}
which we will seek to relax via an event constraint in an effort to improve the objective value. The resulting optimal control formulation is solved over a 62x62 grid of uniformly space points leveraging central finite difference.

\subsubsection{Spatial Event Constraint}
We seek to relax Constraint \eqref{eq:2D_diffusion_constraint} via the event constraint:
\begin{equation}
    \mathbb{P}_x(T(x) \leq T_{max}) \geq \alpha
    \label{eq:2D_diffusion_event}
\end{equation}
where we evaluate the underlying expectation with $p(x) = \frac{1}{4}$ (the total plate area described by $\mathcal{D}_x$) such that $\alpha$ represents the fractional plate area for which Constraint \eqref{eq:2D_diffusion_constraint} is enforced. In integral form this becomes:
\begin{equation}
    \mathbb{E}_x\left[\mathbbm{1}_{T(x) \leq T_{max}}(x)\right] = \frac{1}{4}\int_{x \in \mathcal{D}_x}\mathbbm{1}_{T(x) \leq T_{max}}(x) dx \geq \alpha.
\end{equation}
Below, we apply big-M as an exact reformulation which produces a nonconvex MINLP, and compare it against the proposed continuous approximation methods.

\vspace{2mm}
\noindent \underline{Big-M Approximation}
\vspace{2mm}

\noindent The big-M reformulation of the event constraint is:
\begin{equation}
    \begin{aligned}
        &&&&& T(x)-T_{max} \leq M(1 - y(x)), && x \in \mathcal{D}_x \\
        &&&&& \mathbb{E}_x[y(x)] \geq \alpha &&  \\
    \end{aligned}
    \label{eq:2D_diffusion_bigM}
\end{equation}
where $y(x) \in \{0, 1\}$ is the binary variable that indicates constraint violation and $M$ is a sufficiently large upper bounding constant. The resulting MINLP formulation has 19,221 variables (3,844 binary) and 15,378 constraints. BARON is run until a relative gap of at least 0.5\% is achieved since most of the instances are unable to converge with a tighter tolerance within a 15-hour wall-time.

\begin{figure}[!htb]
     \centering
     \begin{subfigure}[b]{0.4\textwidth}
         \centering
         \includegraphics[width=1\textwidth]{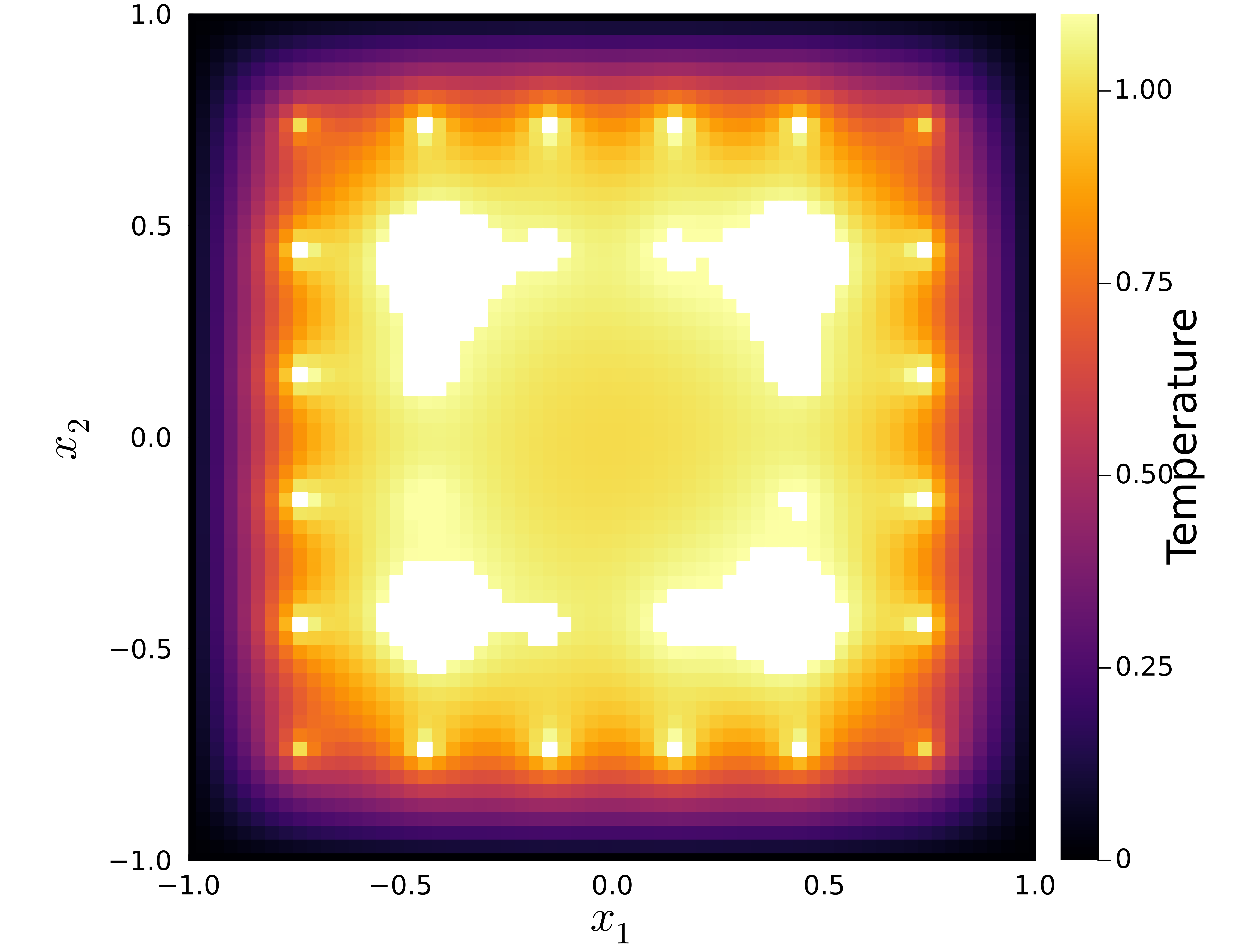}
         \caption{$T(x)$: $\alpha = 0.9$}
         \label{fig:diffusion_NL_0.9_bigM_baron}
     \end{subfigure}
          \hspace{1cm}
     \begin{subfigure}[b]{0.4\textwidth}
         \centering
         \includegraphics[width=1\textwidth]{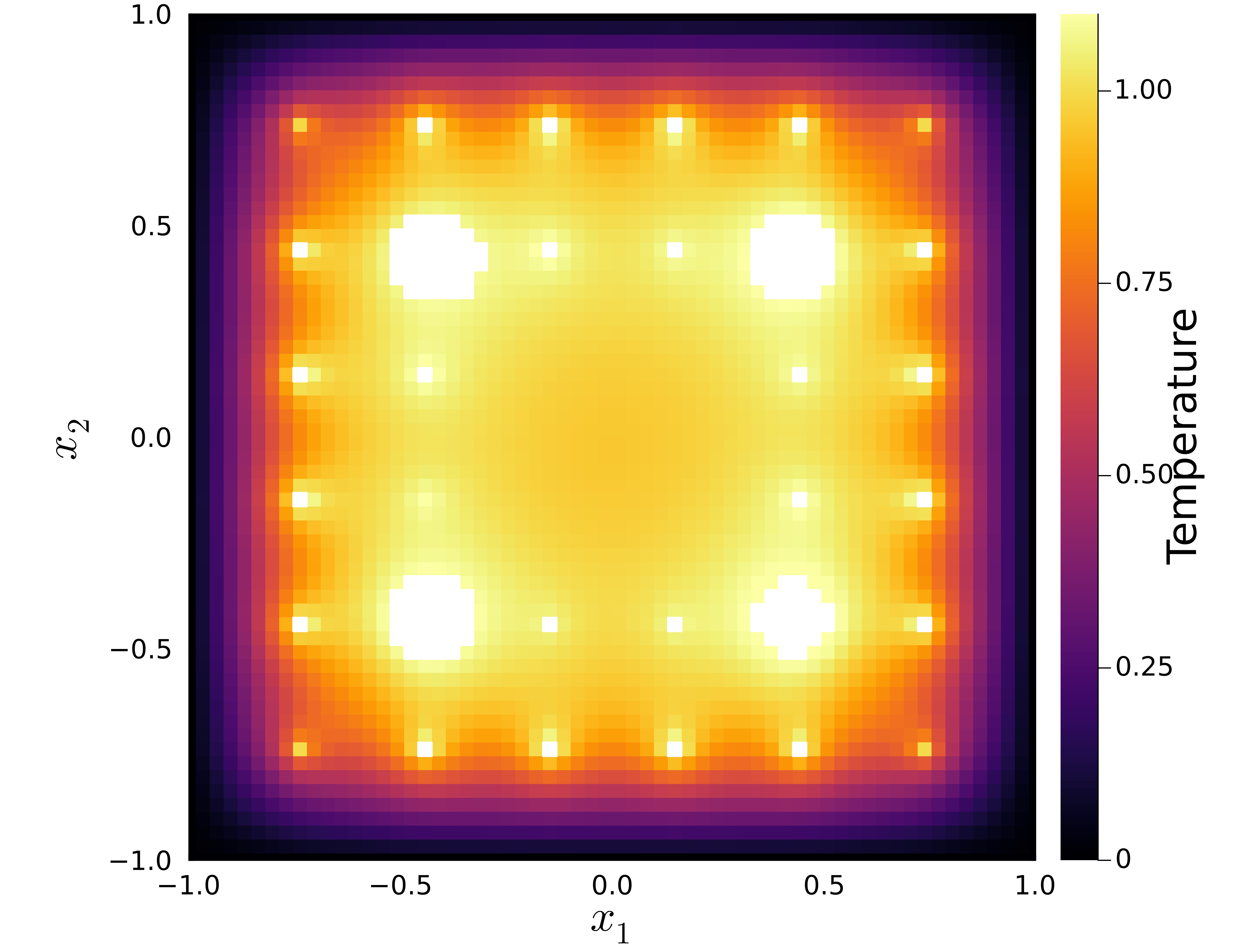}
         \caption{$T(x)$: $\alpha = 0.96$}
         \label{fig:diffusion_NL_0.96_bigM_baron}
     \end{subfigure}
          \hspace{1cm}
     \begin{subfigure}[b]{0.4\textwidth}
         \centering
         \includegraphics[width=1\textwidth]{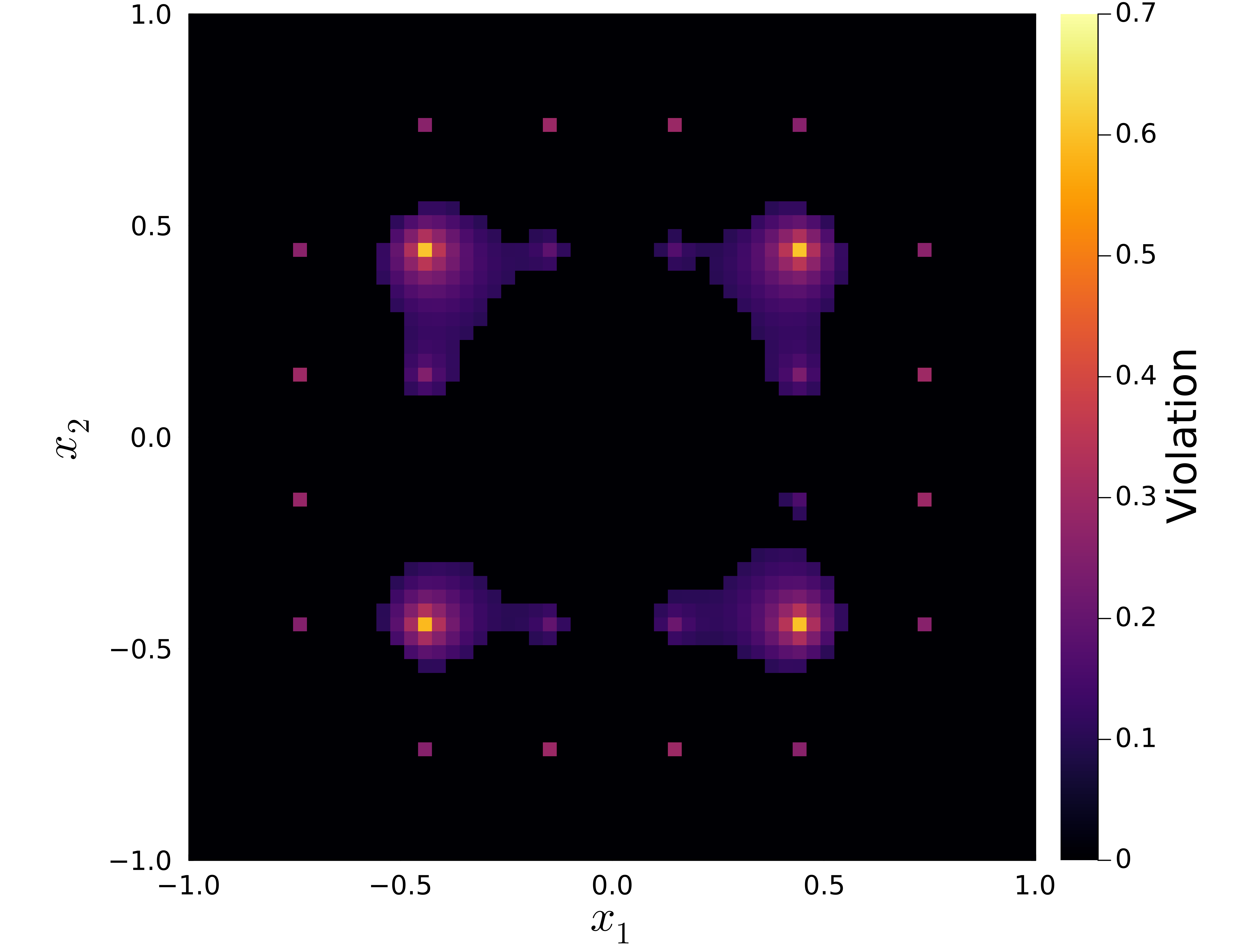}
         \caption{$\max(0, T(x) - T_{max})$: $\alpha = 0.9$}
         \label{fig:diffusion_NL_0.9_bigM_baron_violations}
     \end{subfigure}
          \hspace{1cm}
     \begin{subfigure}[b]{0.4\textwidth}
         \centering
         \includegraphics[width=1\textwidth]{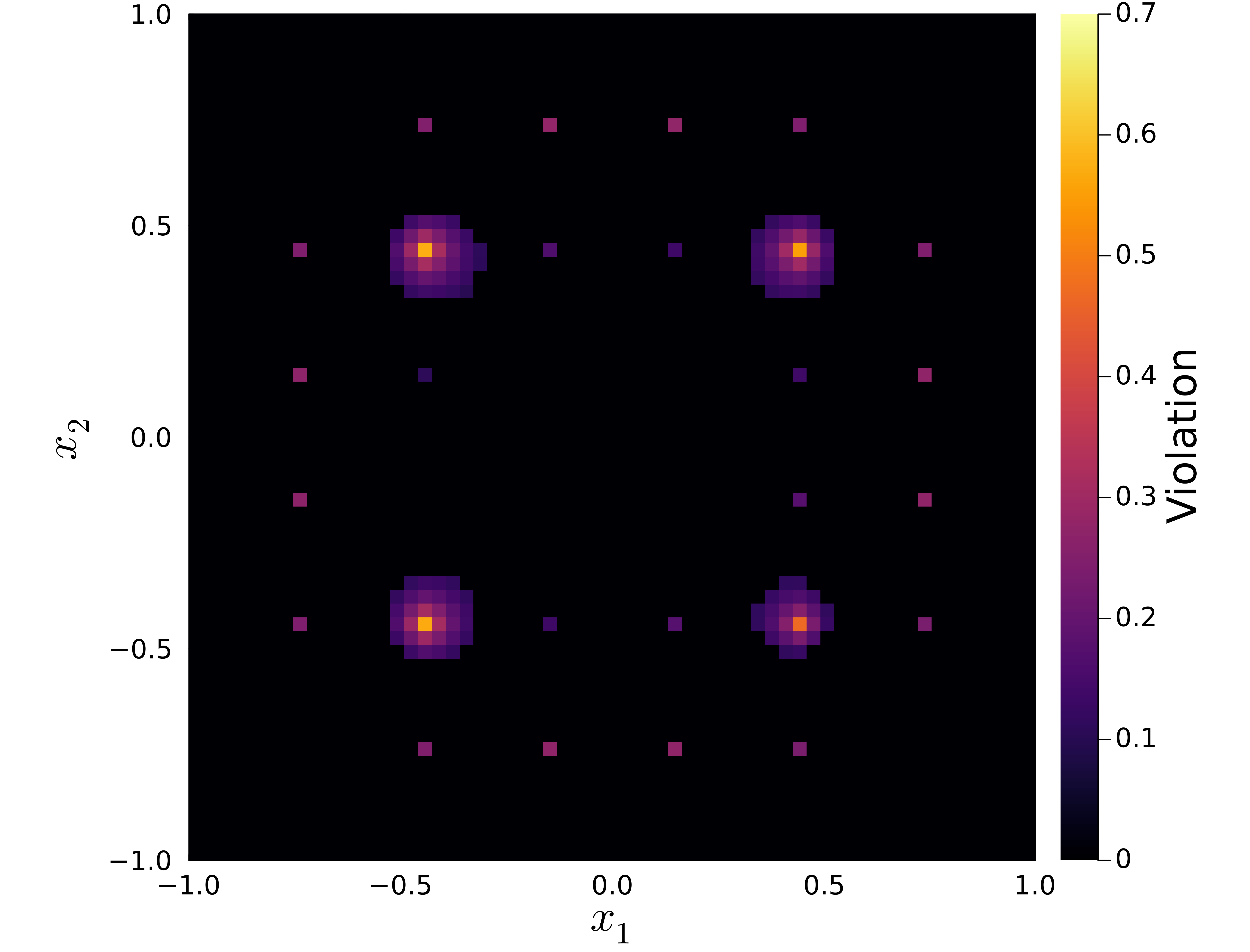}
         \caption{$\max(0, T(x) - T_{max})$: $\alpha = 0.96$}
         \label{fig:diffusion_NL_0.96_bigM_baron_violations}
     \end{subfigure}
    \caption{Temperature and constraint violation distributions for different $\alpha$ values using the big-M reformulation. When plotting $T(x)$, temperatures that exceed $T_{max}$ are displayed in white, regardless of the magnitude. The violation magnitudes are shown in the corresponding violation plots.}
    \label{fig:2D_diffusion_NL_BigM}
\end{figure}

Figure \ref{fig:2D_diffusion_NL_BigM} presents the temperature and constraint violation (i.e., $\max(0, T(x) - T_{max})$ distributions that correspond to the best solutions found for $\alpha \in \{0.90, 0.96\}$. These solutions do show how adjusting $\alpha$ allows us to intuitively control the fraction of the spatial domain that respects Constraint \eqref{eq:2D_diffusion_constraint}. However, the immense computational cost motivates the use of continuous approximations.

\newpage
\noindent \underline{Continuous Approximation via MPCC}
\vspace{2mm}

\noindent The MPCC version of \eqref{eq:2D_diffusion_bigM} is given by:
\begin{equation}
    \begin{aligned}
    &&&&& T(x)-T_{max} \leq M(1 - y^1(x)), && x \in \mathcal{D}_x \\
    &&&&& 0 \leq y^0(x) \perp y^1(x) \geq 0, &&  x \in \mathcal{D}_x\\
    &&&&&  y^0(x) + y^1(x) = 1, && x \in \mathcal{D}_x\\
    &&&&& \mathbb{E}_x[y^1(x)] \geq \alpha &&   \\
    &&&&&  y^0(x), y^1(x) \in [0,1], &&  x \in \mathcal{D}_x.
    \end{aligned}
    \label{eq:2D_diffusion_mpcc}
\end{equation}
This is iteratively solved via the smooth-max method using the $\epsilon$ values listed in Table \ref{tab:tolerances} in Section \ref{app:pandemic}. 
Each of the 40 NLP problems solved has 11,568 variables and 15,376 constraints.
In this case, \textsc{Ipopt} and \textsc{Conopt4} yielded the same solutions. Figure \ref{fig:2D_diffusion_NL_MPCC_CVaR} shows the temperature distributions obtained with \textsc{Ipopt} for $\alpha \in \{0.90, 0.96\}$ for MPCC which are identical to the results obtained by CVaR discussed in the subsection below.

\begin{figure}[!htb]
     \centering
     \begin{subfigure}[b]{0.4\textwidth}
         \centering
         \includegraphics[width=1\textwidth]{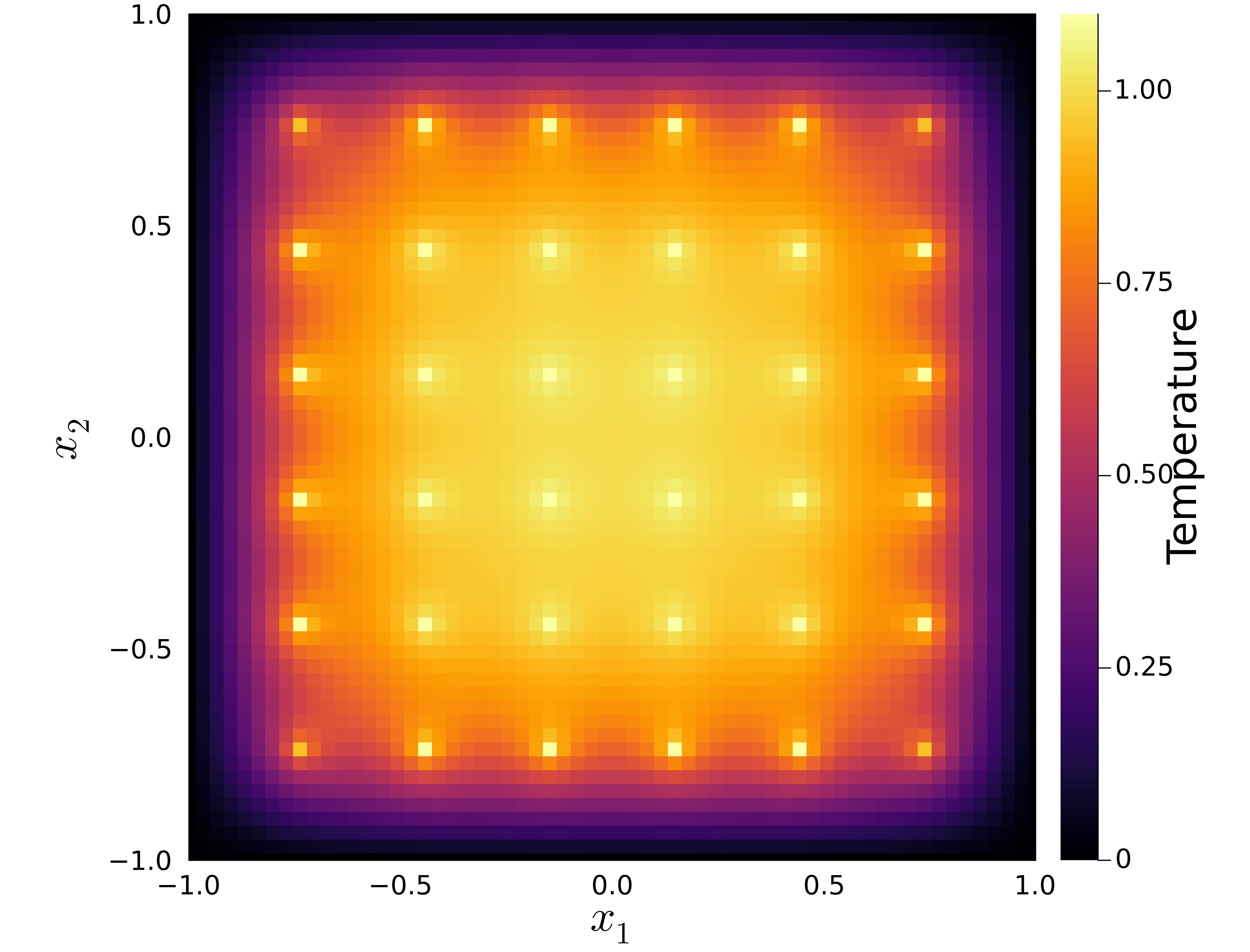}
         \caption{$T(x)$: $\alpha = 0.9$}
         \label{fig:diffusion_NL_0.9_CVaR}
     \end{subfigure}
          \hspace{1cm}
     \begin{subfigure}[b]{0.4\textwidth}
         \centering
         \includegraphics[width=1\textwidth]{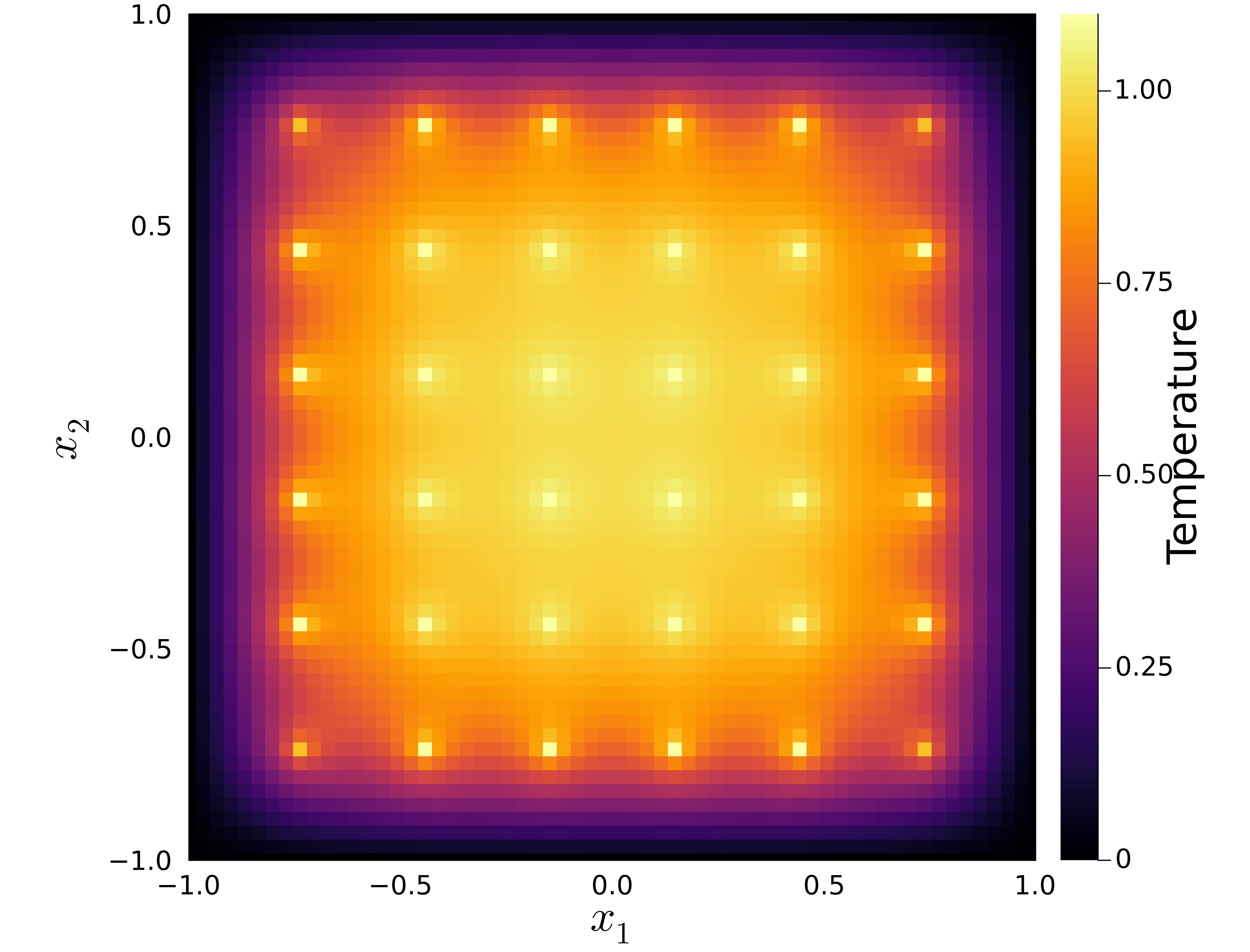}
         \caption{$T(x)$: $\alpha = 0.96$}
         \label{fig:diffusion_NL_0.96_CVaR}
     \end{subfigure}
    \caption{Temperature distributions for different $\alpha$ values using the MPCC and CVaR approximations.}
    \label{fig:2D_diffusion_NL_MPCC_CVaR}
\end{figure}

\vspace{2mm}
\noindent \underline{Continuous Approximation via CVaR}
\vspace{2mm}

\noindent The CVaR approximation for Constraint \eqref{eq:2D_diffusion_event} is:
\begin{equation}
    \begin{aligned}
    &&&&& \phi(x) \geq T(x)-T_{max} - \lambda , && x \in \mathcal{D}_x \\
    &&&&& \mathbb{E}_x[\phi(x)] \leq \lambda(1-\alpha), && x \in \mathcal{D}_x 
    \end{aligned}
    \label{eq:diffusion_cvar}
\end{equation}
where $\phi(x) \in \mathbb{R}_+$ and $\lambda \in \mathbb{R}$ are continuous variables resulting in a model with 15,413 variables and 15,377 constraints. The optimal temperature distributions obtained with $\alpha \in \{0.90, 0.96\}$ via \textsc{Ipopt} are given in Figure \ref{fig:2D_diffusion_NL_MPCC_CVaR}. Note the violation plots are omitted since $T_{max}$ is not violated anywhere on plate. Hence, CVaR again provides an overly conservative solution that effectively enforces Constraint \eqref{eq:2D_diffusion_constraint} as a hard constraint.

\newpage
\noindent \underline{Continuous Approximation via SigVaR}
\vspace{2mm}

\noindent The SigVaR approximation for  \eqref{eq:2D_diffusion_event} is:
\begin{equation}
    \begin{aligned}
    &&&&& \phi(x) \geq 2 \dfrac{1+\beta}{\beta + \exp(-\gamma(T(x)-T_{max}))}  - 1 , && x \in \mathcal{D}_x \\
    &&&&& \mathbb{E}_x[\phi(x)] \leq 1-\alpha, && x \in \mathcal{D}_x 
    \end{aligned}
    \label{eq:diffusion_sigvar}
\end{equation}
where $\phi(x) \in \mathbb{R}_+$ and $\lambda \in \mathbb{R}$ are continuous variables yielding an NLP formulation with 15,413 variables and 15,377 constraints. Also, $\beta$ and $\gamma$ are initialized with $15.50$ and $7.50$, respectively, following the methodology proposed in \cite{cao2020sigmoidal}.

\begin{figure}[!htb]
     \centering
     \begin{subfigure}[b]{0.4\textwidth}
         \centering
         \includegraphics[width=1\textwidth]{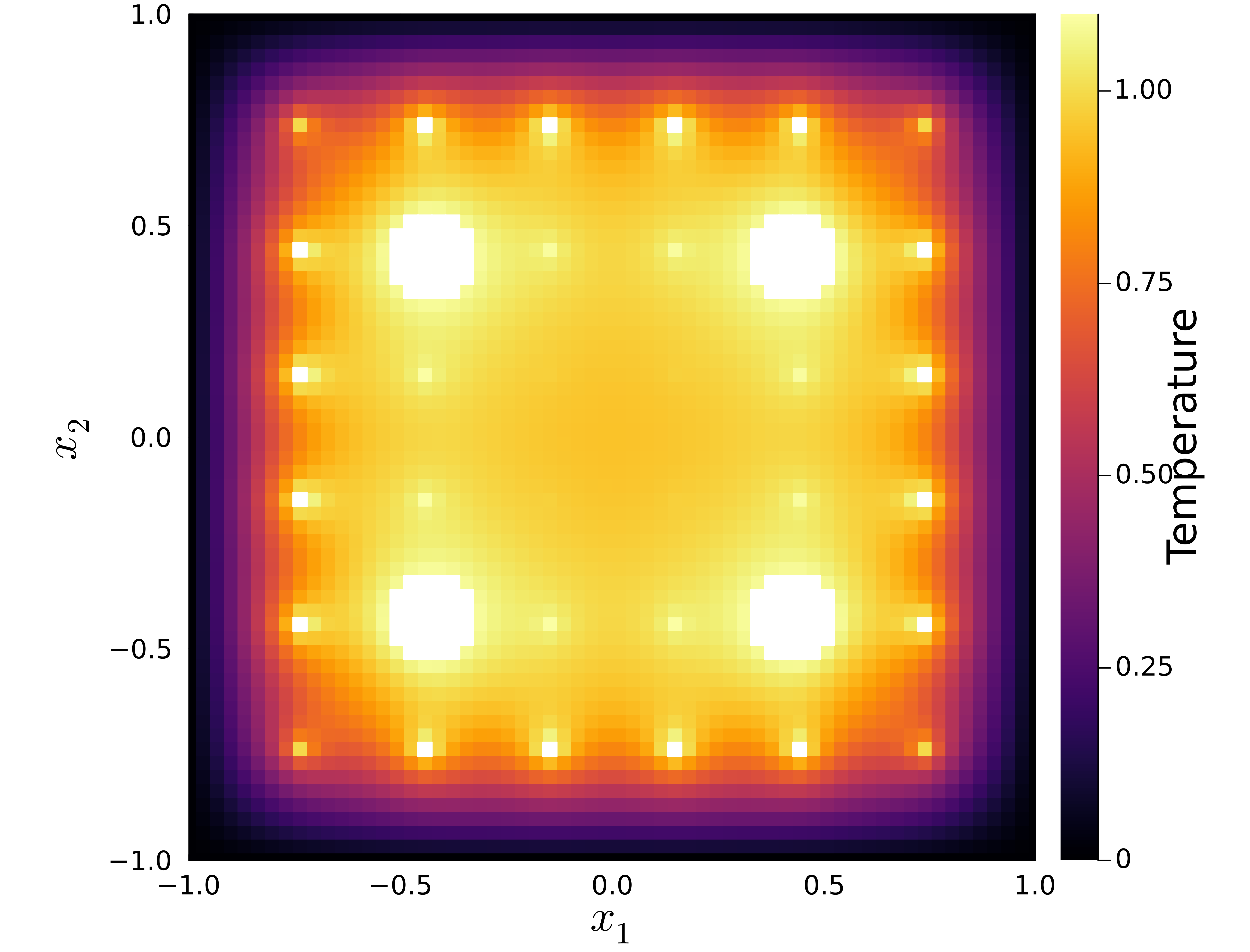}
         \caption{$T(x)$: $\alpha = 0.9$}
         \label{fig:diffusion_NL_0.9_SigVaR}
     \end{subfigure}
          \hspace{1cm}
     \begin{subfigure}[b]{0.4\textwidth}
         \centering
         \includegraphics[width=1\textwidth]{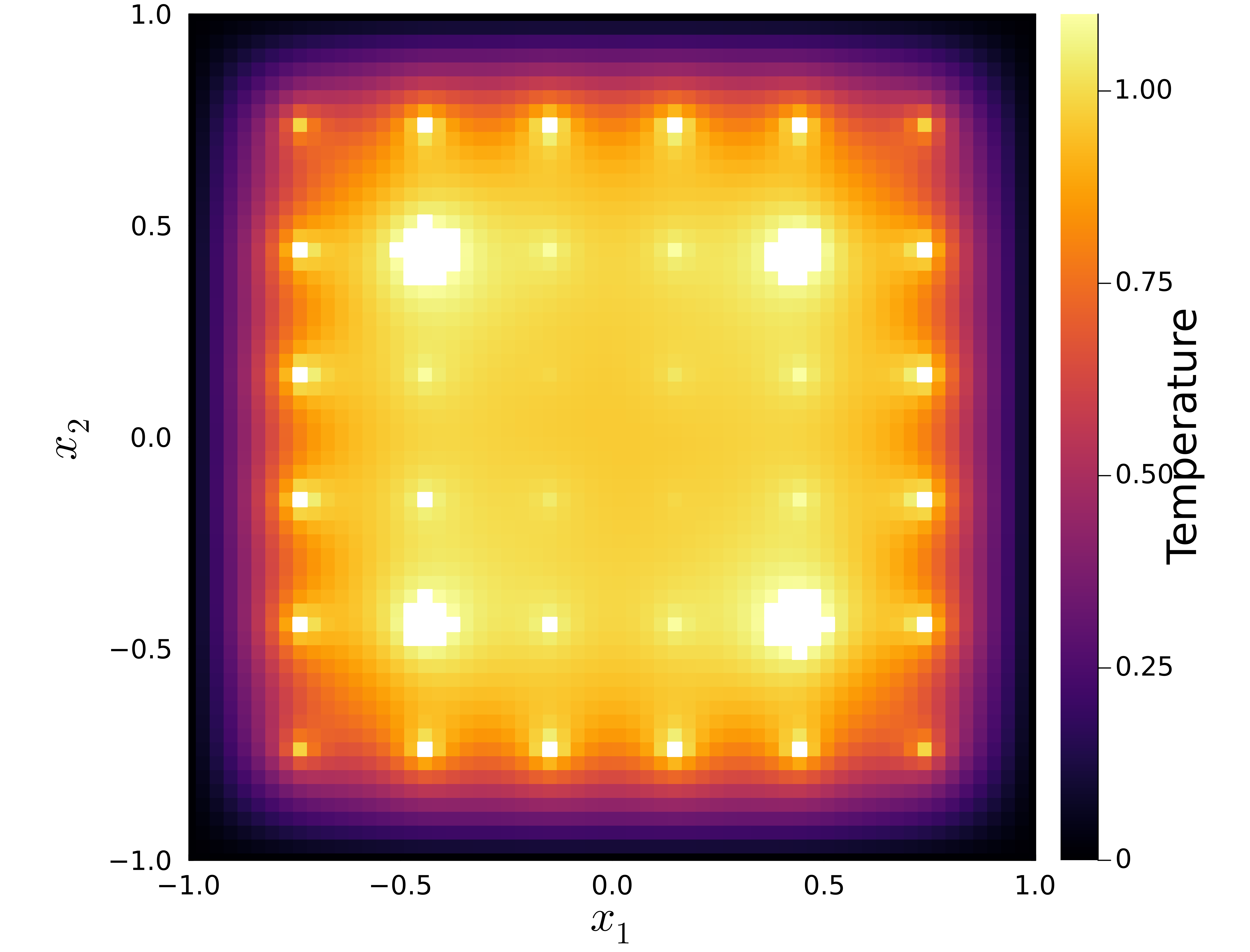}
         \caption{$T(x)$: $\alpha = 0.96$}
         \label{fig:diffusion_NL_0.96_SigVaR}
     \end{subfigure}
          \hspace{1cm}
     \begin{subfigure}[b]{0.4\textwidth}
         \centering
         \includegraphics[width=1\textwidth]{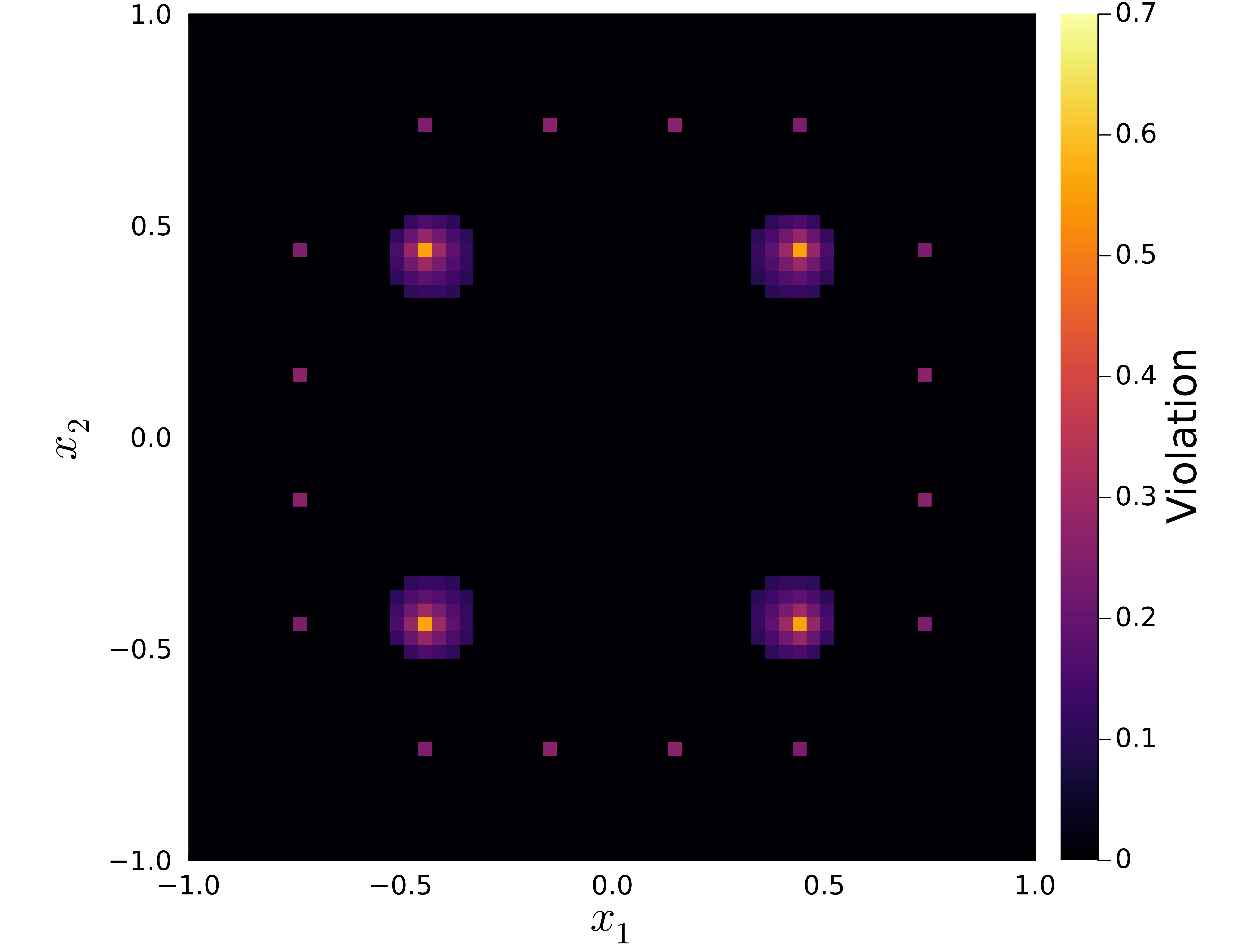}
         \caption{$\max(0, T(x) - T_{max})$: $\alpha = 0.9$}
         \label{fig:diffusion_NL_0.9_SigVaR_violations}
     \end{subfigure}
          \hspace{1cm}
     \begin{subfigure}[b]{0.4\textwidth}
         \centering
         \includegraphics[width=1\textwidth]{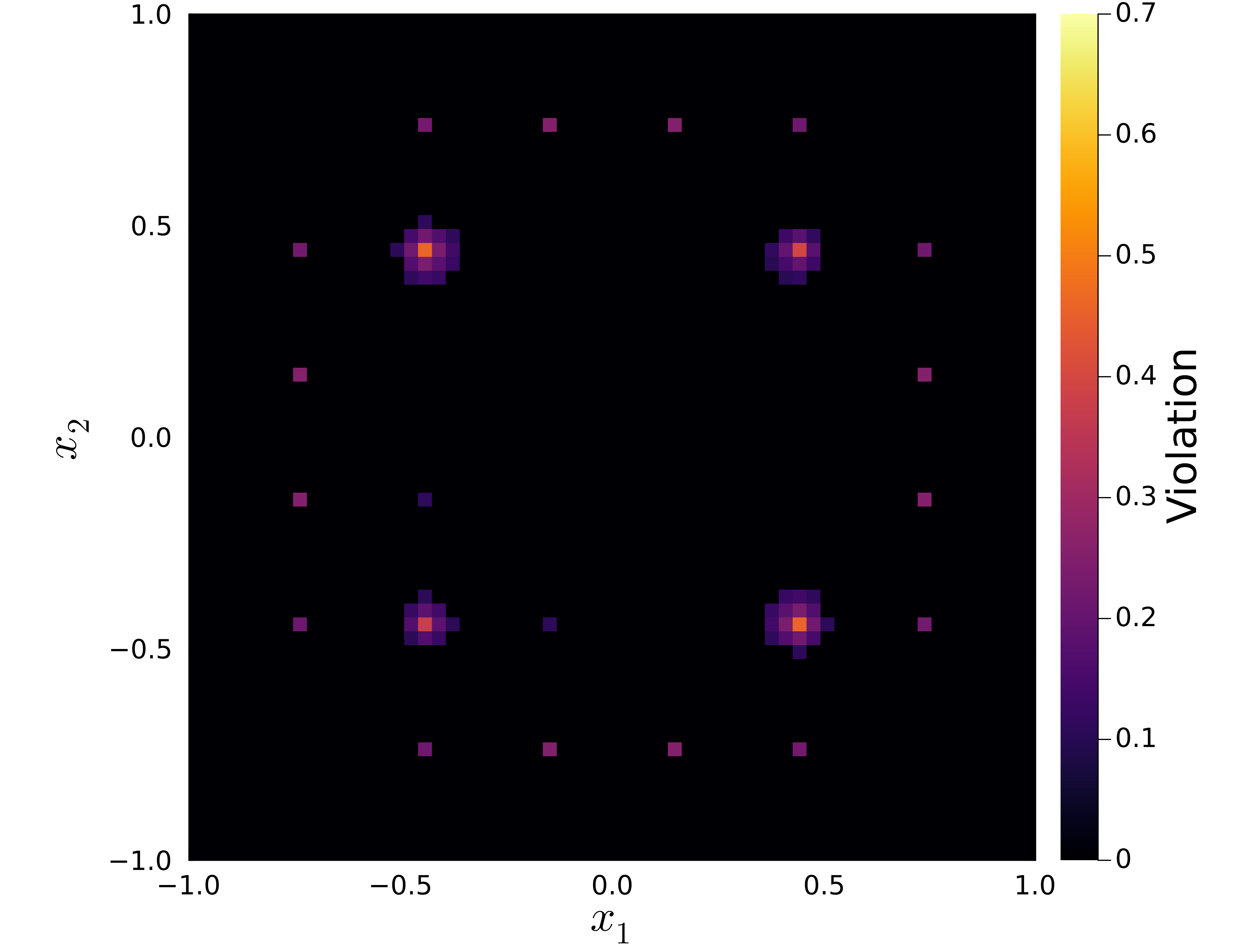}
         \caption{$\max(0, T(x) - T_{max})$: $\alpha = 0.96$}
         \label{fig:diffusion_NL_0.96_SigVaR_violations}
     \end{subfigure}
    \caption{Temperature and constraint violation profiles for different $\alpha$ values using the SigVaR approximation. Temperatures that exceed $T_{max}$ are colored white in the temperature profiles and the magnitude of these violations are shown in the violation plots.}
    \label{fig:2D_diffusion_NL_SigVaR}
\end{figure}

Figure \ref{fig:2D_diffusion_NL_SigVaR} shows the optimal temperature distributions obtained using SigVaR. The SigVaR approximation effectively implements the event constraint, obtaining a high quality solution relative to the best suboptimal solution obtained via big-M. The temperature distributions also exhibit a significantly higher degree of symmetry which aligns with the geometry of the plate. These results also clearly illustrate how $\alpha$ provides an intuitive parameter to adjust fractional area for which Constraint \eqref{eq:2D_diffusion_constraint} is enforced.

\newpage
\noindent \underline{Solution Method Comparison}
\vspace{2mm}

\begin{table}[!htb]
    \centering
    \caption{The optimal objective values and solution times obtained with different solution methods and varied $\alpha$ values for the 2D heated plate study.}
    \label{tab:diffusion_method_comparison_NL}
    \begin{threeparttable}
        \begin{tabular}{c|cc|cc|cc|cc}
        \toprule
        & \multicolumn{2}{c|}{\textbf{Big-M}} & \multicolumn{2}{c|}{\textbf{MPCC}} & \multicolumn{2}{c|}{\textbf{CVaR}} & \multicolumn{2}{c}{\textbf{SigVaR}} \\ 
        \cmidrule{2-9}
        $\alpha$ & Time [s] & Objective & Time [s] & Objective & Time [s] & Objective & Time [s] & Objective \\ \midrule
        $0.90$    & $2,538$ & $0.8334$ & $311.73$  & $0.9465$ & $4.51$  & $0.9465$ & $98.19$ & $0.8467$ \\
        $0.95$    & $2,541$ & $0.8334$ & $340.65$  & $0.9465$ & $2.02$  & $0.9465$ & $107.03$ & $0.8553$ \\
        $0.96$    & $2,808$ & $0.8409$ & $334.48$   & $0.9465$ & $2.84$  & $0.9465$ & $142.56$ & $0.8566$ \\ 
        $0.97$    & $2.808$ & $0.8409$ & $325.61$ & $0.9465$ & $1.87$  & $0.9465$ & $141.40$ & $0.8612$ \\ 
        $0.99$    & $3,252$  & $0.8434$ & $340.37$   & $0.9465$ & $1.96$  & $0.9465$ & $172.31$ & $0.8771$   \\ 
        $0.999$  & $3,191$ & $0.8434$ & $360.30$  & $0.9465$ & $6.77$ & $0.9465$ & $315.72$ & $0.9796$    \\ 
        \bottomrule
        \end{tabular}
    \end{threeparttable}
\end{table}

\noindent Table \ref{tab:diffusion_method_comparison_NL} summarizes the numerical results in terms of solution times and optimal objective values across a range of $\alpha$ values. For the big-M method, the table reports the time that BARON required to obtain a relative optimality gap of 0.5\%. The other methods are solved to local optimality using \textsc{Ipopt} using the default tolerance (i.e., $10^{-8}$). The MINLP formed via big-M incurs a significant computational burden, but is able to find the best objective values reported in this study since the event constraint is directly enforced. Certain solutions are repeated since the change in $\alpha$ is smaller than the optimality gap threshold of 0.5\%. In contrast, MPCC and CVaR dramatically reduce solution time, but both consistently converge to the same overly conservative solution, effectively ignoring our choice of $\alpha$. Though the MPCC solution times are greater since it solves multiple NLPs unlike CVaR. The SigVaR approximation produces solutions comparable to the big-M method while decreasing the computational cost by approximately a factor of 20. This further demonstrates how SigVaR is an efficient and practical approach for solving event-constrained programs. One caveat to this observation is that the solution quality for SigVaR degrades in this case for $\alpha = 0.999$ which can be attributed to poor conditioning in the Jacobian due to the scaling of the discretized PDEs and the transcribed integrals in the event constraint.

\begin{figure}[!htb]
     \centering
     \begin{subfigure}[b]{0.4\textwidth}
         \centering
         \includegraphics[width=1\textwidth]{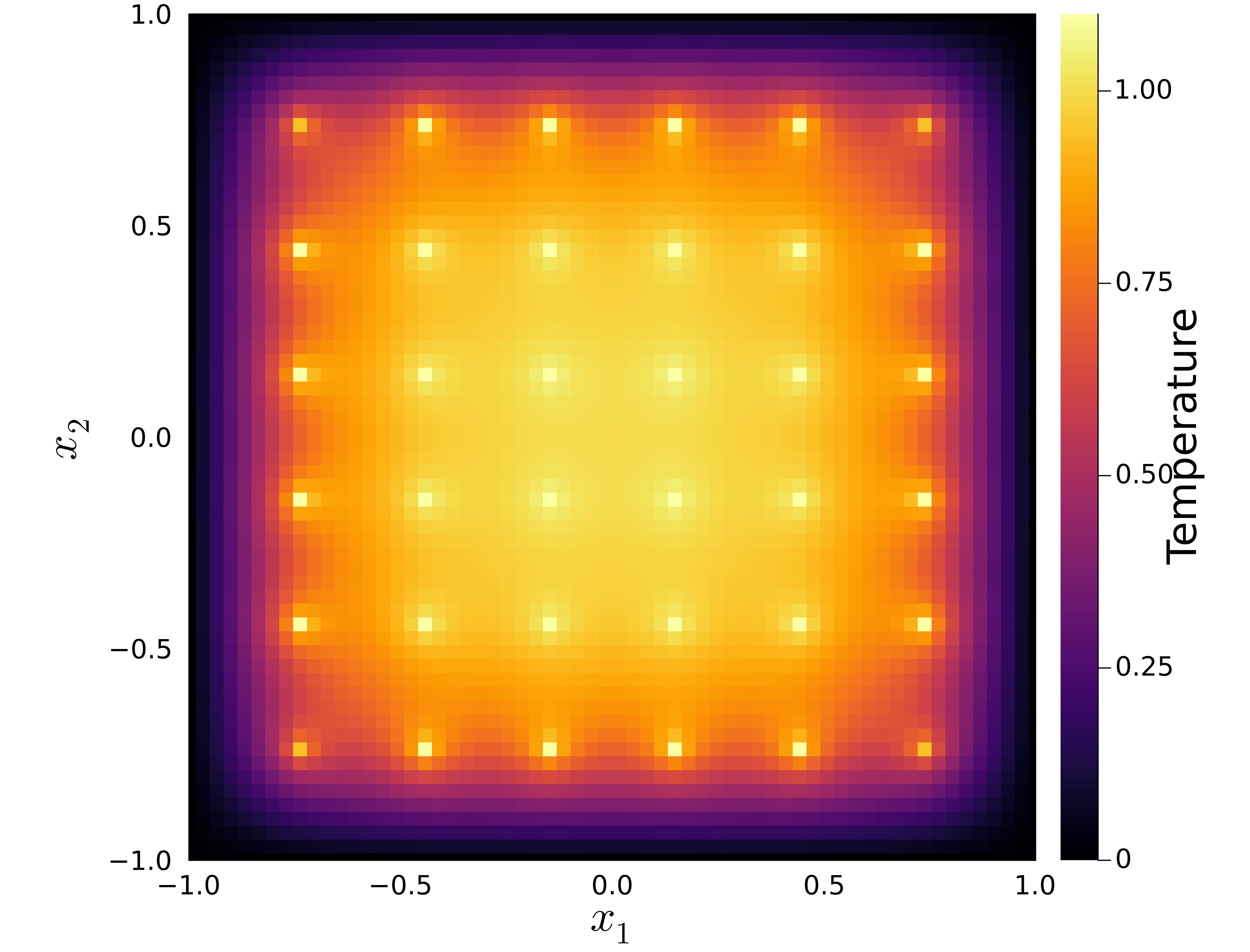}
         \caption{Hard Constraint}
         \label{fig:diffusion_NL_0.95_hard}
     \end{subfigure}
          \hspace{1cm}
     \begin{subfigure}[b]{0.4\textwidth}
         \centering
         \includegraphics[width=1\textwidth]{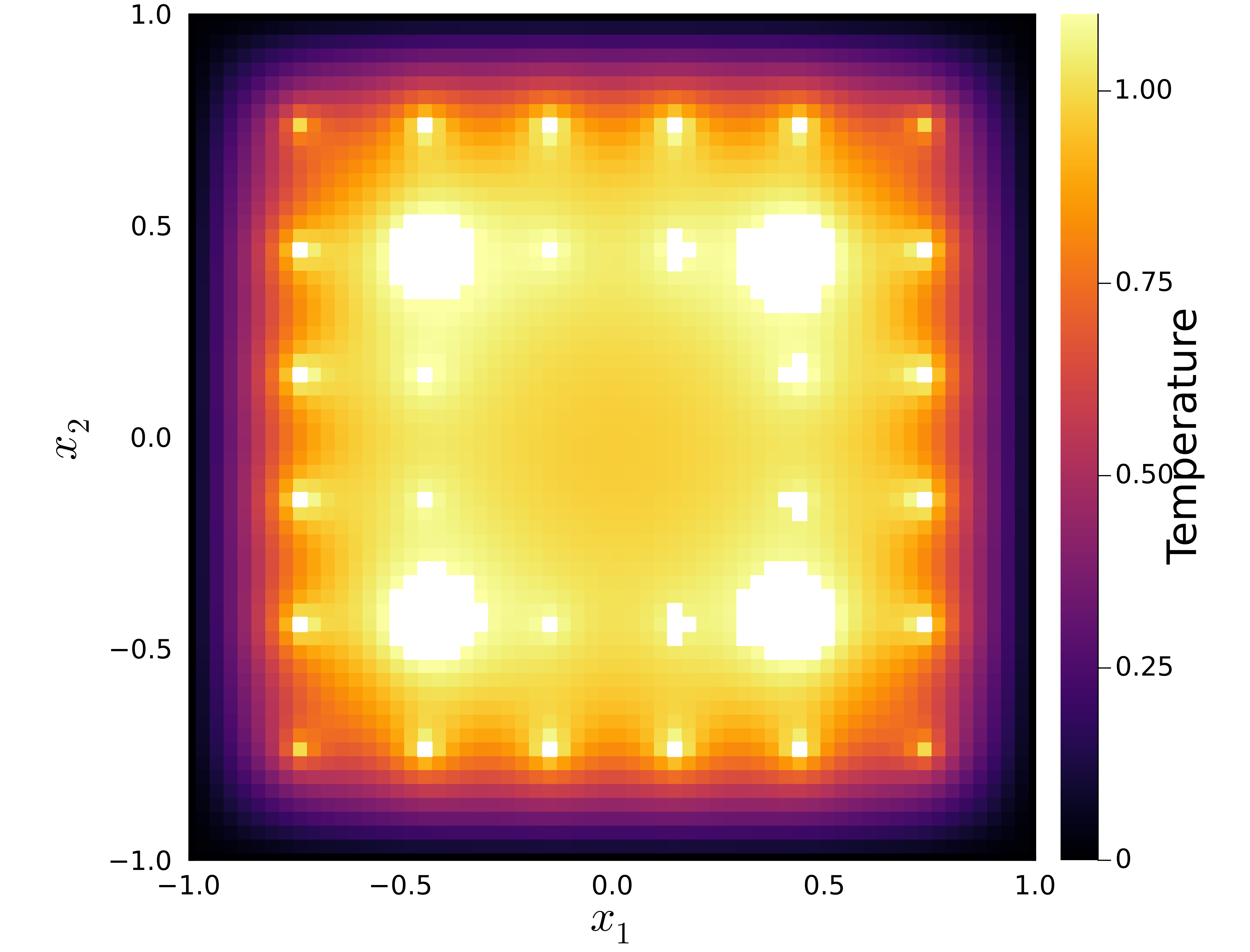}
         \caption{Big-M: $\alpha = 0.95$}
         \label{fig:diffusion_NL_0.95_bigM_baron}
     \end{subfigure}
          \hspace{1cm}
     \begin{subfigure}[b]{0.4\textwidth}
         \centering
         \includegraphics[width=1\textwidth]{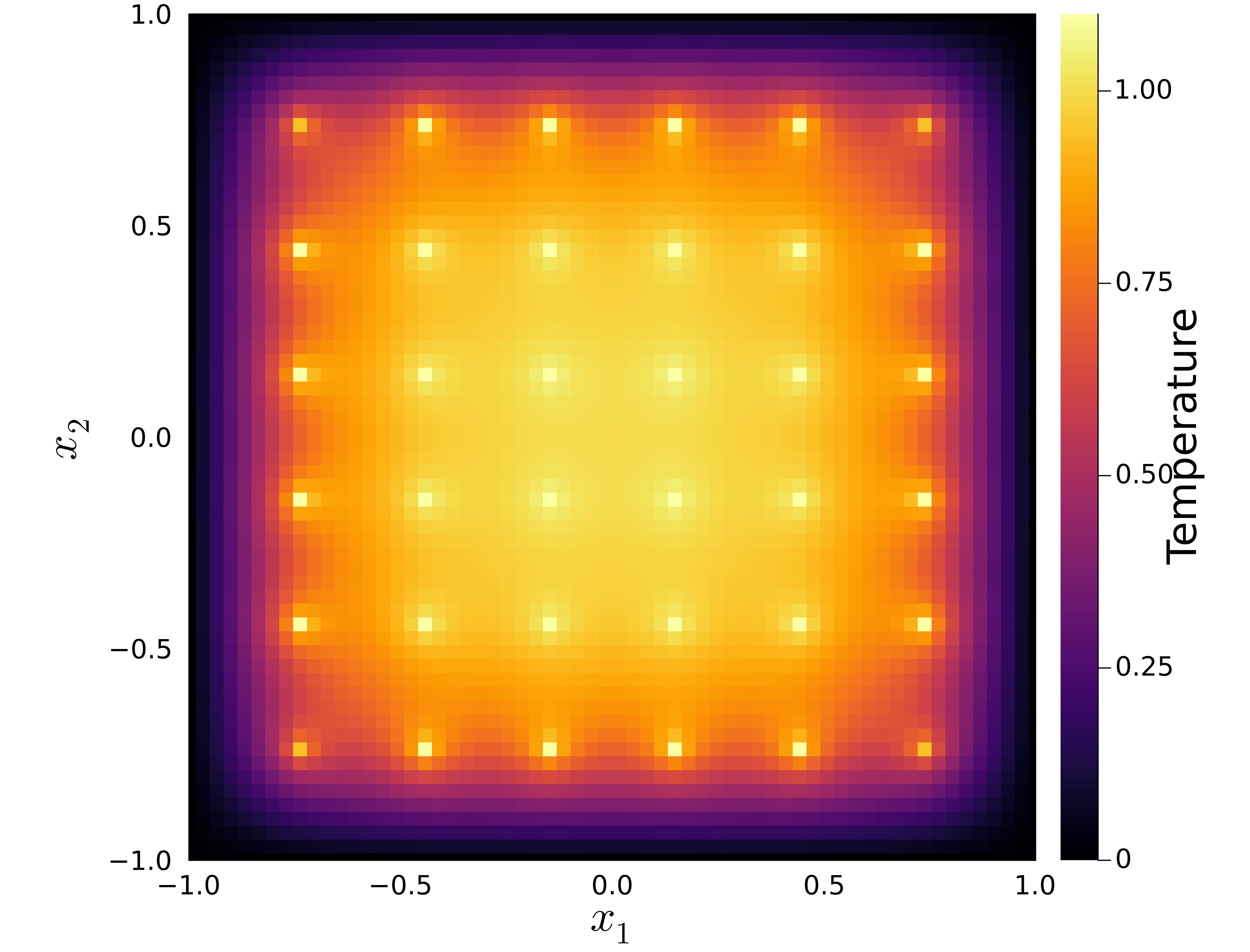}
         \caption{CVaR/MPCC: $\alpha = 0.95$}
         \label{fig:diffusion_NL_0.95_CVaR/MPCC}
     \end{subfigure}
          \hspace{1cm}
     \begin{subfigure}[b]{0.4\textwidth}
         \centering
         \includegraphics[width=1\textwidth]{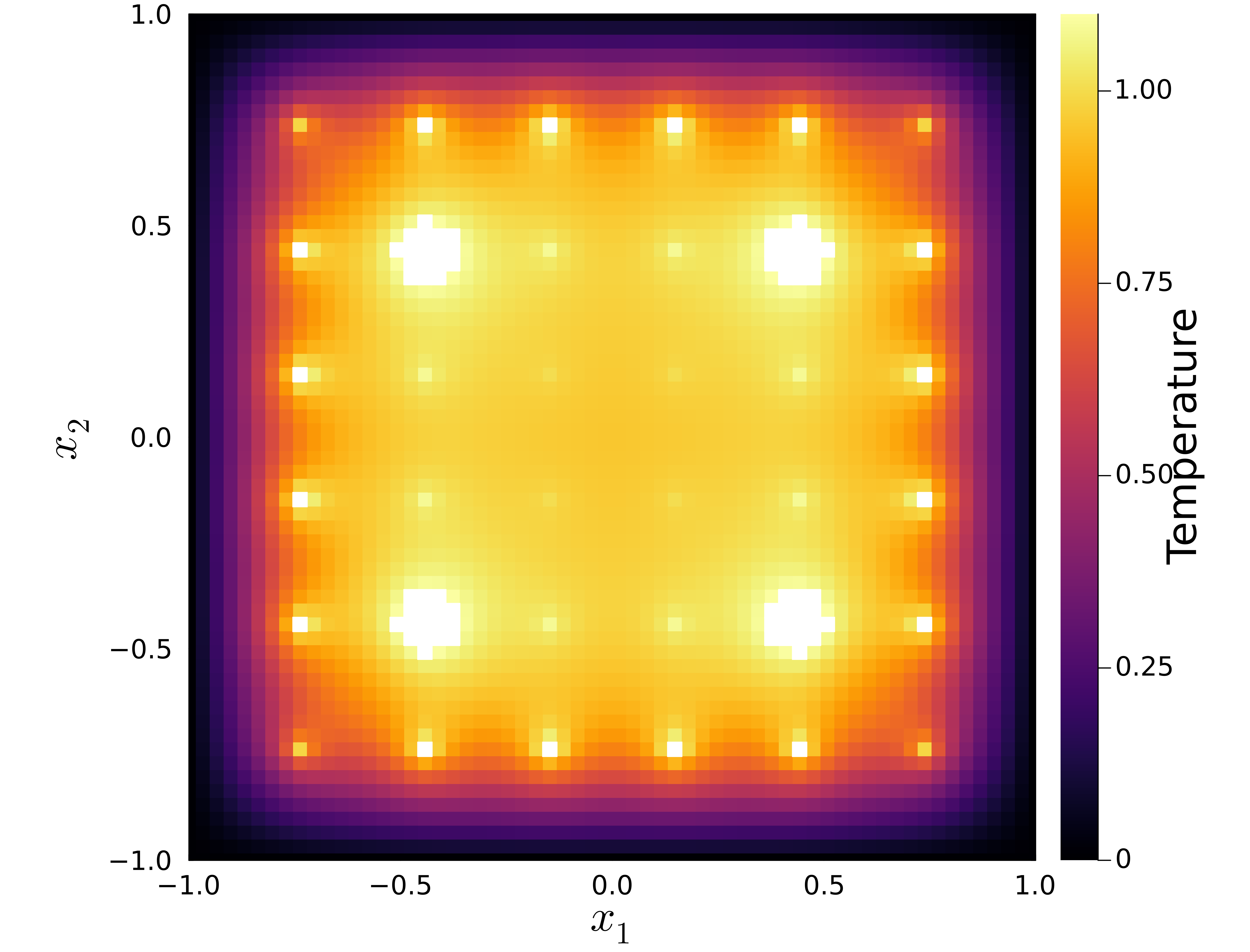}
         \caption{SigVaR: $\alpha = 0.95$}
         \label{fig:diffusion_NL_0.95_SigVaR}
     \end{subfigure}
    \caption{Comparison of the temperature distributions obtained by different solution methods for $\alpha$=0.95. Any temperatures exceeding $T_{max}$ are shown in white.}
    \label{fig:2D_diffusion_NL_methods_0.95}
\end{figure}

Figure \ref{fig:2D_diffusion_NL_methods_0.95} summarizes all the optimal temperature distributions obtained across the methods for $\alpha = 0.95$. Again, it is clear that both CVaR and MPCC approximations give solutions that are identical to enforcing Constraint \eqref{eq:2D_diffusion_constraint} as a hard constraint. Additionally, when comparing the big-M and SigVaR methods, both produce symmetric solutions consistent with the plate geometry and big-M is able to have constraint violations over a greater fractional area (nearly equal to $1-\alpha$). Finally, an analytical comparison of the ISE on both the temperature distribution and the indicator function is computed and shown in Tables \ref{tab:diffusion_NL_T(x)_ISE_app} and \ref{tab:diffusion_NL_indicator(x)_ISE_app} in Section \ref{app:diffusion} of the Appendix.

\FloatBarrier

\section{Conclusion}\label{sec:conclusion}
We have presented a framework and solution techniques for generalized analogs of chance constraints from stochastic programming to a general context for InfiniteOpt problems in a modeling paradigm that we call event constrained programming. This framework generalizes both the aggregation logic and the domain of traditional chance constraints, establishing them as a versatile modeling element. We proposed several solution methods for event-constrained programs based on techniques from chance constraint literature, GDP, and MPCC. Through three case studies, we demonstrated the applicability of event constraints to stochastic domains with complex logic, as well as to other continuous domains such as time and space. These studies highlight how methods originally developed for chance constraints can be effectively adapted to efficiently solve event constraints across general InfiniteOpt problems.

\section*{Acknowledgments}
We thank the Center for Advanced Process Decision-making and the University of Waterloo for supporting this work.

\newpage
\bibliography{references}

\appendix
\newpage
\section{Appendix}\label{sec:appendix}

\subsection{Power Grid Design Supplementary Material}\label{app:power_grid}

Tables \ref{tab:grid_mean} and \ref{tab:grid_bounds} provide supplementary parameter details for the the power grid design formulation described in Section \ref{sec:grid_case}.

\begin{table}[!htb]
    \centering
    \caption{Mean values for the stochastic demand $\xi \sim \mathcal{N}(\mu, \Sigma$).}
    \label{tab:grid_mean}
    \begin{tabular}{cccccccccccc}
        \toprule
        \textbf{Demand node} & $d_1$ & $d_2$  & $d_3$  & $d_4$  & $d_5$  & $d_6$  & $d_7$  & $d_8$  & $d_9$  & $d_{10}$  & $d_{11}$  \\
        \midrule
        $\mu$ & $87.3$ & $50.0$ & $25.0$ & $28.8$ & $50.0$ & $25.0$ & $0.0$ & $0.0$ & $0.0$ & $0.0$ & $0.0$ \\
        \bottomrule
    \end{tabular}
\end{table}

\begin{table}[!thb]
    \centering
    \caption{Generator safety operating thresholds $\bar{\bar{q}}_{g,i}, i \in \mathcal{G}$.}
    \label{tab:grid_bounds}
    \begin{tabular}{ccccccccccc}
        \toprule
        \textbf{Generator} & $g_1$ & $g_2$  & $g_3$  & $g_4$  & $g_5$  \\
        \midrule
        $\bar{\bar{q}}_{g,i}$ & $332$ & $140$ & $100$ & $100$ & $100$ \\
        \bottomrule
    \end{tabular}
\end{table}

\subsection{Optimal Disease Control Supplementary Material}\label{app:pandemic}

Figure \eqref{fig:pandemic_hardconstraint} juxtaposes the optimal trajectories obtained solving the optimal disease control problem with and without Constraint \eqref{eq:pandemic_constraint}, constituting the two extreme cases that we can interpolate between using an event constraint.

\begin{figure}[!htb]
     \centering
     \begin{subfigure}[b]{0.4\textwidth}
         \centering
         \includegraphics[width=1\textwidth]{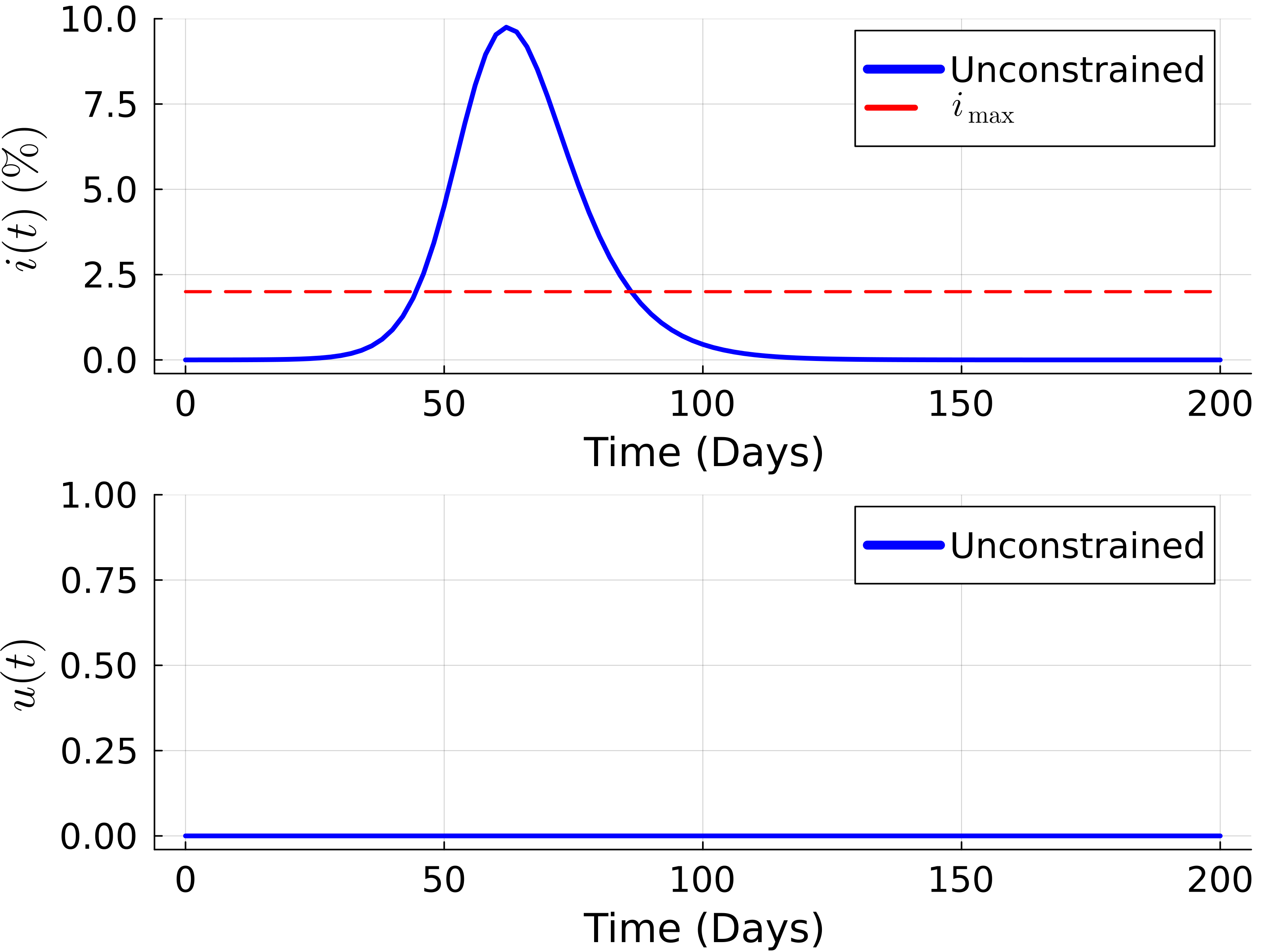}
         \caption{Unconstrained case}
         \label{fig:pandemic_unconstrained}
     \end{subfigure}
     \hspace{1cm}
     \begin{subfigure}[b]{0.4\textwidth}
         \centering
         \includegraphics[width=1\textwidth]{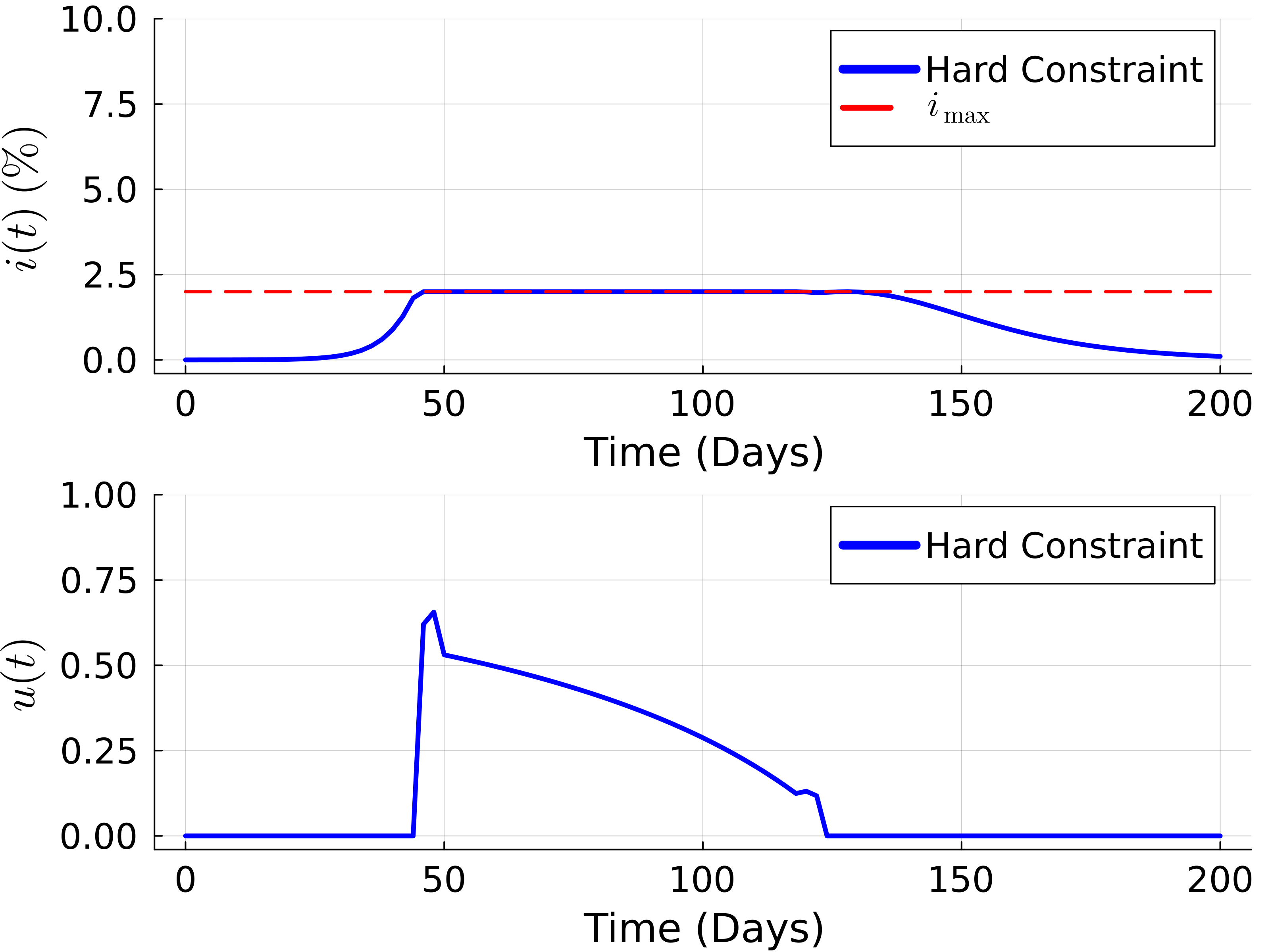}
         \caption{Constrained case}
         \label{fig:pandemic_constrained}
     \end{subfigure}
    \caption{Optimal trajectories for $i(t)$ and $u(t)$ for with and without Constraint \eqref{eq:pandemic_constraint}.}
    \label{fig:pandemic_hardconstraint}
\end{figure}

Figure \ref{fig:pandemic_bigM_app} shows the optimal trajectories for $\alpha = \{0.85, 0.9, 0.96, 0.99\}$ using the big-M reformulation of the event constraint.
\begin{figure}[!htb]
     \centering
     \begin{subfigure}[b]{0.4\textwidth}
         \centering
         \includegraphics[width=1\textwidth]{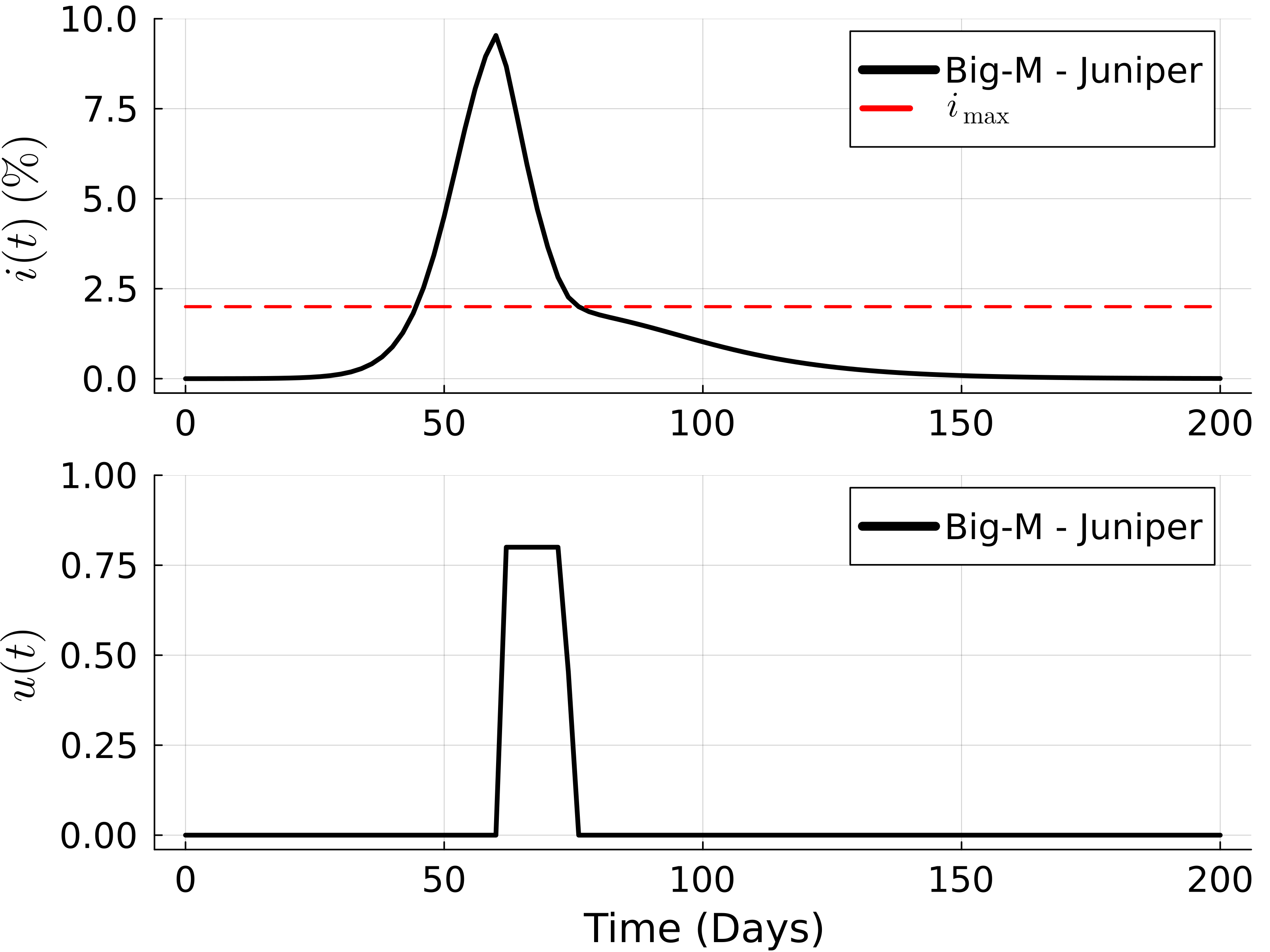}
         \caption{Big-M: $\alpha = 0.85$}
         \label{fig:pandemic_bigM_0.85_app}
     \end{subfigure}
     \hspace{1cm}
     \begin{subfigure}[b]{0.4\textwidth}
         \centering
         \includegraphics[width=1\textwidth]{figures/Pandemic_Control_Case_Study/bigM_0.9.png}
         \caption{Big-M: $\alpha = 0.9$}
         \label{fig:pandemic_bigM_0.9_app}
     \end{subfigure}
          \hspace{1cm}
     \begin{subfigure}[b]{0.4\textwidth}
         \centering
         \includegraphics[width=1\textwidth]{figures/Pandemic_Control_Case_Study/bigM_0.96.png}
         \caption{Big-M: $\alpha = 0.96$}
         \label{fig:pandemic_bigM_0.96_app}
     \end{subfigure}
          \hspace{1cm}
     \begin{subfigure}[b]{0.4\textwidth}
         \centering
         \includegraphics[width=1\textwidth]{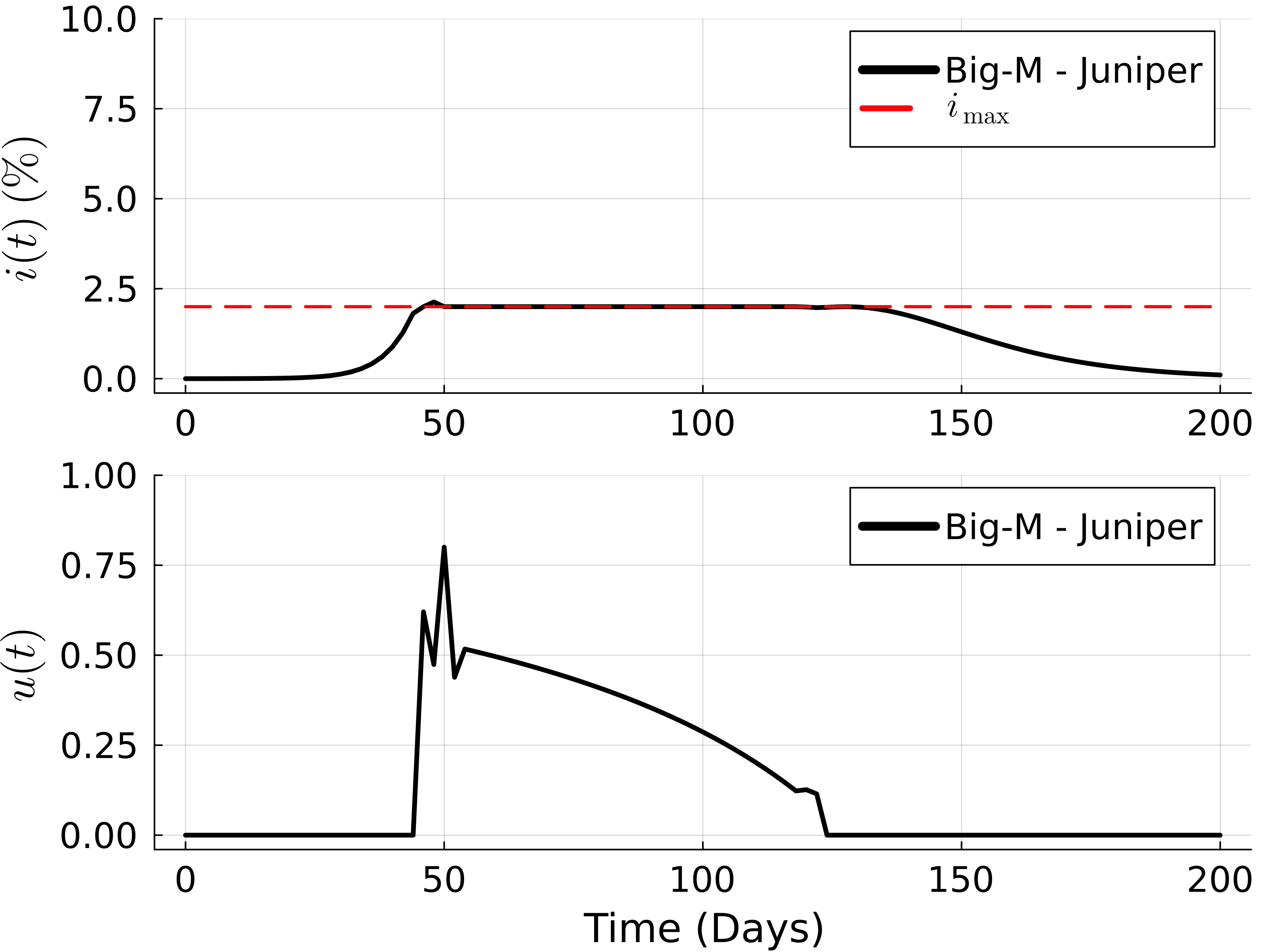}
         \caption{Big-M: $\alpha = 0.99$}
         \label{fig:pandemic_bigM_0.99_app}
     \end{subfigure}
    \caption{Optimal $i(t)$ and $u(t)$ trajectories for a representative range of $\alpha$ values using the big-M reformulation.}
    \label{fig:pandemic_bigM_app}
\end{figure}

\FloatBarrier

Table \ref{tab:tolerances} details the values of $\epsilon$ used in each iteration of the smooth-max approximation of the MPCC formulation.

\begin{table}[!htb]
    \centering
    \small 
    \setlength{\tabcolsep}{3pt} 
    \caption{Tolerance of the relaxation for MPCC}
    \label{tab:tolerances}
    \begin{tabular}{ccccccccccccccc}
        \toprule
        \textbf{Iter.} & 1 & 2 & 3 & 4 & 5 & 6 & 7 & 8 & 9 & 10 & 11 & 12 & 13 \\
        \midrule
        \textbf{$\epsilon$} & $1.00$ & $0.955$ & $0.91$ & $0.865$ & $0.82$ & $0.775$ & $0.73$ & $0.685$ & $0.64$ & $0.595$ & $0.55$ & $0.505$ & $0.46$ \\
        \midrule
        \textbf{Iter.} & 14 & 15 & 16 & 17 & 18 & 19 & 20 & 21 & 22 & 23 & 24 & 25 & 26 \\
        \midrule
        \textbf{$\epsilon$} & $0.415$ & $0.37$ & $0.325$ & $0.28$ & $0.235$ & $0.19$ & $0.145$ & $0.10$ & $4.46e-2$ & $1.99e-2$ & $8.91e-3$ & $3.98e-3$ & $1.77e-3$ \\
        \midrule
        \textbf{Iter.} & 27 & 28 & 29 & 30 & 31 & 32 & 33 & 34 & 35 & 36 & 37 & 38 & 39 \\
        \midrule
        \textbf{$\epsilon$} & $7.94e-4$ & $3.54e-4$ & $1.58e-4$ & $7.07e-5$ & $3.16e-5$ & $1.41e-5$ & $6.30e-6$ & $2.81e-6$ & $1.25e-6$ & $5.62e-7$ & $2.51e-7$ & $1.12e-7$ & $5.01e-8$ \\
        \bottomrule
    \end{tabular}
\end{table}

Figure \ref{fig:pandemic_MPCC_app} shows additional optimal trajectories obtained using the MPCC approximation.

\begin{figure}[!htb]
     \centering
     \begin{subfigure}[b]{0.4\textwidth}
         \centering
         \includegraphics[width=1\textwidth]{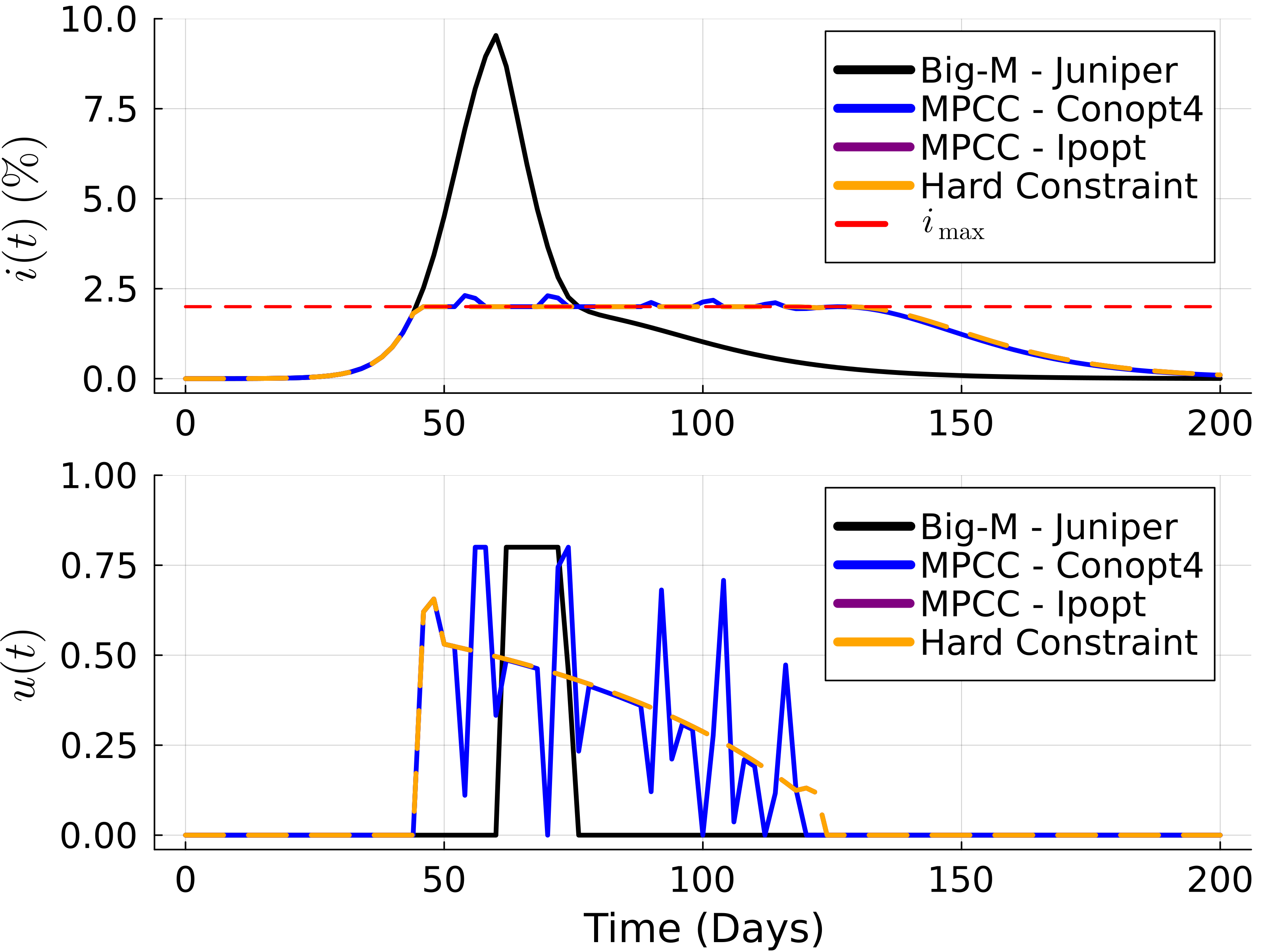}
         \caption{MPCC: $\alpha = 0.85$}
         \label{fig:pandemic_MPCC_0.85_app}
     \end{subfigure}
     \hspace{1cm}
     \begin{subfigure}[b]{0.4\textwidth}
         \centering
         \includegraphics[width=1\textwidth]{figures/Pandemic_Control_Case_Study/comp_plot_0.9.png}
         \caption{MPCC: $\alpha = 0.9$}
         \label{fig:pandemic_MPCC_0.9_app}
     \end{subfigure}
          \hspace{1cm}
     \begin{subfigure}[b]{0.4\textwidth}
         \centering
         \includegraphics[width=1\textwidth]{figures/Pandemic_Control_Case_Study/comp_plot_0.96.png}
         \caption{MPCC: $\alpha = 0.96$}
         \label{fig:pandemic_MPCC_0.96_app}
     \end{subfigure}
          \hspace{1cm}
     \begin{subfigure}[b]{0.4\textwidth}
         \centering
         \includegraphics[width=1\textwidth]{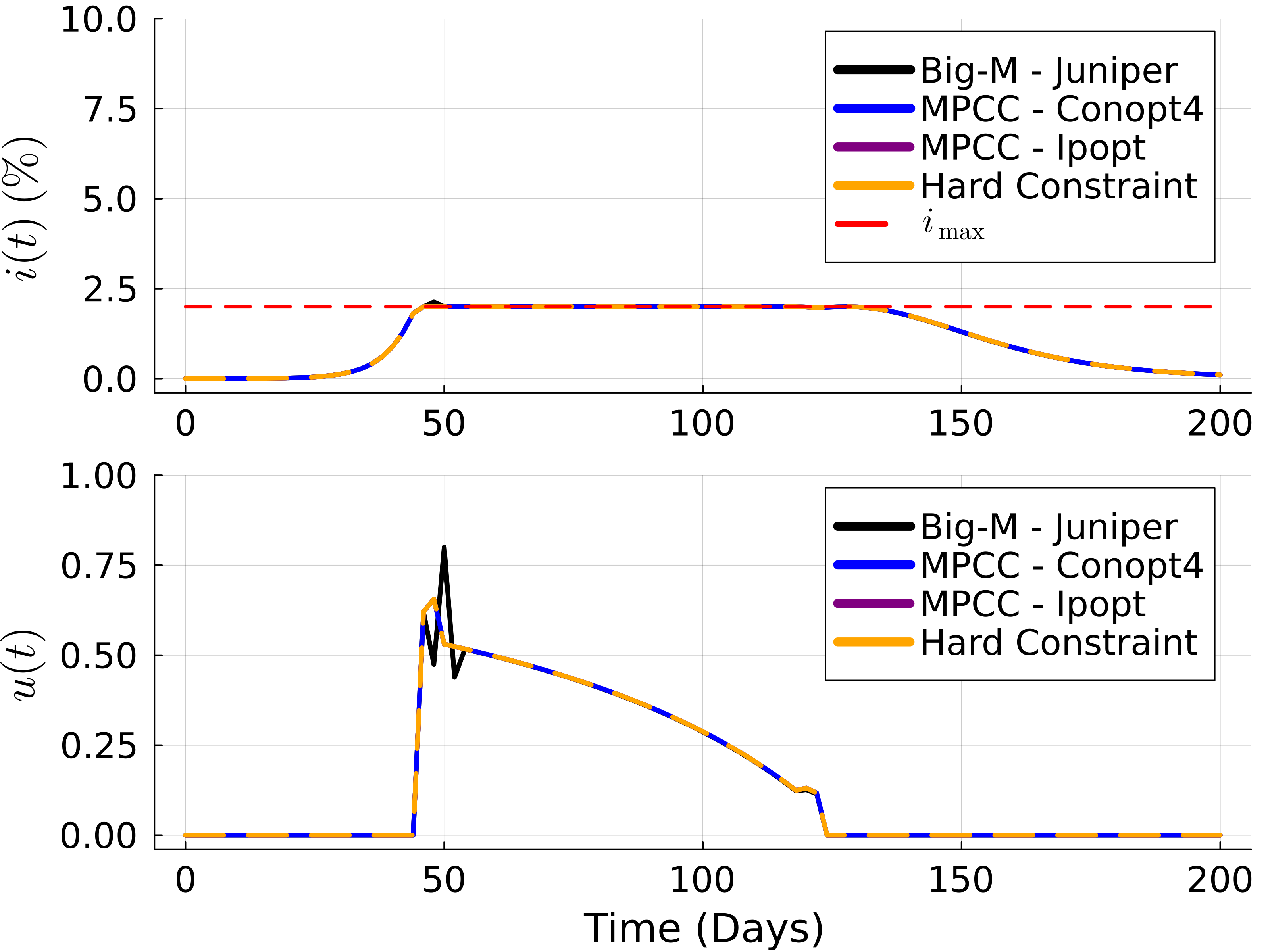}
         \caption{MPCC: $\alpha = 0.99$}
         \label{fig:pandemic_MPCC_0.99_app}
     \end{subfigure}
    \caption{Optimal $i(t)$ and $u(t)$ trajectories for a representative range of $\alpha$ values using MPCC approximation.}
    \label{fig:pandemic_MPCC_app}
\end{figure}

\FloatBarrier

Figure \ref{fig:Pandemic_SigVaR_iterations} compares the indicator approximation of the iterative SigVaR algorithm to the big-M solution. This plot shows qualitative evidence that as the modified sigmoidal parameters are iteratively tuned, the SigVaR approximation converges to the indicator of the big-M solution. Interestingly, as shown in Figure \ref{fig:pandemic_SigVaR}, perfectly approximating the indicator function of the MINLP formulation does not guarantee that the entire solution matches. While the methods may violate the constraint during the exact same time periods, this does not ensure that the local minimum aligns with the global minimum. This distinction is evident in the figure, where the profiles are similar but do not completely overlap.

\begin{figure}[!htb]
	\includegraphics[width=0.6\textwidth]{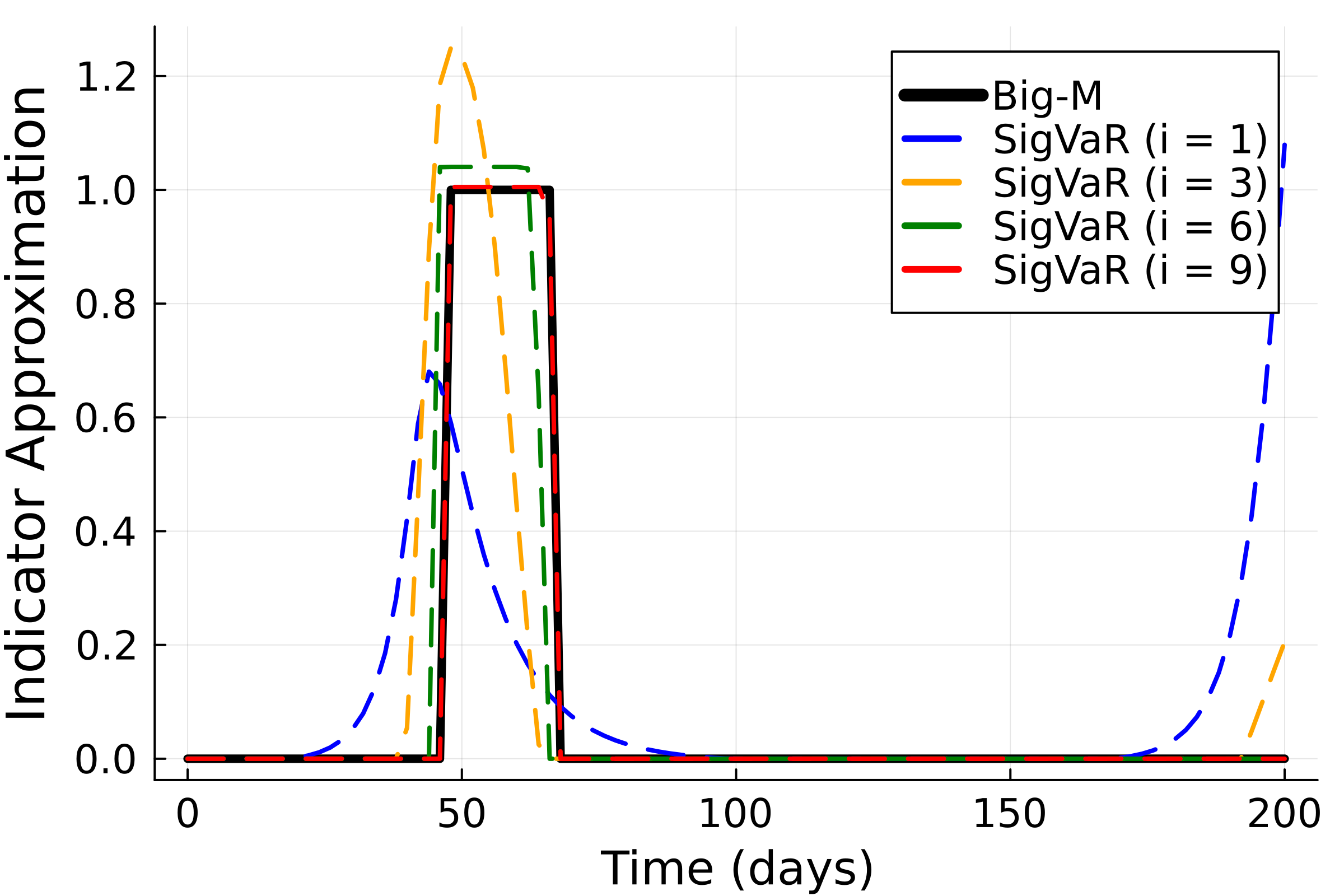}
	\centering
	\caption{The indicator function approximations obtained after each iteration of the SigVaR algorithm for $\alpha = 0.9$.}
	\label{fig:Pandemic_SigVaR_iterations}
\end{figure}

\FloatBarrier

\subsubsection{Integral Squared Error Analysis}\label{app:pandemic_error}

To quantitatively compare the results from the approximations, the integral squared error (ISE) of the indicator function and the $i(t)$ time profile are calculated as:
\begin{equation}
    \begin{aligned}
        &&&&& ISE_{\mathbbm{1}_{i(t)-i_{max} \leq 0}(t)} = \int_{t \in \mathcal{D}_t} (\mathbbm{1}_{i(t)-i_{max} \leq 0}(t)_{Approximation}-\mathbbm{1}_{i(t)-i_{max} \leq 0}(t)_{Big-M})^2 dt 
    \end{aligned}
    \label{eq:pandemic_ISE_indicator}
\end{equation}
\begin{equation}
    \begin{aligned}
        &&&&& ISE_{i(t)} = \int_{t \in \mathcal{D}_t} (i(t)_{Approximation}-i(t)_{Big-M})^2 dt 
    \end{aligned}
    \label{eq:pandemic_ISE_i(t)}
\end{equation}

\begin{table}[!htb]
    \centering
    \caption{$ISE_{i(t)}$ for each solution technique.}
    \begin{tabular}{cccccccc}
    \toprule
        $\alpha$ & \textbf{Hard Constraint} & \textbf{MPCC} & \textbf{CVaR} & \textbf{SigVaR} \\ \midrule
        $\mathbf{0.85}$ & $3.56$ & $3.48$ & $3.56$ & $0.076$ \\ \midrule
        $\mathbf{0.90}$ & $0.789$ & $0.765$ & $0.789$ & $0.0718$ \\ \midrule
        $\mathbf{0.95}$ & $0.035$ & $0.0364$ & $0.035$ & $0.0081$ \\ \midrule
        $\mathbf{0.96}$ & $0.0151$ & $0.0170$ & $0.0151$ & $0.00577$ \\ \midrule
        $\mathbf{0.97}$ & $0.0055$ & $0.0055$ & $0.0055$ & $0.00868$ \\ \midrule
        $\mathbf{0.99}$ & $0.00018$ & $0.00018$ & $0.00018$ & $0.00387$ \\ \bottomrule
    \end{tabular}
    \label{tab:pandemic_i(t)_ISE_app}
\end{table}

Table \ref{tab:pandemic_i(t)_ISE_app} offers additional quantitative validation of the observations in Figure \ref{fig:pandemic_CVaR} in Section \ref{sec:pandemic_case} by computing the $ISE_{i(t)}$ \eqref{eq:pandemic_ISE_i(t)} between the approximations and the big-M trajectories. The conservative nature of the CVaR solution is further substantiated, as the $ISE_{i(t)}$ values for the hard constraint and the CVaR solution remain nearly identical across all tested $\alpha$ values. 

\begin{table}[!htb]
    \centering
    \caption{ISE of the indicator approximation for each solution technique.}
    \begin{tabular}{ccccccc}
    \toprule
        $\alpha$ & \textbf{CVaR} & \textbf{SigVaR} \\ \midrule
        \textbf{0.85} & $0.16$ & $7e-5$ \\ \midrule
        \textbf{0.90} & $0.11$ & $0.0401$ \\ \midrule
        \textbf{0.95} & $0.06$ & $0.0200$  \\ \midrule
        \textbf{0.96} & $0.04$ & $3e-5$  \\ \midrule
        \textbf{0.97} & $0.03$ & $0.0715$  \\ \midrule
        \textbf{0.99} & $0.37$ & $0.0307$  \\ \bottomrule
    \end{tabular}
    \label{tab:pandemic_indicator(t)_ISE_app}
\end{table}

Similarly, Table \ref{tab:pandemic_indicator(t)_ISE_app} evaluates the error in approximating the indicator function relative to the big-M solution using Equation \eqref{eq:pandemic_ISE_indicator}. Since MPCC produces continuous variables that adhere to 0-1 values (to numerical tolerance), it eliminates the need to compute an approximation of the indicator function. The advantages of the SigVaR approximation are most pronounced for $\alpha$ values below 0.96. For higher $\alpha$ values, the ISE of SigVaR's indicator function converges to a magnitude similar to that of CVaR's indicator function, reflecting comparable performance under these conditions.


\begin{table}[!htb]
    \centering
    \caption{Performance of SigVaR on the disease control case study with $\alpha = 0.9$.}
    \begin{tabular}{ccccccc}
    \toprule
        \textbf{Iteration, $\mathbf{i}$} & $\mathbf{\beta}$ & $\mathbf{\gamma}$ & $\alpha_{\text{true}}$ & $\mathbf{ISE_{\mathbbm{1}_{i(t)-i_{max \leq 0}}(t)}}$ & $\mathbf{ISE_{i(t)}}$ & \textbf{Solve Time (s)} \\ \midrule
        CVaR & - & -      & $1.0$  & $0.11$ & $0.789$ & $0.9586$ \\ \midrule
        1 & $1.6$ & $63.8$    & $0.995$  & $0.455$ & $2.687$ & $6.67$ \\\midrule
        2 & $3.1$ & $102.5$   & $0.98$  & $0.145$ & $1.922$ & $1.84$ \\\midrule
        3 & $6.2$ & $180$     & $0.95$  & $0.0850$ & $1.275$ & $1.24$ \\\midrule
        4 & $12.4$ & $335.1$  & $0.93$  & $0.0600$ & $0.937$ & $1.96$ \\\midrule
        5 & $24.8$ & $645.1$  & $0.92$  & $0.0400$ & $0.614$ & $5.15$ \\\midrule
        6 & $49.6$ & $1265.2$ & $0.91$  & $0.0400$ & $0.299$ & $1.62$ \\\midrule
        7 & $99.2$ & $2505.5$ & $0.91$  & $0.0200$ & $0.216$ & $2.08$ \\\midrule
        8 & $198.4$ & $4986$  & $0.91$  & $0.0200$ & $0.184$ & $1.5$ \\\midrule
        9 & $396.9$ & $9946.9$ & $0.91$  & $0.0200$ & $0.0718$ & $4.86$ \\\bottomrule
    \end{tabular}
    \label{tab:pandemic_SigVaR_iteration_table}
\end{table}

Table \ref{tab:pandemic_SigVaR_iteration_table}, along with Figure \ref{fig:Pandemic_SigVaR_iterations}, presents the iterative performance of the SigVaR approximation over ten iterations for $\alpha = 0.90$. This table provides critical quantitative insights that reinforce the qualitative observations from Figures \ref{fig:pandemic_CVaR} and \ref{fig:pandemic_SigVaR} in Section \ref{sec:pandemic_case}. The progressive improvement in precision is quantitatively illustrated through the ISE of the indicator function, $\mathbbm{1}_{h(t) \leq 0}(t)$, and the infectious population, $i(t)$. The initial selections of parameters $\mu$ and $\tau$ establish a notably conservative starting point for the SigVaR solution. Notably, the ISE of $i(t)$ for the SigVaR solution does not surpass the accuracy of the CVaR solution until the 3\textsuperscript{rd} iteration, as shown in Figure \ref{fig:ISE_i_0.9_app}. Over nine iterations, the ISE of SigVaR's indicator approximation achieves a 99.06\% reduction, while the ISE of the infectious population improves by 94.88\% over nine iterations. Evidently, SigVaR demonstrates a robust capability to approximate constraint violations across the entire time domain, with iterative refinements evident after each iteration. 

While the SigVaR algorithm effectively approximates the indicator function of the event constraint, it simultaneously reconstructs the trajectory and behavior of $i(t)$, achieving a close approximation to the big-M $i(t)$ solution.

\begin{figure}[!htb]
	\includegraphics[width=0.6\textwidth]{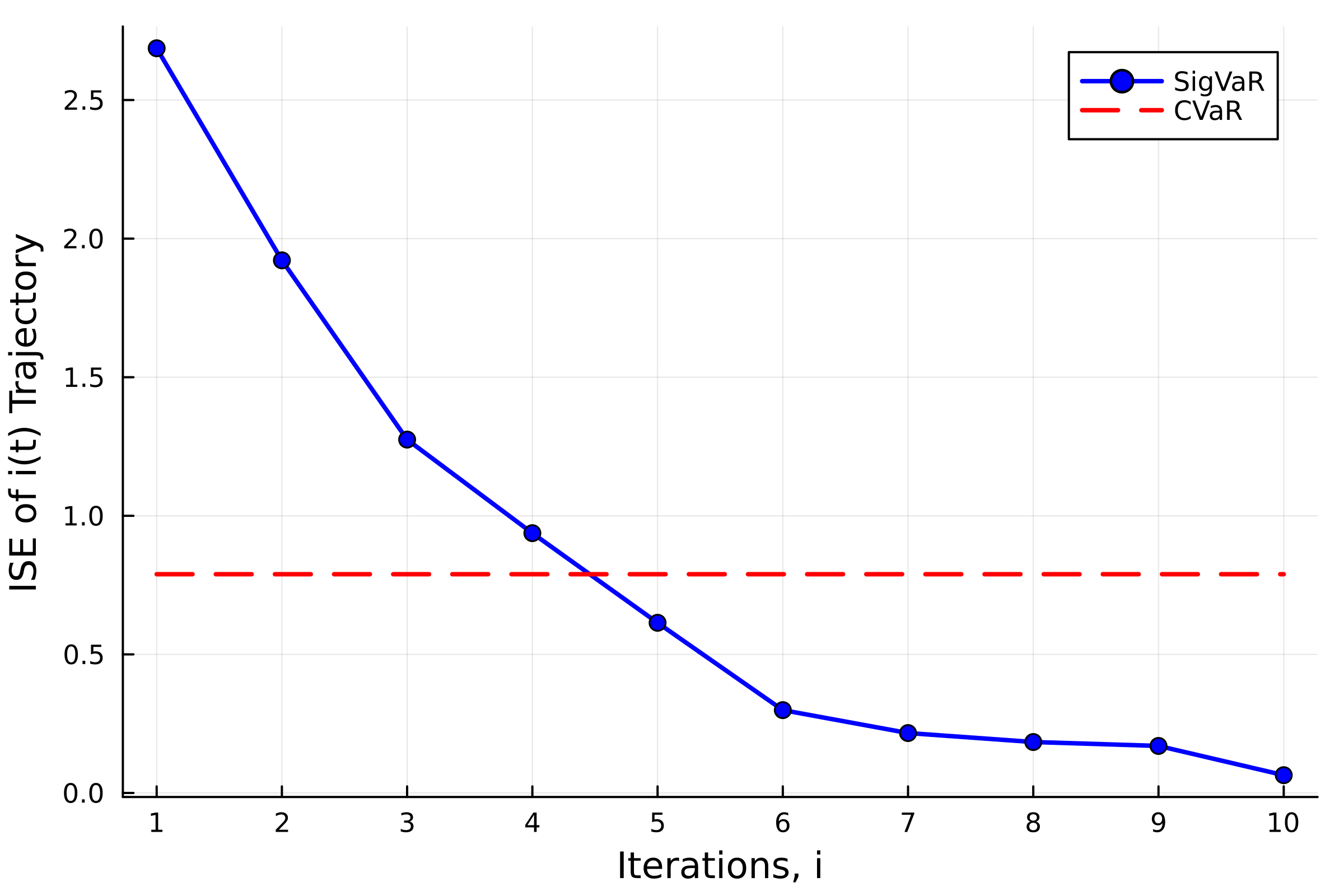}
	\centering
	\caption{Comparison of ISE of i(t) between CVaR and SigVaR, $\alpha = 0.9$}
	\label{fig:ISE_i_0.9_app}
\end{figure}

\FloatBarrier

\subsection{2D Temperature Control of Heated Plate Supplementary Material }\label{app:diffusion}

Several tables in this section support the overly conservative nature the CVaR solution method while also numerically reinforcing the potential of the SigVaR solution method. Firstly, Table \ref{tab:diffusion_violations_app_NL} highlights the number of violations for each method in all $\alpha$ values tested. Evidently, the SigVaR approximation yields the only non-zero number of violations, which align relatively well to the big-M's violation count. Using Equations \eqref{eq:diffusion_ISE_i(t)} and \eqref{eq:diffusion_ISE_indicator}, the numerical values summarized in Tables \ref{tab:diffusion_NL_indicator(x)_ISE_app} and \ref{tab:diffusion_NL_T(x)_ISE_app} further strengthen the claims made in Section \ref{sec:pde_case} about each solution method. Similar to Table \ref{tab:pandemic_SigVaR_iteration_table}, Table \ref{tab:diffusion_NL_SigVaR_iteration_table} shows various metrics for each iteration of the SigVaR solution computed for $\alpha = 0.95$. This table concretely quantifies how the SigVaR solution converges closer to the big-M solution as the sigmoidal parameters, $\beta$ and $\gamma$, increase iteratively.

\begin{equation}
    \begin{aligned}
        &&&&& ISE_{\mathbbm{1}_{T(x)-T_{max} \leq 0}(x)} = \int_{x \in \mathcal{D}_x} (\mathbbm{1}_{T(x)-T_{max} \leq 0}(x)_{Approximation}-\mathbbm{1}_{T(x)-T_{max} \leq 0}(x)_{Big-M})^2 dx 
    \end{aligned}
    \label{eq:diffusion_ISE_indicator}
\end{equation}
\begin{equation}
    \begin{aligned}
        &&&&& ISE_{T(x)} = \int_{x \in \mathcal{D}_x} (T(x)_{Approximation}-T(x)_{Big-M})^2 dx
    \end{aligned}
    \label{eq:diffusion_ISE_i(t)}
\end{equation}

\begin{table}[!htb]
    \centering
    \caption{True $\alpha$ for each solution technique.}
    \begin{tabular}{cccccccc}
    \toprule
        $\alpha$ & \textbf{Hard Constraint} & \textbf{CVaR} & \textbf{MPCC} & \textbf{SigVaR} & \textbf{BigM} \\ \midrule
        $\mathbf{0.90}$ & $1.0$ & $1.0$ & $1.0$ & $0.91$ & $0.90$ \\ \midrule
        $\mathbf{0.95}$ & $1.0$ & $1.0$ & $1.0$ & $0.96$ & $0.95$ \\ \midrule
        $\mathbf{0.96}$ & $1.0$ & $1.0$ & $1.0$ & $0.97$ & $0.96$ \\ \midrule
        $\mathbf{0.97}$ & $1.0$ & $1.0$ & $1.0$ & $1.0$ & $0.97$ \\ \midrule
        $\mathbf{0.99}$ & $1.0$ & $1.0$ & $1.0$  & $1.0$ & $0.99$ \\ \bottomrule
        $\mathbf{0.999}$ & $1.0$ & $1.0$ & $1.0$  & $1.0$ & $0.99$ \\ \bottomrule
    \end{tabular}
    \label{tab:diffusion_real_alphas_app}
\end{table}

\begin{table}[h]
    \centering
    \caption{$ISE_{T(x)}$ for each solution technique.}
    \begin{tabular}{ccccccc}
    \toprule
        $\alpha$ & \textbf{Hard Constraint} & \textbf{MPCC} & \textbf{CVaR} & \textbf{SigVaR} \\ \midrule
        \textbf{0.90} & $0.0385$ & $0.0385$ & $0.0385$ & 
        $0.0115$\\ \midrule
        \textbf{0.95} & $0.0326$ & $0.0326$ & $0.0326$ & $0.00265$ \\ \midrule
        \textbf{0.96} & $0.0253$ & $0.0253$ & $0.0253$ & $0.00249$ \\ \midrule
        \textbf{0.97} & $0.0232$ & $0.0232$ & $0.0232$ & $0.00173$ \\ \midrule
        \textbf{0.99} & $0.0165$ & $0.0165$ & $0.0165$ & $0.00172$ \\ \midrule
        \textbf{0.999} & $0.00129$ & $0.00129$ & $0.00129$ & $0.00743$ \\ \bottomrule
    \end{tabular}
    \label{tab:diffusion_NL_T(x)_ISE_app}
\end{table}

\begin{table}[h]
    \centering
    \caption{ISE of the indicator approximation for each solution technique.}
    \begin{tabular}{ccc}
    \toprule
        $\alpha$ & \textbf{CVaR} & \textbf{SigVaR} \\ \midrule
        \textbf{0.90} & $0.298$ & $0.334$ \\ \midrule
        \textbf{0.95} & $0.142$ & $0.0709$ \\ \midrule
        \textbf{0.96} & $0.0957$ & $0.0763$ \\ \midrule
        \textbf{0.97} & $0.0752$ & $0.0355$ \\ \midrule
        \textbf{0.99} & $0.0387$ & $0.0150$ \\ \midrule
        \textbf{0.999} & $0.00322$ & $0.00645$ \\ \bottomrule
    \end{tabular}
    \label{tab:diffusion_NL_indicator(x)_ISE_app}
\end{table}

\begin{table}[h]
    \centering
    \caption{Performance of the SigVaR approximation for the 2D diffusion case study with $\alpha = 0.95$.}
    \begin{tabular}{cccccc}
    \toprule
        \textbf{Iteration, $\mathbf{i}$} & $\mathbf{\beta}$ & $\mathbf{\gamma}$ & $\mathbf{ISE_{\mathbbm{1}_{T(x)-T_{max \leq 0}}(x)}}$ & $\mathbf{ISE_{T(x)}}$ & \textbf{Solve Time (s)} \\ \midrule
        CVaR & - & - & $0.142$ & $0.0253$ & $2.02$ \\\midrule 
        1 & $15.5$ & $7.5$ & $0.1333$ & $0.168$ & $3.8$  \\\midrule 
        2 & $18.6$ & $8.9$ & $0.1247$ & $0.1249$ & $5.07$  \\\midrule 
        3 & $22.3$ & $10.6$ & $0.1247$ & $0.0918$ & $3.96$ \\\midrule 
        4 & $26.8$ & $12.6$ & $0.1204$ & $0.0669$ & $4.62$ \\\midrule 
        5 & $32.1$ & $15.1$ & $0.1204$ & $0.0496$ & $5.01$  \\\midrule 
        6 & $38.6$ & $18.0$ & $0.1204$ & $0.0352$ & $6.67$  \\\midrule 
        7 & $46.3$ & $21.5$ & $0.1053$ & $0.0255$ & $6.97$  \\\midrule 
        8 & $55.5$ & $25.7$ & $0.1032$ & $0.0183$ & $6.89$ \\\midrule 
        9 & $66.7$ & $30.8$ & $0.0989$ & $0.0131$ & $6.19$  \\\midrule 
        10 & $80.0$ & $36.8$ & $0.0881$ & $0.0093$ & $9.39$  \\\midrule 
        11 & $96.0$ & $44.1$ & $0.0838$ & $0.0066$ & $9.91$  \\\midrule 
        12 & $115.2$ & $52.8$ & $0.0838$ & $0.0046$ & $10.11$  \\\midrule 
        13 & $138.2$ & $63.3$ & $0.0795$ & $0.0033$ & $11.53$  \\\midrule 
        14 & $165.9$ & $75.8$ & $0.0709$ & $0.0026$ & $14.89$ \\ \bottomrule
    \end{tabular}
    \label{tab:diffusion_NL_SigVaR_iteration_table}
\end{table}


\begin{table}
    \centering
    \caption{Number of violations for each solution technique }
\label{tab:diffusion_violations_app_NL}
\begin{tabular}{cccccc}
\toprule
    $\alpha$ & \textbf{Hard Constraint} & \textbf{Big-M} & \textbf{MPCC} & \textbf{CVaR} & \textbf{SigVaR}  \\ \midrule
    $\mathbf{0.90}$    & 0 & 277 & 0 & 0 & 144 \\ \midrule
    $\mathbf{0.95}$    & 0 & 132 & 0 & 0 & 84  \\ \midrule
    $\mathbf{0.96}$    & 0 & 89 & 0 & 0 & 76  \\ \midrule
    $\mathbf{0.97}$    & 0 & 70 & 0 & 0 & 59  \\ \midrule
    $\mathbf{0.99}$    & 0 & 48 & 0 & 0 & 28   \\  \midrule
    $\mathbf{0.999}$   & 0 & 3 & 0 & 0 & 3 \\ \bottomrule
\end{tabular}
\end{table}
\end{document}